\documentclass[10pt,reqno]{amsart}

\usepackage[sort]{cite}
\usepackage{amsfonts} 
\usepackage{amsmath} 
\usepackage{amssymb} 
\usepackage{amsthm} 
\usepackage{latexsym} 
\usepackage{mathrsfs} 
\usepackage{mathtools} 
\usepackage{slashed}  

\usepackage{verbatim}
\usepackage{cases}
\usepackage[dvips]{epsfig}
\usepackage{epsf} 
\usepackage[bookmarksnumbered,pdfpagelabels=true,plainpages=false,colorlinks=true,
            linkcolor=black,citecolor=black,urlcolor=black]{hyperref}
\usepackage{url}
\usepackage{color}
\usepackage{xcolor}
\usepackage{graphicx}
\usepackage{enumitem} 

\usepackage{pifont} 
\usepackage{upgreek} 
\usepackage{fancyhdr} 
\usepackage{calligra} 
\usepackage{marvosym} 
\usepackage[percent]{overpic} 
\usepackage{pict2e} 
\usepackage[hypcap=false]{caption} 
\usepackage{wasysym} 
\usepackage{accents}
\usepackage[margin=3cm]{geometry}
\usepackage{lipsum}
\theoremstyle{plain}
\newtheorem{theorem}{Theorem}[section]

\newtheorem{lemma}[theorem]{Lemma}
\newtheorem{proposition}[theorem]{Proposition}

\theoremstyle{definition}
\newtheorem{remark}[theorem]{Remark}

\renewcommand{\tilde}{\widetilde}
\renewcommand{\bar}{\overline}
\numberwithin{equation}{section}
\allowdisplaybreaks

\newcommand{\tr}{\operatorname{tr}}

\newcommand{\cA}{\mathcal{A}}
\newcommand{\cB}{\mathcal{B}}
\newcommand{\cC}{\mathcal{C}}

\newcommand{\cE}{\mathcal{E}}
\newcommand{\cF}{\mathcal{F}}

\newcommand{\cL}{\mathcal{L}}
\newcommand{\cM}{\mathcal{M}}

\newcommand{\cP}{\mathcal{P}}
\newcommand{\cR}{\mathcal{R}}
\newcommand{\cS}{\mathcal{S}}
\newcommand{\cT}{\mathcal{T}}
\newcommand{\cV}{\mathcal{V}}
\newcommand{\cU}{\mathcal{U}}

\newcommand{\cSM}{\mathcal{SM}}

\newcommand{\bB}{\mathbb{B}}
\newcommand{\bC}{\mathbb{C}}
\newcommand{\bD}{\mathbb{D}}
\newcommand{\bE}{\mathbb{E}}
\newcommand{\bF}{\mathbb{F}}
\newcommand{\bG}{\mathbb{G}}

\newcommand{\bM}{\mathbb{M}}
\newcommand{\bR}{\mathbb{R}}
\newcommand{\bS}{\mathbb{S}}

\newcommand{\bX}{\mathbb{X}}

\newcommand{\fF}{\mathfrak{F}}
\newcommand{\fG}{\mathfrak{G}}

\newcommand{\sA}{\mathsf{A}}
\newcommand{\sB}{\mathsf{B}}
\newcommand{\sE}{\mathsf{E}}
\newcommand{\sF}{\mathsf{F}}
\newcommand{\sM}{\mathsf{M}}
\newcommand{\sT}{\mathsf{T}}
\newcommand{\sL}{\mathsf{L}}
\newcommand{\sN}{\mathsf{N}}
\newcommand{\sH}{\mathsf{H}}

\newcommand{\sP}{\mathsf{P}}

\newcommand{\sx}{\mathsf{x}}
\newcommand{\sy}{\mathsf{y}}

\newcommand{\pa}{\partial}

\newcommand{\wK}{\widehat{K}}

\DeclareMathOperator*{\dist}{dist}

\newcommand{\zbE}{\prescript{}{0}{\mathbb{E}}}
\newcommand{\zbF}{\prescript{}{0}{\mathbb{F}}}

\newcommand{\zW}{\prescript{}{0}{W}}

\begin{document}
\title[thermalmagetoviscoelasti]{On a thermodynamically consistent model for magnetoviscoelastic fluids in 3D}

\author{Hengrong Du}
\address{Department of Mathematics\\
        Vanderbilt University\\
        Nashville, Tennessee\\
        USA}
\email{hengrong.du@vanderbilt.edu}

\author{Yuanzhen Shao}
\address{The University of Alabama\\ 
	Tuscaloosa, Alabama \\
	USA}
\email{yshao8@ua.edu}
\author{Gieri Simonett}
\address{Department of Mathematics\\
        Vanderbilt University\\
        Nashville, Tennessee\\
        USA}
\email{gieri.simonett@vanderbilt.edu}

\thanks{This work was supported by a grant from the Simons Foundation (\#426729 and \#853237, Gieri Simonett).}




\begin{abstract} 
	We introduce a system of equations
	that models a non-isothermal  magnetoviscoelastic fluid. We show that the model is thermodynamically consistent, 
	and that the critical points of the entropy functional with prescribed energy correspond exactly with the equilibria of the system.
	The system is investigated in the framework of quasilinear parabolic systems and shown to be
	 locally well-posed in an $L_p$-setting. Furthermore, we prove that 
	constant equilibria are normally stable. In particular, we show that solutions
	that start close to a constant equilibrium exist globally and converge exponentially fast to a (possibly different)
	constant equilibrium. Finally, we establish that the negative entropy serves as a strict  Lyapunov functional 
	and we then show that every solution that is eventually bounded in the topology of the natural state
       space exists globally and converges to the set of equilibria. 
\end{abstract}

\maketitle


\section{Introduction}\label{S:Intro}
We study the following system of equations that models the evolution of a
magnetoviscoelastic fluid, allowing for a non-constant temperature in a $C^3$-bounded domain $\Omega\subset\bR^3$ with outward unit normal~$\nu$:
\begin{equation}
\label{magneto sys}
\left\{\begin{aligned}
\partial_t u + u \cdot \nabla u -\nabla\cdot(\upmu(\theta) \nabla u)  +\nabla \pi &=-\nabla \cdot(\nabla m \odot \nabla m) + \nabla \cdot (F F^{\sT}) &&\text{in}&&\Omega ,\\
\nabla \cdot u &=0 &&\text{in}&&\Omega ,\\
u&=0 &&\text{on}&&\partial \Omega  , \\
\partial_t F + u \cdot\nabla F - \nabla\cdot(\kappa(\theta) \nabla F) &= (\nabla u)^{\sT} F   &&\text{in}&&\Omega,\\
F &=0   &&\text{on}&&\partial\Omega, \\
\partial_t\theta+u\cdot\nabla\theta +  \nabla\cdot q &= \upmu(\theta)|\nabla u|^2+\kappa(\theta)|\nabla F|^2 
+\alpha(\theta)|\Delta m+|\nabla m|^2m|^2&& \text{in}&&\Omega,
\\
q\cdot\nu&=0&& \text{on}&&\partial \Omega,
\\
\partial_t m + u \cdot\nabla m &= - \alpha(\theta) m\times (m \times \Delta m) - \beta(\theta) m \times \Delta m  &&\text{in}&&\Omega,\\
\partial_\nu m  &=0   &&\text{on}&&\partial\Omega, \\
|m|  &=1   &&\text{in}&& \Omega, \\
(u(0), F(0), \theta(0), m(0))& =(u_0, F_0, \theta_0, m_0) &&\text{in}&& \Omega. \\
\end{aligned}\right.
\end{equation}
This is a coupled system consisting of 
\begin{enumerate}
  \item the incompressible Navier--Stokes equations with variable viscosity coefficient $\upmu(\theta)$ for the velocity field
  $$u:(0, T)\times \Omega \to \bR^3$$ 
  with a right hand side that includes  the elastic stress tensor  induced by the magnetization
  field $m$ and the deformation tensor  $F$.
  Moreover, $\pi:(0, T)\times \Omega\to \bR$ denotes
  the pressure function;
 \vspace{1mm} 
\item a transport-dissipative system for the deformation tensor 
$$F:(0, T)\times \Omega \to \bM^3:=\bR^{3\times 3}$$ 
with variable dissipative coefficient $\kappa(\theta)$ and stretching term $(\nabla u)^\top F$; 

\vspace{1mm}
\item a transported anisotropic heat equation for the (absolute) temperature function
$$\theta: (0, T)\times \Omega \to\bR$$ 
with the heat flux $q$ given by the generalized Fourier law \cite{DhSh80, SoVi12}
\begin{equation}
\label{eqn:heatflux}
	q=q(u, F, \theta, m) =-K(u, F, \theta, m) \nabla \theta,
\end{equation}
where $K(u, F, \theta, m)$ is a  positive-definite, matrix-valued function of
$(u, F, \theta, m)$ which reflects the inhomogeneity of the material.   For example, one may choose 
$$q=-h(\theta)\nabla \theta-k(\theta)(\nabla \theta \cdot m)m,$$
 where $h(\theta)$ describes the variable heat conductivity, while $k(\theta)$ represents the inhomogeneous thermal conductivity
  along the direction preferred by the magnetization within the medium.
 In this case, $K(u,F,\theta ,m)=h(\theta)I_3+k(\theta) (m\otimes m)$; 
\vspace{1mm}
\item a convected Landau--Lifshitz--Gilbert system for the magnetization field
$$m:(0, T)\times \Omega \to \mathbb{S}^2=\{d\in \bR^3: |d|=1\}$$ 
with variable
Gilbert damping parameter $\alpha(\theta)$ and exchange parameter
$\beta(\theta)$. 
\vspace{1mm}
  \end{enumerate}
  
The model can be used to describe so-called smart fluids (magnetorheological fluids), that is, fluids
carrying magnetoelastic particles. 
Because of their remarkable properties,  magnetoelastic materials
are widely used in  technical applications. 

\medskip
 When fluids are subjected to heat, their molecules experience internal movement due to changes in temperature. This compound effect can be described by differential equations that govern the laws of these changes. 
 While numerous publications have been devoted to the dynamics of magnetoviscoelastic fluids 
 in the isothermal case (see, for example, \cite{BFLS18, KorSch22, DSS23,For16, GKMS21, KKS21, SchZa18,Zhao18}), there is a lack of research on the thermodynamic effects associated with these fluids.
 
A related class of thermodynamically consistent models for incompressible non-isothermal {\em nematic liquid crystal flows}  has been developed using the Ericksen--Leslie formalism, as discussed in \cite{DeALiu19, FFRS12,HiPr16,HiPr17,HiPr18}. 
Specifically,  De Anna  and Liu   extended the general Ericksen--Leslie system 
and the general Oseen--Frank energy density in \cite{DeALiu19}, and they obtain a global well-posedness result  of strong solutions for initial data that are close to equilibrium in suitable homogeneous Besov spaces.
Another approach to modeling non-isothermal nematic liquid crystals is presented in \cite{FFRS12}, where
 the authors introduce an energetically closed system and derive the equations using a generalized variational principle. 
 They show the existence of global weak solutions for suitable initial data in a three-dimensional bounded domain with a sufficiently regular boundary. 
 Meanwhile, Hieber and Pr\"{u}ss analyzed the non-isothermal Ericksen--Leslie system by means of maximal $L_p$-regularity techniques for 
 quasilinear parabolic evolution equations in \cite{HiPr16, HiPr18}. 
 They demonstrate the local existence of classical solutions and the stability of solutions subject to initial data that are close 
 enough to an equilibrium   for the case of linear boundary conditions.  It is also worth noting that  the authors in~\cite{FRSZ14, FRSZ15} consider
non-isothermal nematics using the $Q$-tensor as the order parameter. They prove global existence of weak solutions with the Landau--De Gennes polynomial potential and the Ball--Majumdar singular potential, which are commonly used to describe the configuration of the liquid crystal molecules.

\medskip\noindent
A detailed analysis reveals that the system~\eqref{magneto sys} features a quasilinear parabolic  structure,
and we are employing the theory of  maximal regularity, see for instance~\cite{PrSi16}, 
to study existence, uniqueness, and qualitative properties of solutions.
With this approach in mind, there are several difficulties that arise in the mathematical analysis. For instance,
\begin{itemize}
\item
The equation \eqref{magneto sys}$_6$ for the temperature $\theta$ contains terms that have a highly nonlinear dependence on the magnetic field $m$;  
namely, the equation contains a term
that is quadratic in second order derivatives of $m$.  On a technical level, this means that we need to work in function spaces that encode higher regularity
for the magnetic field $m$.
\item The flux vector $q$ in \eqref{eqn:heatflux} is allowed to depend in a nonlinear way on the variables $(u, F, \theta, m)$. 
The boundary condition $q\cdot \nu=0$ then leads to a nonlinear boundary condition    which is to be satisfied by the solutions. 
This implies that solutions `live on a nonlinear manifold,' and this adds significant challenges to a mathematical treatment.
\end{itemize}

Quasilinear parabolic systems with nonlinear boundary conditions have been studied in the literature by several authors,
for instance in \cite{Lun95, JLS14, LPS06, LPS08, MeyriesThesis, Mey12}.
However, the results contained in these publications cannot be applied directly to the model~\eqref{magneto sys}.
For instance, while these works cover a very general class of quasilinear parabolic systems with nonlinear boundary conditions,
they do not include a coupling to the Navier--Stokes equations.  
Moreover, these works do not feature equations which contain quadratic terms in the highest derivatives of some of the variables.

\medskip
Due to the structure of the nonlinear terms in the temperature equation, it seems infeasible to find a realization of \eqref{magneto sys} 
in so-called extrapolation spaces,
which would be helpful in order to absorb nonlinear boundary conditions of co-normal type, see for instance~\cite{Ama90b, Sim95}.

\medskip
Finally, we would like to mention that the techniques developed in this paper can be generalized to also 
cover fully nonlinear boundary conditions, and we could also admit nonlinear boundary conditions for the remaining variables.

\medskip
The manuscript is structured as follows. In Section \ref{secition:thermo}, we show that the system  \eqref{magneto sys}
is thermodynamically consistent. In addition, we provide a characterization of the equilibria and
we show that the critical points of the constrained entropy functional correspond exactly to these equilibria.
In Section~\ref{section:functional}, we introduce  a functional analytic setting to study the system \eqref{magneto sys}.
In Section~\ref{section:linear}, we provide existence and uniqueness results for some related linear problems, 
which will then form the basis to establish the local well-posedness of strong
solutions for system \eqref{magneto sys}, carried out in Section~\ref{section:localwell}.
In the main theorem of this section, Theorem~\ref{thm:MainLocal}, we show that the system  \eqref{magneto sys} generates 
a Lipschitz continuous semiflow on the state manifold (defined by the nonlinear boundary condition).
Here we have been inspired by the approach in \cite{LPS06, MeyriesThesis, Mey12}.
In addition, we show that the temperature satisfies a maximum principle.
In section~\ref{section:stability}, we provide criteria for global existence.
In addition, we study stability of constant equilibria;
in particular, we show that solutions that start close to a constant equilibrium exist globally and converge exponentially fast
to a (possibly different) constant equilibrium.
Finally, in Appendix A, we establish some relevant properties of fractional Sobolev spaces with temporal weights,
and in Appendix B, we study mapping properties of the nonlinearities associated to system \eqref{magneto sys}.

\bigskip\noindent
\textbf{Notation:} 
For the readers' convenience, we list here some notations and conventions used throughout the manuscript.

In the following, all vectors $a=(a_1,\cdots, a_n)\in \bR^n$ are viewed as column vectors. 
For two vectors $a,b\in\bR^n$, the Euclidean inner product is denoted by $a\cdot b$.
Given two matrices $A,B\in \bM^n$, the Frobenius matrix inner product $A:B$ is given by
$$
A:B={\rm Tr} (AB^{\sT}),
$$
where ${}^{\sT}$ is the transpose.
Suppose $\Omega$ is an open subset of $\bR^n$.
If $u\in C^1(\Omega;\bR^n)$, we set $\nabla u(x)= e_j\otimes \partial_ju(x)$
for $x\in\Omega$.
Hence, for $u=(u_1,\cdots , u_n)\in C^1(\Omega; \bR^n)$, we have 
$$[\nabla u(x)]_{ij}= \partial_i u_j(x), \;\; 1\le i, j\le n, \;\; x\in \Omega.$$
We note that $[\nabla u(x)]^{\sT}$ corresponds to the Fr\'echet derivative of $u$ at $x\in\Omega$.

If $A\in C^1(\Omega;\bM^n)$, its divergence $\nabla \cdot A$ is the vector function defined by
\begin{equation}
\label{divergence-matrix}
(\nabla \cdot A)(x)=(\partial_j A(x))^{\sT}e_j, \;\; x\in \Omega.
\end{equation}
Hence, if $A=[a_{ij}]\in C^1(\Omega;\bM^n)$, its divergence is given by
 $$[(\nabla\cdot A)(x)]_i=\partial_j a_{ji}(x), \;\; i=1,\cdots , n,\;\; x\in \Omega. $$
Here and in the sequel, we use the summation convention, indicating that terms with repeated indices are added.
We note   that 
 \eqref{divergence-matrix} implies
\begin{equation}
\label{divergence-property}
(\nabla\cdot A)\cdot u = \nabla\cdot (Au)- A:\nabla u, \quad A\in C^1(\Omega;\bM^n), \ u\in C^1(\Omega;\bR^n).
\end{equation}
For a matrix $A\in C^1(\Omega;\bM^n)$, we set 
$|\nabla A|^2=\partial_j A: \partial_jA.$

\smallskip
For functions 
$f,g\in L_2(\Omega; \bR^m)$, 
$$(f|g)_\Omega =\int_\Omega f\cdot g\, dx$$ denotes the $L_2$-inner product.
For any  Banach space $X$, $s\ge 0$, $p \in (1,\infty)$, 
$W^s_p(\Omega;X)$ denote the  $X-$valued  Sobolev(-Slobodeckij)   spaces.
When the choice of $X$ is clear from the context, we will just write $W^s_p(\Omega)$. 

\goodbreak
\medskip
Given any $T\in (0,\infty]$, we will denote the interval $(0,T)$ by 
$J_T$.
For   $p\in (1,\infty)$ and $\mu\in (0,1]$, the $X$-valued $L_p$-spaces with temporal weight are defined by
$$
L_{p,\mu}(J_T;X):=\left\{ f: (0,T)\to X: \, t^{1-\mu}f\in L_p(J_T;X)  \right\}.
$$
Similarly,
$$
W^k_{p,\mu}(J_T;X):=\left\{ f \in    W^k_{1,loc}(J_T;X):\,  \pa_t^j f\in L_{p,\mu}(J_T;X), \, j=0,1,\ldots,k \right\}.
$$
For $s\in (0,1)$, the Sobolev-Slobodeckij spaces   with temporal  weights  are defined as
$$
W^s_{p,\mu}(J_T;X):= \{ u \in L_{p,\mu}(J_T;X): \, 
\|u\|_{W^s_{p,\mu}(J_T;X)} =  \|u\|_{L_{p,\mu}(J_T;X)}  + [u]_{W^s_{p,\mu}(J_T;X)}<\infty \} ,
$$
where  
\begin{equation}
\label{seminorm}
[u]_{W^s_{p,\mu}(J_T;X)} := \left(\int_0^T  \int_0^t \tau^{p(1-\mu)} \frac{\| u(t)- u(\tau)\|_X^p}{(t-\tau)^{s p+1}}\, d\tau  d t \right)^{1/p}.
\end{equation}
See \cite[Formula~(2.6)]{MeSc12}.
$\|\cdot\|_{W^s_{p,\mu}(J_T;X)}$ is termed the intrinsic norm of  $W^s_{p,\mu}(J_T;X)$.

For any two Banach spaces $X$ and $Y$, the notation $\cL(X,Y)$ stands for the set of all bounded linear operators from $X$ to $Y$ and $\cL(X):=\cL(X,X)$. 
$\cL is(X,Y)$ denotes the subset of $\cL(X,Y)$ consisting of linear isomorphisms from $X$ to $Y$.

Finally, in this article,     $\Phi: \bR_+\to \bR_+$ always denotes a continuous non-decreasing function satisfying
$$
\Phi(r)\to 0^+ \quad \text{as } r\to 0^+.
$$

\section{Thermodynamic consistency}\label{secition:thermo}
   In this section  we discuss  the thermodynamic properties of
   \eqref{magneto sys}. We introduce the following assumptions:
   \begin{equation}
   \label{assumption}
     \begin{aligned}
   &\upmu, \kappa, \alpha, \beta \in  C^5 (\bR),\quad  K\in C^5 (\bR^{3} \times \bM^3 \times \bR\times \bR^{3};{\rm sym}\,(\bM^3));   \\
    & \upmu\ge \underline{\upmu},\quad \kappa\ge  \underline{\kappa}, \quad \alpha\ge \underline{\alpha},\quad K\ge \underline{c}I_3, 
  \end{aligned}
  \end{equation}
  where   $\underline{\upmu}, \underline{\kappa}, \underline{\alpha}$ and $\underline{c}$ are given positive constants. 
Here we assume $C^5$-smoothness for convenience.

      We assume that the Helmholtz free energy density $\psi$ is given by 
      \begin{equation*}
        \psi= \psi(F,\theta, m)= \frac{1}{2}|F|^2 + \frac{1}{2}|\nabla m|^2-\theta\ln\theta. 
      \end{equation*}
      Then the entropy density $\eta$ and the internal energy density $e_{\rm int}$ can be obtained via the following thermodynamic relations:
      \begin{equation*}
        \begin{array}[]{ll}
          \eta=-\partial_\theta \psi=1+\ln \theta & \text{(the Maxwell relation)}
          \\
          e_{\rm int}=\psi+\theta \eta =\frac{1}{2}|F|^2 + \frac{1}{2}|\nabla m|^2+\theta & \text{(the Legendre transform of $\psi$ w.r.t. $\eta$).}\\
        \end{array}
      \end{equation*}
      We can derive from the $\theta$ equation in \eqref{magneto sys} the entropy evolution 
      \begin{equation}
        \partial_t \eta+u\cdot \nabla \eta + \nabla\cdot g=r,\label{eqn:entropy}
      \end{equation}
      where $g$ denotes the entropy flux which satisfies the Clausius--Duhem relation
      \begin{equation*}
        g=\frac{q}{\theta},
      \end{equation*}
and where the entropy production rate $r$ is given as 
\begin{equation}
\label{entropy-production}
  r=\frac{1}{\theta}\left[ \upmu(\theta)|\nabla u|^2+\kappa(\theta)|\nabla F|^2+\alpha(\theta)|\Delta m+|\nabla m|^2 m|^2-\frac{q\cdot \nabla \theta }{\theta}\right].     
\end{equation}
The \emph{thermodynamic consistency} of \eqref{magneto sys} is given by the following proposition.
\begin{proposition}\label{Prop: thermo consistent}
Suppose $(u, F, \theta, m)$ is a solution of \eqref{magneto sys} with the regularity properties asserted in Theorem~\ref{thm:MainLocal}.
  Then the following properties hold.
  \begin{enumerate}
    \item[{\rm (a)}] {\em (}First law of thermodynamics{\em )}. The total energy 
    $$\sE=\sE(u,F,\theta,m)=\int_{\Omega}(\frac{1}{2}|u|^2+e_{\rm int})\, dx=\int_\Omega (\frac{1}{2} |u|^2 + \frac{1}{2}|F|^2 + \frac{1}{2} |\nabla m|^2 +\theta)\,dx$$
    is preserved along the solution $(u, F, \theta, m)$. 
\vspace{1mm}
 \item[{\rm (b)}] {\em (}Second law of thermodynamics{\em )}. The total entropy 
 $$\sN=\sN(\theta) =\int_\Omega \eta\, dx=\int_\Omega (1+ \ln\theta)\,dx $$ 
 is non-decreasing   along the solution $(u, F, \theta, m)$.  \\
 In fact, the entropy production rate $r$ is always non-negative, i.e., $r\ge 0$. 
\end{enumerate}   
\end{proposition}
\begin{proof}
Let $(u, F, \theta, m)$ be a solution of  \eqref{magneto sys} with initial value $z_0=(u_0,F_0,\theta_0,m_0)$ defined on its maximal interval of existence
$[0, T_+(z_0))$, see  Theorem~\ref{thm:MainLocal} for the precise regularity assertions. Let $T\in (0, T_+(z_0))$ be fixed.
For notational simplicity, we suppress the time variable in the following computations.

\medskip\noindent
  For ({\rm a}), we follow the calculations in \cite[Proposition 4.1]{DSS23} to obtain
  \begin{equation}
  \label{eqn:ener1}
    \begin{aligned}
      \frac{d}{dt}&\int_{\Omega}\frac{1}{2}\left( |u|^2+|F|^2+|\nabla m|^2 \right) dx\\
      &=-\int_{\Omega}\left[ \upmu(\theta)|\nabla u|^2+\kappa(\theta)|\nabla F|^2+\alpha(\theta)|\Delta m+|\nabla m|^2 m|^2 \right]dx, \quad t\in (0, T).
    \end{aligned}
  \end{equation}
  Meanwhile, integrating the $\theta$ equation in \eqref{magneto sys} over $\Omega$, from an integration by parts 
 and the relations $\nabla\cdot u=0$ and $(u,q\cdot \nu)=(0,0)$ on $\partial\Omega$, we have
  \begin{equation}
    \frac{d}{dt}\int_\Omega \theta\, dx=\int_{\Omega}\left[ \upmu(\theta)|\nabla u|^2+\kappa(\theta)|\nabla F|^2+\alpha(\theta)|\Delta m+|\nabla m|^2 m|^2 \right] dx,
    \quad t\in (0, T).
    \label{eqn:ener2}
  \end{equation}
  Now adding \eqref{eqn:ener1} and \eqref{eqn:ener2} yields $\frac{d}{dt}\sE=0$ for $t\in (0,T)$.
  
  \medskip\noindent
  For ({\rm b}), we obtain from the $\theta$ equation in \eqref{magneto sys}, the Maxwell relation $\eta=1+\ln \theta$, and \eqref{eqn:entropy} that
  \begin{equation*}
  	\begin{aligned}
    &\partial_t\eta+u\cdot \nabla \eta&\\
    &=\frac{1}{\theta}(\partial_t \theta+u\cdot \nabla \theta)\\
    &=\frac{1}{\theta}\left( -\nabla\cdot q+\upmu(\theta)|\nabla u|^2+\kappa(\theta)|\nabla F|^2+\alpha(\theta)|\Delta m+|\nabla m|^2 m|^2 \right)\\
    &=-\nabla \cdot  g+\frac{1}{\theta}\left[ \upmu(\theta)|\nabla u|^2+\kappa(\theta)|\nabla F|^2+\alpha(\theta)|\Delta m+|\nabla m|^2 m|^2-\frac{q\cdot \nabla \theta}{\theta} \right]\\
    &=-\nabla\cdot g+r,\qquad t\in (0,T).
\end{aligned}
 \end{equation*}
 By \eqref{eqn:heatflux} we have $-q\cdot \nabla \theta=K\nabla \theta\cdot \nabla \theta\ge\underline{c}|\nabla \theta|^2\ge 0$, and hence $r\ge 0$. 
This implies 
 \begin{equation}
 \label{entropy-derivative}
  \frac{d}{dt} {\sf N}=  \frac{d}{dt} \int_\Omega \eta\,dx =\int_\Omega (-u \cdot \nabla \eta - \nabla\cdot g +r)\,dx = \int_\Omega r\,dx \ge 0, \quad t\in (0,T),
\end{equation}
 and hence the assertion holds.
\end{proof}

\subsection{Entropy and equilibria} 
Here, we follow the arguments in  \cite[Section 1.2]{PrSi16},  see also~\cite{HiPr18},
 to discuss the equilibria of system~\eqref{magneto sys} and their connection to the critical points of the entropy functional.
 We begin with a characterization of the equilibria.
\begin{proposition}
\label{Lyapunov}
\phantom{hallo}
\begin{enumerate}
\item[{\rm (a)}]
The set  $\mathcal{E}$ of equilibria of \eqref{magneto sys} is given by
\begin{equation*}
  \mathcal{E}=\{(u_*, F_*, \theta_*, m_*)\in \{0\}\times\{0\}\times(0, \infty)\times C^\infty(\overline\Omega)\},
\end{equation*}
where $m_*$ satisfies the harmonic map equation with homogeneous Neumann boundary condition
\begin{equation}
\label{harmonic-map}
\left\{\begin{aligned}
\Delta m + |\nabla m |^2 m &= 0 &&\text{in}&&\Omega,\\
  |m| &\equiv 1  &&\text{in}&&\Omega,\\
\partial_\nu m  &=0   &&\text{on}&&\partial\Omega .
\end{aligned}\right.
\end{equation}
Moreover, the equilibrium pressure is given by $\pi_*=-\frac{1}{2}|\nabla m_*|^2 +C$,
where $C$ is some constant. 
\vspace{1mm}
\item[{\rm (b)}]
$-{\sf N}$ is a strict Lyapunov function for \eqref{magneto sys}.
\end{enumerate}
\end{proposition}
\begin{proof}
(a) Suppose $z_*=(u_*, F_*, m_*,\theta_*)$ is an equilibrium for~\eqref{magneto sys}.
Then $\frac{d}{dt}{\sf N}=0$ and it follows from~\eqref{entropy-production}, \eqref{entropy-derivative} that
 $(\nabla u_*, \nabla F_*,\nabla \theta_*, \Delta m_* + |\nabla m_*|^2 m_* )=(0,0,0,0)$, as 
$$-q_* \cdot \nabla \theta_* = K(z_*)\nabla \theta_* \ge \underline{c} |\nabla \theta_*|^2.$$
The boundary conditions $(u_*, F_*)=(0,0)$ readily imply  $(u_*,F_*)=(0,0)$. In addition, we conclude that $\theta_*$ is a constant.
Finally, we can derive from \eqref{magneto sys}$_1$ that $\nabla \pi_*=-\nabla(\frac{1}{2}|\nabla m_*|^2)$, and hence
$\pi_*=-(\frac{1}{2}|\nabla m_)|^2+C)$ with some constant $C$.
 
 \medskip\noindent
 (b) 
 Let  $(u,F, \theta , m)$ be a solution of~\eqref{magneto sys} with the regularity properties of Theorem~\ref{thm:MainLocal}, defined on the 
 maximal interval of existence $[0, T_+(z_0))$. 
Suppose that ${d\over dt} \sN=0 $ on some interval $(t_1, t_2)\subset (0, T_+(z_0))$.
Then $r\ge 0$ implies  $r(t)\equiv0$ in $\Omega$ for  $t\in(t_1, t_2)$.
We conclude as above that $(u(t),F(t))=(0,0)$ and $\theta(t)=c$ for $t\in (t_1,t_2)$, where $c>0$ is a constant. Therefore,
\begin{equation*}
\label{contstant}
(\partial_t u(t), \partial_t F(t), \partial_t \theta (t))=(0,0,0), \quad t\in (t_1,t_2).
\end{equation*}
Moreover, $\Delta m(t) + |\nabla m(t)|^2 m(t) = 0$ for $t\in (t_1,t_2).$
Taking the cross product of both sides of this equation  with $m(t)$ results in 
\begin{equation*}
	m(t)\times \Delta m(t)=-|\nabla m(t)|^2 (m(t)\times m(t))=0,\quad t\in (t_1,t_2).
\end{equation*}
Hence, we conclude that $\pa_t m(t)=0$ for $t\in (t_1,t_2)$. 
This shows that the solution is at equilibrium. 
\end{proof}
\medskip\noindent
We will now provide an informal discussion concerning the critical points of the \emph{constrained entropy functional.}

Suppose that $(u, F, \theta, m)$ is a sufficiently smooth critical point of $\sN$
with $\theta>0$, subject to the constraints $G(m)=(|m|^2-1)/2=0$,
$\sE=\sE_0$, and the boundary conditions in \eqref{magneto sys}. By the Lagrange multiplier method we get
$\lambda_{\sE}\in\mathbb{R}$ and $\lambda_G\in L_2(\Omega)$ such that the first variation of $\sN$ at $z=(u, F, \theta, m)$ with respect to 
$w= (v, J,  \vartheta,n)$ satisfies
\begin{equation*}
  \langle (\sN'+\lambda_{\sE}\sE'+\lambda_G G')(z) | w \rangle=0.
\end{equation*}
We have
\begin{equation*}
	\begin{aligned}
  \langle \sN'(z) | w \rangle&=\int_\Omega \partial_\theta \eta \,\vartheta\, dx=\int_{\Omega}\frac{\vartheta}{\theta}  \, dx,\\
  \langle \lambda_{\sE}\sE'(z)|w \rangle&=\int_{\Omega}\lambda_{\sE}(u\cdot v+F:J +\vartheta+ \nabla m :\nabla  n)\, dx,\\
  \langle \lambda_{G}G'(z)|w\rangle&=\int_{\Omega}\lambda_{G}\,m\cdot n  \, dx.
\end{aligned}
\end{equation*}
This yields the relation
\begin{equation}\label{eqn:critialeq}
  0=\int_{\Omega}[({1}/{\theta}+\lambda_{\sE})\vartheta+\lambda_{\sE}(u\cdot v+F:J- \Delta m \cdot n)+\lambda_G\, m\cdot n ]\, dx,
\end{equation}
where we employed the boundary condition $\partial_\nu m=0$ to derive
$$\int_\Omega \nabla m : \nabla n\,dx = \int_\Omega -\Delta m \cdot n\,dx.$$
Now, setting $(v,J,n)=(0,0,0)$ in \eqref{eqn:critialeq}, it follows that $\lambda_{\sE}=-1/\theta$, as $\vartheta$ can be arbitrary.
 Notice that $\lambda_{\sE}\in\mathbb{R}$, hence $\theta=\theta_*$ is constant and
$\lambda_{\sE}<0$. Similarly,  setting $(u, F)=(0,0)$ we see that  $m$ solves the equation
$-\lambda_{\sE}\Delta m+\lambda_{G}m=0$, with boundary condition
$\partial_\nu m=0$. Moreover, $|m|=1$ on $\Omega$. We can then conclude that
\begin{equation}
\label{lambda-G}
  \lambda_G =\lambda_G|m|^2=\lambda_{\sE}\Delta m\cdot m
  =\lambda_{\sE}\nabla\cdot [(\nabla m)m]-\lambda_{\sE}|\nabla m|^2
  =-\lambda_\sE |\nabla m|^2,
\end{equation}
where we used the fact that $(\nabla m)m=0$ due to $|m|=1$. 
This implies 
$$\Delta m+|\nabla m|^2 m=0,$$ hence $m$ satisfies the harmonic map equation~\eqref{harmonic-map}.
Therefore, the critical points $z=(u,F, \theta, m)$ of the constrained entropy functional correspond exactly to the equilibria $\cE$ of the system.

\medskip
Meanwhile, let 
\begin{equation*}
	\sH(z):=(\sN''+\lambda_\sE \sE''+\lambda_G G'')(z)
\end{equation*}
be the second variation of $\sN$ at a generic point $z=(u,F,\theta, m)$. A direct computation yields
\begin{equation*}
  ( \sH(z) w|w )_\Omega=\int_{\Omega}[-\frac{1}{\theta^2}\vartheta^2+
  \lambda_{\sE}(|v|^2+|J|^2 + |\nabla n|^2 ) + \lambda_G\, |n|^2]\, dx,
\end{equation*}
where $w=(v,J,\vartheta, n)$.
At a critical point $z_*=(0,0,\theta_*, m_*)$ we obtain, in conjunction with the relation $\lambda_G = -\lambda_\sE |\nabla m_*|^2$, see~\eqref{lambda-G},  
\begin{equation*}
 ( \sH(z_*) w|w )_\Omega =\int_{\Omega}[-\frac{1}{\theta_*^2}\vartheta^2- \frac{1}{\theta_*}(|v|^2+|J|^2 +|\nabla n|^2  - |\nabla m_*|^2 |n|^2]\, dx.
 \end{equation*}
A moment of reflection shows that ${N}({\sf E}^\prime (z_*))$ and ${N}( G^\prime(z_*))$,
the  null space of ${\sf E}^\prime (z_*)$ and $G^\prime(z_*)$, respectively, are given by
\begin{equation*}
\begin{aligned}
&{N}({\sf E}^\prime (z_*))  =  \{(v,J,\vartheta, n): \; \int_\Omega \vartheta\,dx =0,\ \int_\Omega \nabla m_*: \nabla n\,dx=0\}, \\ 
& {N}( G^\prime(z_*))=  \{(v,J,\vartheta, n):\;  m_*\cdot n =0\ \text{ in }\ \Omega \}.
\end{aligned}
\end{equation*}
We note that the condition $m_*\cdot n=0$ in $\Omega$ implies $ \int_\Omega \nabla m_*: \nabla n\,dx=0$.
This follows from  the relation
$$0= |\nabla m_*|^2 m_* \cdot n = -\Delta m_* \cdot n$$
and an integration by parts.
Hence, we have 
\begin{equation*}
{N}_*:= {N}({\sf E}^\prime (z_*))\cap {N}( G^\prime(z_*))=\{(v,F,\vartheta, n):\; 
 \int_\Omega \vartheta\,dx =0,\ \ m_*\cdot n=0\ \text{ in }\ \Omega\}.
\end{equation*}
This yields
\begin{equation*}
 \label{eqn:secondVar}
 (\sH(z_*) w|w )_\Omega =-\frac{1}{\theta_*} \int_{\Omega}\left( \frac{ \vartheta^2 } { \theta_*} + |v|^2+|J|^2 +|\nabla n|^2  - |\nabla m_*|^2 |n|^2\right)\, dx,\quad  w=(v,J,\vartheta, n)\in {N}_*.
\end{equation*}
We conclude that $\sH(z_*)|_{N_*}$, the restriction of $\sH(z_*)$ on $N_*$, is negative semi-definite iff
\begin{equation}
\label{stab}
\int_\Omega (|\nabla n|^2  - |\nabla m_*|^2 |n|^2)\,dx \ge 0,\quad n\in C^\infty_c(\Omega; \bR^3).
\end{equation}
In case the relation  \eqref{stab} holds for all $n$ satisfying $m_*\cdot n=0$ in $\Omega$,  $m_*$ is called a stable harmonic map,
see for instance \cite[formula (1.5)]{LiWa06}.
Note that \eqref{stab} holds if $m_*\in\bS^2$ is constant.
This shows that the validity of the relation \eqref{stab} is necessary for the  constrained entropy functional to have a (local) maximum at
a critical point. 
We refer to \cite[Theorem 26.2]{Dei85} for more background on extrema for constrained problems in infinite-dimensional Banach spaces.

\medskip
 Summarizing, we have (informally) shown the following result.
\begin{itemize}
    \item
    The equilibria of \eqref{magneto sys} are precisely the critical points of the entropy functional with prescribed energy. 
 \item
The condition \eqref{stab} is necessary for the constrained entropy functional to have a local maximum at a critical point.
This always holds true in case $m_*\in \bS^2$.
  \end{itemize}

\begin{remark}

\medskip\noindent
(a) We notice that equilibria which are (local) maxima of the constrained entropy functional are the ultimate states where the system is evolving towards.
These are necessarily (locally) stable, as entropy can then no longer increase.

\medskip\noindent
(b) It is stated in \cite[Lemma 5.2]{HNPS16},   see also~\cite[Lemma 1]{HiPr16}  and \cite[Lemma 12.2.4]{PrSi16},
that the nonlinear problem~\eqref{harmonic-map}
admits only constant solutions $m_*\in\bS^2$.
However, this assertion is not correct in the form stated, as the following example shows:
Let $\Omega=\{x\in\bR^3: 0<r_1<|x|<r_2\}$ and $m:\Omega \to\bS^2$ be defined by $m_*(x)=x/|x|.$
Then $m$ is a (non-constant) solution of ~\eqref{harmonic-map}.
The statement holds, for instance, if the magnetic field $m$ is restricted to a spherical cap on $\mathbb S^2$.
\end{remark}


\section{The functional analytic setting}
\label{section:functional}

Following the formulation in \cite[Section 2]{DSS23}, we rewrite the
system \eqref{magneto sys} as
\begin{equation}
  \label{eqn:Rewrit1}
  \left\{
  \begin{aligned}
    \pa_t u+u\cdot \nabla u-\nabla\cdot (\upmu(\theta)\nabla u)+\nabla \pi&=-\nabla\cdot (\nabla m\odot \nabla m)+\nabla \cdot (F F^\top) &&\text{in}&&\Omega, \\
    \nabla\cdot u&=0&&\text{in}&&\Omega, \\
    u&=0&&\text{on}&&\pa\Omega, \\
    \pa_tF   +u\cdot \nabla F  -\nabla\cdot (\kappa(\theta)\nabla F)&=(\nabla u)^\top F &&\text{in}&&\Omega, \\
    F&=0&&\text{on}&&\pa\Omega, \\
        \pa_t \theta   +u\cdot \nabla \theta  -\nabla\cdot (K(z) \nabla \theta) &=\upmu(\theta)|\nabla u|^2+\kappa(\theta)|\nabla F|^2 && &&\\
    &\quad +\alpha(\theta)|\Delta m+|\nabla m|^2 m|^2  &&\text{in}&&\Omega, \\
     \nu \cdot \left( K(z) \nabla \theta \right)&=0&&\text{on}&&\pa\Omega, \\
    \pa_t m-(\alpha(\theta) I_3-\beta(\theta) \sM(m))\Delta m&=\alpha(\theta)|\nabla m|^2 m-u\cdot \nabla m&&\text{in}&&\Omega, \\
    \pa_\nu m&=0&&\text{on}&&\pa\Omega, \\
    (u(0), F(0), m(0), \theta(0))&=(u_0, F_0, m_0, \theta_0)&&\text{in}&&\Omega,
  \end{aligned}
  \right.
\end{equation}
where  $z=(u, F, \theta, m)$,
\begin{equation*}
  \sM(m)=
  \begin{bmatrix}
    0&-m_3&m_2\\
    m_3&0&-m_1\\
    -m_2&m_1&0
  \end{bmatrix}, \qquad\text{with}\quad  m=(m_1, m_2, m_3).
\end{equation*}
One readily verifies that $m\times \Delta m=\sM(m)\Delta m$. 
Notice that the constraint $|m|=1$ is dropped in \eqref{eqn:Rewrit1}. It will be shown later  that the condition $|m|\equiv 1$ is in fact preserved, provided $|m_0|=1$. 

Given any $\tilde{z}=(\tilde{u}, \tilde{F}, \tilde{\theta}, \tilde{m})\in  C^1(\overline{\Omega};\bR^3\times \bR^9\times \bR  \times \bR^3)  $,
we introduce the operators
\begin{equation*} 
\begin{aligned}
&A^1 (\tilde{\theta}) u       : = - P_H (\nabla \cdot ( \upmu (\tilde{\theta}) \nabla u)),  \\
&A^2 (\tilde{\theta}) F       : =- \nabla\cdot (\kappa(\tilde{\theta}) \nabla F),   \\
&A^3(\tilde{z})\theta         : = - \nabla\cdot (K(\tilde{z}) \nabla \theta),   \\
&A^4 (\tilde{\theta},\tilde{m}) m  := - ( \alpha(\tilde{\theta}) I_3 - \beta(\tilde{\theta}) \sM(\tilde{m}) )\Delta m - \alpha(\tilde{\theta} ) |\nabla \tilde{m}|^2 m.
\end{aligned}
\end{equation*}


For later use, we show that the principal part  $A^4_\sharp(\tilde \theta, \tilde m)=( \alpha(\tilde{\theta}) I_3 - \beta(\tilde{\theta}) \sM(\tilde{m}) )\Delta $
of  $A^4 (\tilde{\theta},\tilde{m})$ is (uniformly) normally elliptic, see for instance  \cite[Definition~6.1.1]{PrSi16}.

For this, let $(\tilde \theta, \tilde m)\in C(\overline{\Omega})\times C(\overline{\Omega};\bR^3)$ be given.
The symbol
${S}A^4_\sharp(\tilde \theta,\tilde m)(x,\xi)$ of $A^4_\sharp(\tilde \theta,\tilde m)$ is given by
$$
{S}A^4_\sharp(\tilde \theta,\tilde m)(x,\xi)=\big(\alpha(\tilde{\theta}(x))I_3-\beta(\tilde{\theta}(x))\sM(\tilde m(x))\big)|\xi|^2, 
\quad (x,\xi) \in \overline\Omega\times \bR^3.
$$
For $|\xi|=1$, we obtain for the spectrum $\sigma$ (i.e., the eigenvalues) 
$$
\sigma({S}A^4_\sharp(\tilde \theta,\tilde{m})(x,\xi))=\{\alpha(\tilde{\theta}(x)),\ \alpha(\tilde{\theta}(x))\pm i\beta(\tilde{\theta}(x))|\tilde{m}(x)|\},
\quad (x,\xi) \in \overline\Omega\times \bS^2.
$$
Hence, for each  $(\tilde \theta, \tilde m)\in C(\overline{\Omega})\times C(\overline{\Omega};\bR^3)$ and $(x,\xi) \in \overline\Omega\times \bR^3$,
the spectrum of ${S}A^4_\sharp(\tilde \theta,\tilde m)(x,\xi)$ is contained in a sector $\Sigma_\vartheta =\{z\in \bC\setminus\{0\} : | \arg z | <\vartheta\}$   with opening angle $\vartheta$,
 satisfying
$$
\tan({\vartheta}) \le \frac{\max_{x\in\overline{\Omega}} |\beta (\tilde\theta(x))| |\tilde m(x)|} {\min_{x \in \overline{\Omega}}{\alpha(\tilde\theta}(x))}
\le \frac{\max_{x\in\overline{\Omega}} |\beta (\tilde\theta(x))| |\tilde m(x)|} {\underline \alpha}\le M,
$$
where $\underline\alpha>0$ is the constant introduced in \eqref{assumption} and $M\ge 0$ is an appropriate constant.   Hence $\vartheta<\pi/2$.

\medskip\noindent
Setting 
\begin{equation*}
[\cC^1(\tilde{m})m]_i=  \partial_i \tilde{m} \cdot \Delta m+   \nabla \tilde{m} : \partial_i \nabla  m ,   \quad i=1,2,3,
\end{equation*}
it follows that  
\begin{equation*}
\begin{aligned}
&\cC^1(m)m=\nabla\cdot (\nabla m \odot \nabla m), \quad &&m\in W^2_p(\Omega; \bR^3), \\
&\cC^1 (\tilde{m}) \in \cL(W^2_p(\Omega;\bR^3), L_p(\Omega;\bR^3)),\quad &&\tilde{m}\in C^1(\overline{\Omega};\bR^3).
\end{aligned}
\end{equation*}
For  $(\tilde{\theta},  \tilde{m})\in C(\overline{\Omega})\times C^2(\overline{\Omega};\bR^3)$
we set
\begin{equation*}
\begin{aligned}
\cC^3 (\tilde{\theta},\tilde{m})m:= 
-  \alpha(\tilde{\theta})(\Delta \tilde{m} + |\nabla \tilde{m}|^2 \tilde{m} )  \cdot (\Delta m +|\nabla \tilde m|^2 m).
\end{aligned}
\end{equation*}
One readily verifies that 
\begin{equation*}
\cC^3(\tilde\theta, \tilde m) \in \cL(C^2(\overline{\Omega};\bR^3), C(\overline{\Omega}, \bR^3)).
\end{equation*}

\medskip
\noindent
We now introduce a functional analytic setting to study problem \eqref{eqn:Rewrit1}. For this, let
\begin{equation*}
  X_0:=L_{p, \sigma}(\Omega;\bR^3)\times L_p(\Omega;\bM^3)\times  L_p(\Omega;\bR)\times W^1_p(\Omega;\bR^3),\quad 1<p<\infty.
\end{equation*}
Here,
$
L_{p,\sigma}(\Omega;\bR^3) :=P_H (L_p(\Omega;\bR^3) )
$
is the space of all solenoidal vector fields in $L_p(\Omega; \bR^3)$, where
$P_H: L_p(\Omega;\bR^3) \to L_{p,\sigma}(\Omega;\bR^3) $   is the Helmholtz projection.
For all $s\geq 0$, we define
\begin{align*}
W^s_{p,\sigma}(\Omega;\bR^3) :& = W^s_p(\Omega;\bR^3)\cap L_{p,\sigma}(\Omega;\bR^3).
\end{align*}
Moreover, we set
\begin{equation*} 
\label{differential operators domain}
\begin{aligned}
& X^1_1:=\{u\in W^2_{p,\sigma}(\Omega;\bR^3): \, u=0 \text{ on }\partial\Omega\},   \\
& X_1^2 :=\{F\in  W^2_p(\Omega;\bM^3):\,  F=0 \text{ on }\partial\Omega\}, \\
&  X_1^3:= \{\theta \in W^2_p(\Omega )\}, \\
& X_1^4:= \{m\in W^3_p(\Omega;\bR^3) : \, \partial_\nu m=0 \text{ on } \partial\Omega \}.\\
\end{aligned}
\end{equation*}
For $X_1:= X_1^1\times X_1^2\times  X_1^3\times X_1^4$, we introduce the space of initial data as
\begin{equation*}
\begin{split}
  X_{\gamma,\mu}:&= (X_0,X_1)_{\mu-1/p,p}
  \end{split}
\end{equation*}
for some $\mu\in (1/p,1]$.
Observe that by \cite[Theorem 3.4]{Ama00}  and \cite[Theorem~4.3.3]{Tri78},   
\begin{equation*}
(u,F, \theta, m)\in X_{\gamma,\mu} \ \Leftrightarrow \
\left\{
\begin{aligned}
& u\in W^{2\mu-2/p}_{p, \sigma}(\Omega;\bR^3) &&\text{ and }\;\;   u=0\;\; \text{ on } \;\; \partial \Omega ,\\
& F\in W^{2\mu-2/p}_{p, \sigma}(\Omega;\bM^3)  &&\text{ and } \;\;  F=0\;\; \text{ on } \;\; \partial \Omega, \\
& \theta \in  W^{2\mu-2/p}_p(\Omega), && \\
& m\in  W^{1+2\mu-2/p}_{p}(\Omega;\bR^3)  && \text{ and }\;\;  \partial_\nu m=0  \;\; \text{ on }\;\;  \partial \Omega.
\end{aligned}
\right.
\end{equation*}
In the following, we assume that
 \begin{equation}
\label{indices cond}
p>5 \quad \text{and} \quad \mu>\frac{1}{2}+\frac{5}{2p},
\end{equation}
which ensures that the  embedding
\begin{equation}
\label{embedding-trace-space}
W^{j+2\mu-2/p}_p(\Omega)   \hookrightarrow  C^{j+1}(\overline{\Omega} ),\quad j=0,1,
\end{equation}
holds true.
Given any $\tilde{z}=(\tilde{u}, \tilde{F}, \tilde{\theta}, \tilde{m})\in X_{\gamma,\mu}$,
the operator 
\begin{equation}
\label{A-matrix}
  \sA(\tilde{z}):=
  \begin{bmatrix}
    A^1(\tilde{\theta}) & 0& 0 & P_H \cC^1(\tilde{m}) \\
0 & A^2 (\tilde{\theta}) & 0& 0\\
0 & 0  & A^3(\tilde{z})& \cC^3(\tilde{\theta}, \tilde{m})\\
0& 0&0 &  A^4(\tilde{\theta}, \tilde{m})
  \end{bmatrix}
  \end{equation}
satisfies $\sA(\tilde{z}) \in \cL(X_1,X_0)$.

\medskip
\noindent
In addition, given $z=(u,F,\theta, m)$, we introduce the boundary operator $\sB(\tilde{z})$, defined by
\begin{equation}
\label{B-matrix}
\sB(\tilde{z})z= \nu\cdot {\rm tr}_{\partial\Omega} \left((K(\tilde{z})\nabla \theta) \right),
\end{equation}
where ${\rm tr}_{\partial\Omega}$ is the boundary trace operator.

Finally, we set
\begin{equation}
\label{definition-F}
\sF(z) = \left[ \begin{array}{l}
  P_H \left[\nabla \cdot (F F^\sT) - ( u\cdot\nabla u)\right]  \\[.5em] 
 (\nabla u)^\sT F - u\cdot \nabla F \\ [0.5em]
 \upmu(\theta) |\nabla u|^2 + \kappa(\theta) |\nabla F|^2  - u \cdot \nabla \theta  \\ [0.5em]
 -u \cdot \nabla m 
\end{array}\right].
\end{equation}
Using the notation introduced in \eqref{A-matrix}, \eqref{B-matrix}, and~\eqref{definition-F}
we can restate the nonlinear system \eqref{eqn:Rewrit1} 
in the condensed form
  \begin{equation}
   \label{nonlinear abstract equation}
  \left\{ 
   \begin{aligned}
      \pa_t z+\sA(z)z&=\sF (z)&&\text{in}&&\Omega, \\
      \sB (z)z &=0&&\text{on}&&\pa\Omega, \\
      z(0)&=z_0 &&\text{in}&&\Omega.
    \end{aligned}
    \right.
  \end{equation}

\medskip
For notational brevity, given any $T\in (0,\infty]$, we define
\begin{align*}
& \bE_{0,\mu}(J_T):= L_{p,\mu}(J_T; X_0), \qquad \bE_{1,\mu}(J_T):= W^1_{p,\mu}(J_T; X_0)   \cap L_p(J_T; X_1), \\
& \bB_\mu(J_T)   :=    BU\!C(J_T;X_{\gamma,\mu}),  \\
&
\bF_\mu(J_T)       :=    W^{1/2-1/2p}_{p,\mu}(J_T;L_p(\partial\Omega))\cap L_{p,\mu}(J_T; W^{1-1/p}_p (\partial\Omega))
\end{align*}
and
$$
Y_{\gamma,\mu}= W^{2\mu-1-3/p}_p(\partial\Omega).
$$ 
For future analysis, we also introduce the spaces with vanishing  trace   at $t=0$:
\begin{equation*}
\begin{aligned}
	\zbE_{1,\mu}(J_T)&:=\zbE_{1,\mu}^1(J_T)\times\zbE_{1,\mu}^2(J_T)\times\zbE_{1,\mu}^3(J_T)\times\zbE_{1,\mu}^4(J_T)\\
	&:=\{(z_1, z_2, z_3, z_4)\in \bE_1(J_T): \gamma_0 (z_1, z_2, z_3, z_4) =(0, 0,0, 0) \}, \\
	   {_0}\bB_\mu(J_T)  &:=  \{ z\in   \bB_\mu(J_T) : \gamma_0 z  = 0\},  \\
	 \zbF_\mu(J_T)& :=\{ g\in \bF_\mu(J_T): \gamma_0 g  = 0 \},
\end{aligned}
\end{equation*}
where $\gamma_0$ denotes the  trace operator at $t=0$.

\begin{lemma}
\label{lem:B-mu}
Let $T\in (0, \infty]$. Then we have
\begin{enumerate}
\item[{\rm (a)}]  $\bE_{1,\mu}(J_T)\hookrightarrow  \bB_\mu(J_T)$. 
\vspace{1mm} \\
The embedding constant for the embedding ${_0}\bE_{1,\mu}(J_T)\hookrightarrow {_0}\bB_\mu(J_T)$
is independent of $T$.
\vspace{1mm} 
\item[{\rm (b)}]
  The trace operator $\gamma_0\in \cL(\bE_{1,\mu}(J_T), X_{\gamma,\mu})$ has a bounded right inverse 
$\gamma_0^c \in\cL(X_{\gamma,\mu}, \bE_{1,\mu}(J_T) ).$
\vspace{-3mm} 
\item[{\rm (c)}] For each $z_0\in X_{\gamma,\mu}$, there exists a function $z_*\in \bE_{1,\mu}(\bR_+)$ such that $z_*(0)=z_0$.
\vspace{1mm} 
\item[{\rm (d)}] 
$\bF_\mu(J_T) \hookrightarrow  BU\!C(J; Y_{\gamma,\mu}) \hookrightarrow  BU\!C(J\times \partial\Omega)$. 
Thus,  $\bF_\mu(J_T)$ is a multiplication algebra.
\end{enumerate}
\end{lemma}
\begin{proof}  
Although the properties listed above are known, for the reader's convenience, we include a proof nonetheless.

\medskip\noindent
(a)  By \cite[Lemma 2.5]{MeSc12}, there exists an extension operator $\cE_{J_T}\in \cL(\bE_{1,\mu}(J_T), \bE_{1,\mu}(\bR_+))$.
Moreover,  one shows that $\bE_{1,\mu}(\bR_+)$ is continuously translation invariant.
The first assertion follows now from  \cite[Proposition III.1.4.2]{Ama95}.

By Proposition~\ref{extension-zero}, there exists an extension operator $\cE^0_{J_T}\in \cL({_0}\bE_{1,\mu}(J_T), {_0}\bE_{1,\mu}(\bR_+))$
whose norm is independent of $T$. The second assertion follows then from the commutativity of the diagram
\begin{equation*}
\begin{aligned}
&{_0}\bE_{1,\mu}(J_T)\quad\overset{\cE^0_{J_T}}{\longrightarrow} &&{_0}\bE_{1,\mu}(\bR_+) \\
\phantom{X}&\quad\downarrow                       &&\quad\ \downarrow \\ 
&{_0}\bB_{\mu}(J_T)\qquad\overset{{\mathcal R}_{\!J}}{\longleftarrow} &&\ {_0}\bB_\mu(\bR_+),
\end{aligned}
\end{equation*} 
where ${\mathcal R}_{\!J}$ denotes the restriction operator.

\smallskip\noindent
(b) Pick any $\tilde z=(\tilde u, \tilde F, \tilde\theta, \tilde m)\in X_{\gamma,\mu}$. 
Then we define for $z_0\in(u_0,F_0,\theta_0,m_0)\in X_{\gamma, \mu}$
$$
(\gamma^c_0 z_0)(t)
:=\left( e^{-t(I-A^1(\tilde \theta))}u_0, e^{-t( I-A^2(\tilde \theta))}F_0, \cR_\Omega e^{-t(I-\Delta)}\cE_\Omega \theta_0, e^{-t(\omega I-A^4(\tilde \theta,\tilde m))}m_0 \right),
\quad t\in J,
$$
where $D(A^i(\tilde z)):=X^i_1$ for $i\in\{1,2,4\}$, with $\tilde z$ interpreted appropriately, and  $\omega$ is a sufficiently large positive constant, 
to be specified below. 

Moreover, $\cE_\Omega\in \cL(W^{1+2\mu-2/p}_p(\Omega),W^{1+2\mu-2/p}_p(\bR^3))$
denotes an appropriate extension operator, $\cR_\Omega$ is the restriction operator, and $\Delta$ is the Laplacian defined on $\bR^3$.
  
It follows from the maximal regularity result in Proposition~\ref{Pro: MR linear} below and \cite[Proposition 3.5.2(ii)]{PrSi16}  that, for each $\tilde z\in X_{\gamma,\mu}$, the operators
$\omega I - A^i(\tilde z)$,  $i\in\{1,2,4\}$, generate a strongly continuous exponentially stable analytic semigroup on $X^i_0$,  
where $\omega=1$ for $i\in\{1,2\}$, and $\omega$ is sufficiently large for  $i=4$.
The assertion follows then from \cite[Theorem 1.14.5]{Tri78}.

\smallskip\noindent
(c) This follows by choosing $T=\infty$ in (b) and setting $z_*= \gamma^c_0 z_0$.

\smallskip\noindent
(d)
The first assertion follows from  \cite[formula (4.10)]{MeSc12}, and the second one from \eqref{indices cond} and Sobolev embedding.
\end{proof}

 \section{Linearized problems}
 \label{section:linear}
Before investigating the nonlinear system~\eqref{nonlinear abstract equation}, 
we first study some related linear problems.
We start with the system
  \begin{equation}
  \label{linear-tilde}
  \left\{
    \begin{aligned}
      \pa_t z+\sA (\tilde{z})z&={\sf f}(t)&&\text{in}&&\Omega, \\
      \sB (\tilde{z})z&={\sf g}(t)&&\text{on}&&\pa\Omega, \\
      z(0)&=z_0 &&\text{in}&&\Omega,
    \end{aligned}
    \right.
  \end{equation}
where  $\tilde z\in X_{\gamma,\mu}$.
\goodbreak
\begin{proposition}
\label{Pro: MR linear}
Assume \eqref{assumption} and \eqref{indices cond}. Let $\tilde{z}\in X_{\gamma,\mu}$ and  $T>0$ be given. 
\begin{itemize}
\item[{\rm (a)}]
For every $({\sf f},{\sf g},z_0) \in  \bE_{0,\mu}(J_T) \times \bF_\mu(J_T)\times X_{\gamma,\mu}$, 
where   ${\sf g} $ satisfies  the compatibility condition
$\sB(\tilde{z})z_0 = \gamma_0 {\sf g}  $,
the linear initial boundary value problem \eqref{linear-tilde} admits a unique solution $z\in \bE_{1,\mu}(J_T)$.
Moreover, 
$$
\sL(\tilde{z}):=(\sA(\tilde{z}), \sB(\tilde{z}), \gamma_0)\in \cL is(\bE_{1,\mu}(J_T),{\mathbb D}_\mu(\tilde z,T)),
$$
where
$
 {\mathbb D}_\mu(\tilde z,T):= \{({\sf f}, {\sf g} ,z_0)\in \bE_{0,\mu}(J_T)\times \bF_\mu(J_T)\times X_{\gamma,\mu}:  \sB(\tilde{z})z_0 = \gamma_0 {\sf g} \}.
$
\item[{\rm (b)}] 
For each $\tilde z\in X_{\gamma,\mu}$ and $({\sf f}, {\sf g})\in  \bE_{0,\mu}(J_T)\times {_0}\bF_\mu (J_T)$,
let 
$z=:{\sf S}(\tilde z) ({\sf f}, {\sf g})$ 
be the (unique) solution of \eqref{linear-tilde} with $z_0=0$.
Then
\begin{equation}
\label{S-Ck}
[\tilde z\mapsto {\sf S}(\tilde z)]\in C^1( X_{\gamma,\mu}, \cL(\bE_{0,\mu}(J_T)\times {_0}\bF_\mu (J_T), {_0}\bE_{1,\mu}(J_T)).
\end{equation}
Moreover, given any $T_*>0$,  the norm of ${\sf S}$ is uniform in  $T\in (0,T_*]$.
\end{itemize}
\end{proposition}

 \begin{proof}
(a) The proof is based on the upper triangular structure of $\sA$ and the results in  \cite[Theorems~6.3.2, 6.3.3  and 7.3.2]{PrSi16}.
Let $$
{\sf f}=(f_1,f_2,f_3,f_4)\in \bE_{0,\mu}(J_T), \quad {\sf g} \in \bF_\mu(J_T),\quad
z_0\in X_{\gamma,\mu}
 $$
 be given.
We first solve the  equation for $m$:
  \begin{equation*}
    \begin{aligned}
      \pa_t m+A^4(\tilde{\theta}, \tilde{m})m&=f_4 &&\text{in}&&\Omega, \\
      \partial_\nu m&=0 &&\text{on}&&\pa\Omega, \\
      m(0)&=m_0 &&\text{in}&&\Omega
    \end{aligned}
  \end{equation*}
and obtain a unique solution 
$$
m\in L_{p,\mu}(J_T; W^3_p(\Omega;\bR^3) ) \cap W^1_{p,\mu}(J_T; W^1_p(\Omega;\bR^3) ) 
$$
by means of \cite[Theorem~6.3.3]{PrSi16}. In view of   \eqref{embedding-trace-space}, 
$$
(\cC^1(  \tilde{m}) m,0, \cC^3(\tilde{\theta}, \tilde{m}) m ,  0)     \in \bE_{0,\mu}(J_T).
$$
Therefore, the remaining equations of \eqref{linear-tilde} can be rewritten as 
  \begin{equation*}
    \begin{aligned}
      \pa_t u+A^1(\tilde{\theta})u&=f_1 - \cC^1(  \tilde{m}) m &&\text{in}&&\Omega, \\
      u&=0 &&\text{on}&&\pa\Omega, \\
            \pa_t F+A^2(\tilde{\theta})F&=f_2 &&\text{in}&&\Omega, \\
     F&=0 &&\text{on}&&\pa\Omega, \\
           \pa_t \theta +A^3(\tilde{z})\theta&=f_3 - \cC^3(\tilde{\theta}, \tilde{m}) m &&\text{in}&&\Omega, \\
      \nu\cdot {\rm tr}_{\partial\Omega} (K(\tilde{z}) \nabla \theta) &={\sf g}  &&\text{on}&&\pa\Omega, \\
     (u(0), F(0), \theta(0))&=(u_0,F_0,\theta_0) &&\text{in}&&\Omega.
    \end{aligned}
   \end{equation*}
The existence and uniqueness of a solution then follows from \cite[Theorems~6.3.2  and 7.3.2]{PrSi16}.

\medskip\noindent
(b) 
Lemma~\ref{lem: Nemyskii} implies that 
$$[\tilde z \mapsto (\sA(\tilde z), \sB(\tilde z))] \in C^1 (X_{\gamma, \mu}, \cL({_0}\bE_{1,\mu}(J_T), \bE_{0,\mu}(J_T)\times {_0}\bF_\mu(J_T))).$$
The continuous differentiability of the map ${\sf S}$ follows from the following diagram:
\begin{equation*}
\begin{aligned}
  &\ \tilde z\quad\mapsto \qquad (\pa_t +\sA(\tilde z), \sB(\tilde z))\qquad \qquad\mapsto \qquad(\pa_t +\sA(\tilde z), \sB(\tilde z))^{-1}\\
& X_{\gamma,\mu} \to  \cL({_0}\bE_{1,\mu}(J_T), \bE_{0,\mu}(J_T)\times {_0}\bF_\mu(J_T)) \to 
 \cL(\bE_{0,\mu}(J_T)\times {_0}\bF_\mu(J_T), {_0}\bE_{1,\mu}(J_T) ).
\end{aligned}
\end{equation*}
Fix $T_*>0$.
For any $0<T\le T_*$, it follows from Proposition~\ref{extension-zero} that there exists an extension map $\cE_{J_T}: \bX(J_T)\to \bX(J_{T_*})$
with $\bX\in \{\bE_{0,\mu}, \zbE_{1,\mu}, \zbF_\mu\}$. 
Moreover, the norm of $\cE_{J_T}$ is  uniform in $T\in (0,T_*] $.
Given any $({\sf f}, {\sf g})\in \bE_{0,\mu}(J_T)\times \zbF_\mu(J_T)$, 
 let $(\hat{{\sf f}},\hat{{\sf g}})=(\cE_{J_T} {\sf f},\cE_{J_T} {\sf g}) $.
Part (a) implies that one can find a unique function $\hat{z}\in \zbE_{1,\mu}(J_{T_*})$ such that ${\hat{z}}= {\sf S} (\hat{{\sf f}},\hat{{\sf g}} )$. 
Put $z=\hat{z}|_{[0,T]}$.
It is clear that $\sL(\tilde{z}) z= ({\sf f},{\sf g},0)$.
Direct computations yield
\begin{align*}
\| z \|_{\bE_{1,\mu}(J_T)} \leq \| \hat {z}\|_{\bE_{1,\mu}(J_{T_*})}   
\leq   C  \left( \| \hat{{\sf f}}\|_{\bE_{0,\mu}(J_{T_*})} + \| \hat{{\sf g}}\|_{\zbF_\mu(J_{T_*})} \right) 
\leq  C  \left( \| {\sf f}\|_{\bE_{0,\mu}(J_T)} + \| {\sf g}\|_{\zbF_\mu(J_T)} \right)
\end{align*}
with a constant $C$ that is independent of $T$. 
Therefore, the norm of ${\sf S}(\tilde{z}) $ is uniform in $T\in (0,T_*]$.
\end{proof} 

Next we consider the non-autonomous linear system
\begin{equation}
   \label{linear-time}
   \left\{
    \begin{aligned}
      \pa_t z +\sA(z_*(t))z&={\sf f}(t)&&\text{in}&&\Omega, \\
      \sB (z_*(t))z &={\sf g}(t)&&\text{on}&&\pa\Omega, \\
      z(0)&=z_0 &&\text{in}&&\Omega,
    \end{aligned}
    \right.
  \end{equation}
where  $z_*\in \bE_{1,\mu}(J_T)$ is given.

\begin{proposition}
\label{pro: linear-time}
Assume \eqref{assumption}   and \eqref{indices cond}. 
Let $T>0$ and $z_*\in \bE_{1,\mu}(J_T)$ be given. Then the system \eqref{linear-time}
has a unique solution $z=S({\sf f},{\sf g},z_0)\in \bE_{1,\mu}(J_T)$
if and only if 
\begin{equation*}
({\sf f},{\sf g}, z_0)\in \bD_\mu(z_*,T):=\{\bE_{0,\mu}(J_T)\times \bF_\mu(J_T)\times X_{\gamma,\mu} : {\sf B}(z_*(0))z_0={\sf g}(0)\}.
\end{equation*}
    In this case, there is a constant $c_1=c_1(T)>0$ such that
    \begin{equation}\label{indep T}
    \| z \|_{\bE_{1,\mu}(J_T)}\le c_1(T) \big( \| {\sf f} \|_{\bE_{0,\mu}(J_T)} + \| {\sf g} \|_{\bF_\mu(J_T)} + \|z_0\|_{X_{\gamma,\mu}}\big).
    \end{equation}
Given any $T_*>0$, the constant $c_1(T)$ is uniform in $T\in (0,T_*] $ in case ${\sf g}\in {_0}\bF_\mu(J_T)$ and $z_0=0$.  
\end{proposition}
\begin{proof}
Let $z_*\in \bE_{1,\mu}(J_T)$ be given. In the following, we use the notation
\begin{equation*}
\sA_*(t):= \sA(z_*(t)),\quad \sB_*(t):= \sB(z_*(t)),\quad {\sf S}_*(t):= {\sf S}(z_*(t)), \quad K_*(t):=K(z_*(t))
\end{equation*}
for $t\in [0,T]$, 
where the solution operator ${\sf S}$ is defined in Proposition~\ref{Pro: MR linear}.
By  Lemma~\ref{lem:B-mu}(a) we know that $z_*\in C([0,T]; X_{\gamma,\mu})$ and therefore, the set  
$\{z_*(s) : s\in [0,T]\}\subset X_{\gamma,\mu}$
is compact. We conclude from \eqref{S-Ck} that there exists a constant $M=M(T)>0$ such that  for $\nu\in \{\mu,1\}$
\begin{equation}
\label{S(t)-bounded}
\| {\sf S}_*(s)\|_{\cL(\bE_{0,\nu}(J_T)\times {_0}\bF_\nu (J_T), {_0}\bE_{1,\nu}(J_T)) }\le M,\quad s\in [0,T].
\end{equation}
By Lemma~\ref{lem: Nemyskii} 
and absolute continuity of the integral $\|{\rm tr}_{\partial\Omega} K\|_{\bF_\mu (J_T)}$,
we can find a partition
$$0=t_0<t_1 \cdots <t_n =T\quad\text{of} \quad [0,T]$$
such that
\begin{equation}
\label{z-small}
\begin{aligned}
 \max_{t\in I_j}\| z_*(t)-z_*(t_j)\|_{X_{\gamma,\mu}}& \le \eta, \\
 \|{\rm tr}_{\partial\Omega} \left(K_*(\cdot)-K_*(t_j) \right)\|_{\bF_\nu (I_j)}   &\le \eta 
\end{aligned}
\end{equation}
for a  pre-determined  (fixed) number $\eta$, 
where $I_j=[t_j, t_{j+1}]$ for $j=0,\ldots, n-1$, and where we set 
\begin{equation*}
\nu:= \left\{
\begin{aligned} 
& \mu &&\text{if} \quad j=0, \\ 
 &1     && \text{if} \quad j=1,\ldots, n-1.
\end{aligned}
\right.
\end{equation*}
Now it follows from \eqref{z-small}, and Lemma  \ref{lem: multiplication}(i) that 
\begin{equation}
\label{AB-small}
\begin{aligned}
\| \sA_*(\cdot )-\sA_*(t_j)\|_{\cL(_{0}\bE_{1,\nu}(I_j), \bE_{0,\nu}(I_j))}& \le 1/(4M), \\ 
\| \sB_*(\cdot )-\sB_*(t_j)\|_{\cL(_{0}\bE_{1,\nu}(I_j), _{0}\bF_{\nu}(I_j))}& \le 1/(4M)
\end{aligned}
\end{equation}
for $j=0,\ldots,n-1.$
In fact, we will first choose a partition point $t_1$ so that the properties of \eqref{AB-small} hold true for $I_0=[0,t_1]$, 
and then partition the remaining interval $[t_1,T]$ if needed.

\medskip
We will consider problem~\eqref{linear-time} on subintervals $I_j$.
In the first step, we deal with the interval $I_0=[t_0, t_1]=[0,t_1]$.
In order to resolve the compatibility condition $\sB_*(0)z_0={\sf g}(0)$, we consider the linear problem
\begin{equation}
   \label{linear-w}
   \left\{
    \begin{aligned}
      \pa_t w +\sA_*(0)w&={\sf f}(t)&&\text{in}&&[0,t_1]\times \Omega, \\
      \sB_* (0)w &={\sf g}(t)&&\text{on}&&[0,t_1]\times\pa\Omega, \\
      w(0)&=z_0 &&\text{in}&&\Omega.
    \end{aligned}
    \right.
\end{equation}
Let $w\in \bE_{1,\mu}([0,t_1])$ be the unique solution of  \eqref{linear-w} (whose existence is guaranteed by Proposition~\ref{Pro: MR linear}) and consider the system
\begin{equation}
   \label{problem-1}
   \left\{
    \begin{aligned}
      \pa_t \hat z +\sA_*(t)\hat z & =\hat{\sf f}(t)&&\text{in}&&[0,t_1]\times \Omega, \\
                          \sB_*(t)\hat z  & =\hat{\sf g}(t)&&\text{on}&&[0,t_1]\times\pa\Omega, \\
     \hat z(0)&=0 &&\text{in}&&\Omega,
    \end{aligned}
    \right.
\end{equation}
where
\begin{equation*}
\begin{aligned}
 \hat{\sf f}(t) &=-[\sA_*(t)-\sA_*(0)]w(t), \\
 \hat{\sf g}(t) &=- [\sB_*(t)-\sB_*(0)]w(t).
\end{aligned}
\end{equation*}
Suppose $\hat z\in \bE_{1,\mu}([0,t_1])$ is a solution of \eqref{problem-1}.
Then one verifies that the function
$z_1=w+\hat z \in\bE_{1,\mu}([0,t_1])$
is a solution of~\eqref{linear-time} on the interval $[0,t_1]$.

Hence, it remains to show that~\eqref{problem-1} has a (unique) solution.
For this, we first note that the necessary compatibility condition ${\sf B}_*(0)\hat z(0)=\hat{\sf g}(0)$  is satisfied. 
To show the solvability, we rewrite~\eqref{problem-1} as 
\begin{equation}
   \label{problem-1-altered}
   \left\{
    \begin{aligned}
      \pa_t \hat z +\sA_*(0)\hat z + {\sf R}_1(t) \hat z& =\hat{\sf f}(t)&&\text{in}&&[0,t_1]\times\Omega, \\
                          \sB_*(0)\hat z + {\sf R}_2(t) \hat z & =\hat{\sf g}(t)&&\text{on}&&[0,t_1]\times\pa\Omega, \\
     \hat z(0)&=0 &&\text{in}&&\Omega,
    \end{aligned}
    \right.
\end{equation}
where
\begin{equation*}
{\sf R}_1(t)=  [\sA_*(t)-\sA_*(0)],\quad  {\sf R}_2(t) =[\sB_*(t)-\sB_*(0)],\quad t\in [0,t_1].
\end{equation*}
It follows from \eqref{S(t)-bounded} and \eqref{AB-small} that
$$
\|({\sf R}_1(\cdot), {\sf R}_2(\cdot )){\sf S}_*(0)\|_{\cL(\bE_{0,\mu}([0,t_1])\times {_0}\bF_{\mu}([0,t_1])) }\le 1/2,
$$
so that $\big[   I+ ({\sf R}_1(\cdot), {\sf R}_2(\cdot )){\sf S}_*(0)   \big]\in \cL(\bE_{0,\mu}([0,t_1])\times {_0}\bF_{\mu}([0,t_1]))$ is invertible.
Hence,
\begin{equation*}
\hat z= {\sf S}_*(0) \big[  I+ ({\sf R}_1(\cdot), {\sf R}_2(\cdot)){\sf S}_*(0)     \big]^{-1}(\hat {\sf f}, \hat {\sf g})\in {_0\bE}_{1,\mu}([0,t_1])
\end{equation*}
is the (unique) solution of \eqref{problem-1-altered} on the interval $[0,t_1]$.
It follows that $z_1:= \hat z +w\in \bE_{1,\mu}([0,t_1])$ is a solution of~\eqref{linear-time} 
on the time interval $[0,t_1]$. 

Assume that there exists another solution $\widetilde{z}\in \bE_{1,\mu}([0,t_1])$ to \eqref{linear-time}  on $[0,t_1]$.
Then $h =z -\widetilde{z} $ solves  \eqref{problem-1-altered} with $\hat{\sf f}=0$ and $\hat{\sf g}=0$. 
The unique solvability of \eqref{problem-1-altered} implies that $h=0$.
This proves the uniqueness of a solution on $[0,t_1]$.

We can now repeat the steps above for the interval $[t_1,t_2]$. In this case, we consider the problem
\begin{equation}
   \label{linear-time-2}
   \left\{
    \begin{aligned}
      \pa_t z +\sA_*(t_1+t)z&={\sf f}(t_1+t)&&\text{in}&&[0,t_2-t_1]\times \Omega, \\
      \sB_* (t_1+t)z &={\sf g}(t_1+t)&&\text{on}&&[0,t_2-t_1]\times\pa\Omega, \\
      z(0)&=z_1(t_1) &&\text{in}&&\Omega,
    \end{aligned}
    \right.
\end{equation}
where $z_1$ is the function obtained in step 1.
As $z_1$ solves~\eqref{linear-time} on the time interval $[0,t_1]$, the compatibility condition
$\sB_*(t_1)z(0)=\sB_*(t_1)z_1(t_1)={\sf g}(t_1)$
is satisfied.
Repeating the arguments of step 1, we obtain a unique solution 
$z_2\in \bE_{1,1}([0,t_2-t_1])$ of~\eqref{linear-time-2}.
Let
\begin{equation*}
z(t):=
\left\{
\begin{aligned}
&z_1(t), && 0\le t\le t_1\\
&z_2(t-t_1), && t_1\le t\le t_2.
\end{aligned}
\right.
\end{equation*}
 As $z_1(t_1)=z_2(0)$ we conclude that $z\in W^1_1((0, t_2);X_0)$. It is then easy to see that $z\in \bE_{1,\mu}([0,t_2])$.

\medskip
 We can now repeat the steps above to find a solution $z\in \bE_{1,\mu}(J_T)$ of \eqref{linear-time} on $[0,T]$.
To show uniqueness, let
$$
t_*:=\sup\{ t\in [0,T]: \text{\eqref{linear-time} has a unique solution on $[0,t]$}\}. $$
By step 1, the set under consideration is non-empty and, therefore, $t_*$ is well-defined.
Suppose $t_*<T$.  We can then repeat step 2 from above to get a unique solution on $[t_*, t_* +\delta]$ for some $\delta$.
This leads to a contradiction.

Given  $T_*>0$, when $	z_0=0$ and ${\sf g}\in \zbF_\mu(J_T)$,  the uniformity of the constant  $C $ in \eqref{indep T} with respect to $T\in (0,T_*]$  can be shown in an analogous way to Proposition~\ref{Pro: MR linear}(b).   
\end{proof}

For simplicity, we will introduce the following notation
\begin{equation}
\label{AB}
\cA(z):= \sA(z)z  \quad \text{and} \quad \cB(z):=\sB(z)z.
\end{equation}
It follows from Proposition~\ref{Prop: mapping properties} that
\begin{equation*}
\begin{aligned}
\cA & \in C^1(\bE_{1,\mu}(J_T),  \bE_{0,\mu}(J_T)),                     \quad &&  \cA^\prime (z_*)z=  \sA(z_*)z + [\sA^\prime(z_*)z] z_*, \\
\cB & \in C^1(\bE_{1,\mu}(J_T),  \bF_\mu(J_T)) ,                           \quad &&  \cB^\prime (z_*)z= \sB(z_*)z + [\sB^\prime(z_*)z] z_* .
\end{aligned}
\end{equation*}
\noindent
Next, we 
will study solvability of the linearized system   
\begin{equation}
   \label{linearized}
   \left\{
    \begin{aligned}
      \pa_t z   + \cA'(z_*(t))z - \sF^\prime(z_*(t))z&={\sf f}(t)&&\text{in}&&\Omega, \\
                    \cB'(z_*(t))z  &={\sf g}(t)&&\text{on}&&\pa\Omega, \\
      z(0)&=z_0 &&\text{in}&&\Omega,
    \end{aligned}
    \right.
  \end{equation}
where $z_*\in \bE_{1,\mu}(J_T)$.    We obtain the following result.
\goodbreak
\begin{proposition}\label{Prop: linearized abstract pb}
Let $z_*\in\bE_{1,\mu}(J_T)$ be given. Then the linearized system \eqref{linearized}
  has a unique solution $z=\cS({\sf f},{\sf g},z_0)\in \bE_{1,\mu}(J_T)$
if and only if 
\begin{equation*}
\begin{split}
({\sf f},{\sf g}, z_0) \in \widetilde{\bD}_\mu(z_*,T):&=\{\bE_{0,\mu}(J_T)\times \bF_\mu(J_T)\times X_{\gamma,\mu} : 
  \cB(z_*(0))z_0  ={\sf g}(0) \}.
\end{split}
\end{equation*}
In this case, there is a constant $c_2=c_2(T)>0$ such that
    \begin{equation*}
    \| z \|_{\bE_{1,\mu}(J_T)}\le c_2(T) \big( \| {\sf f} \|_{\bE_{0,\mu}(J_T)} + \| {\sf g} \|_{\bF_\mu(J_T)} + \|z_0\|_{X_{\gamma,\mu}}\big).
    \end{equation*}
   Given any $T_*>0$, the constant $c_2$ is independent of $T\in (0,T_*]$ in case ${\sf g}\in {_0}\bF_\mu(J_T)$ and $z_0=0$.\\
\end{proposition}
\begin{proof}
We observe that for $z\in {_0}\bE_{1,\nu}(I)$  and any interval $I$ contained in $[0,T]$,
\begin{equation*}
\begin{aligned}
 \left(\int_I \|t^{1-\nu} [\sA^\prime(z_*(t))z(t)] z_*(t) \|^p_{X_0}\,dt\right)^{1/p} 
& = \left(\int_I \| [\sA^\prime(z_*(t))z(t)] t^{1-\nu}z_*(t) \|^p_{X_0}\,dt\right)^{1/p} \\
& \le  M\left(\int_I \| t^{1-\nu} z_*(t) \|^p_{X_1}\,dt\right)^{1/p} \sup_{t\in I} \|z(t)\|_{ X_{\gamma,\mu}  }  \\
&  \le  CM \left(\int_I \| t^{1-\nu} z_*(t) \|^p_{X_1}\,dt\right)^{1/p} \|z \|_{{_0}\bE_{1,\mu}(I)},
\end{aligned}
\end{equation*}
where the constants $C$ and $M$ do not depend on the length of $I$, due to Lemma~\ref{lem:B-mu}(a).
Here $\nu=\mu$ in case the interval $I$ contains $0$, and $\nu=1$ otherwise.
By absolute continuity of the integral 
$
\int_{[0,T]} \| t^{1-\nu}  z_*(t)\|^p_{X_1}\,dt,
$
we can, again, chose a partition $\{I_j: 0\le j\le n-1\}$ of $[0,T]$ such that
\begin{equation*}
\begin{aligned}
 \left(\int_{I_j}\|t^{1-\nu} z_*(t) \|^p_{X_1}\,dt\right)^{1/p}  \le \eta, 
\end{aligned}
\end{equation*}
where  $\eta>0$ is a given, predetermined (small) number. 

Let $z_*=(u_*, F_*, \theta_*, m_*)\in \bE_{1,\mu}(J_T)$ be given. Then
$
  [\sB^\prime (z_* )z]z_*=  {\tr}_{\partial\Omega} [ K'(z_* )z]\nabla \theta_* .
$
It follows from Lemma~\ref{lem: multiplication}(ii)  that for any  $z=(z_i)_{i=1}^{16} \in {_0}\bE_{1,\nu}(I)$
\begin{align*}
\|  {\rm tr}_{\partial\Omega} ([ K' (z_* ) z ] \nabla \theta_*) \|_{ \zbF_\nu (I) }
 \leq C \sum_{i=1}^{16}  \| {\rm tr}_{\partial\Omega} (\partial_{i} K (z_* ) \nabla \theta_* ) \|_{  {_0}\bF_\nu (I )}   \|{\rm tr}_{\partial\Omega} z_i\|_{\zbF_{1,\nu}(I)},
\end{align*} 
where the constant $C$ is independent of the length of the interval $I\subset [0,T]$. 
 By absolute continuity,  for any given $\eta>0$, there exists a partition $\{I_j: 0\le j\le n-1\}$ of $[0,T]$ (which can be chosen to be compatible with the one 
 for $\sA^\prime$)  such that

\begin{equation*}
\| {\rm tr}_{\partial\Omega} ([ K' (z_* ) z ] \nabla \theta_* )\|_{{_0}\bF_{\nu}(I_j)}\le  \eta \, \| z \|_{ \zbE_{1,\nu}(I_j)}, \quad 0\le j\le n-1 ,
\end{equation*}
for some constant $C$ that is independent of the length of the interval $I_j$. 
See  Proposition~\ref{extension-zero} and \cite[Theorems~4.2 and~4.5]{MeSc12}.
Finally, we note that the term $\sF^\prime(z_*(t))z$  is of lower order in $z$ and can therefore be handled by a standard perturbation argument.
The assertion follows now by similar arguments as in the proof of Proposition~\ref{pro: linear-time}.
\end{proof}
\goodbreak
\section{Local well-posedness}\label{section:localwell}

\begin{theorem}[Local existence and uniqueness for the abstract problem]
\label{Thm:wellposed abstract}
Assume \eqref{assumption} and \eqref{indices cond}.  
\begin{itemize}
\item[{\rm (a)}]
Let 
$ 
\cM_\mu= \{z\in X_{\gamma,\mu}: \sB(z)z=0 \}.
$ 
Then for every  $z_0\in \cM_\mu$, there exists $T>0$ such that the nonlinear system~\eqref{nonlinear abstract equation} has a unique solution
$z\in\bE_{1,\mu}(J_T)$.
The solution can be continued to a maximal solution $z=z(\cdot,z_0)$ on an interval $[0, T_+(z_0))$.
\vspace{1mm}
\item[{\rm (b)}]
Let $T<T_+(z_0)$. Then there exists a number $\rho>0$ such that the system 
  \eqref{nonlinear abstract equation} has a unique solution  $z(\cdot,w_0)\in \bE_{1,\mu}(J_T)$
  for  each initial value $w_0\in \cM_\mu\cap B_{X_{\gamma,\mu}}(z_0, \rho)$.
Moreover,
the mapping
\begin{equation*}
[w_0\mapsto z(\cdot ,w_0)]:  \cM_\mu \cap B_{X_{\gamma,\mu}}(z_0, \rho)\to \bE_{1,\mu}(J_T) 
\end{equation*}
is Lipschitz continuous. \\
Hence,  \eqref{nonlinear abstract equation} generates a Lipschitz continuous semiflow on $\cM_\mu$.
\vspace{1mm}
\item[{\rm (c)}]
 Let $T<T_+(z_0)$ and   $z=z(\cdot,z_0)$ be the (unique) solution of~\eqref{nonlinear abstract equation}. Then
 $$ tz \in W^2_{p,\mu}(J_T; X_0)\cap W^1_{p,\mu}(J_T; X_1).$$
 Moreover, $z\in C^1((0,T]; X_{\gamma,\mu})$.
\end{itemize}
\end{theorem}

\begin{proof}
In this proof, we will follow the ideas in   \cite[Theorem~14]{LPS06} and \cite[Proposition~4.3.2]{MeyriesThesis}.

\medskip\noindent
(a) 
Fix $z_*\in \bE_{1,\mu}(\bR_+)$ with $z_*(0)=z_0$, whose existence is supported by Lemma~\ref{lem:B-mu}(b).
We put 
\begin{equation*}
\begin{aligned}
\sA_* (t) z  &=  \cA'(z_* (t)) z  - \sF^\prime(z_* (t))  z, \\
\sB_* (t) z  &= \cB'(z_* (t))z  , 
\end{aligned}
\end{equation*}
where the functions $(\cA, \cB)$ are defined in \eqref{AB}, 
and consider the linear problem
\begin{equation}
\label{linear pb-z*}
\left\{
\begin{aligned}
      \pa_t z+\sA_* (t)  z  &=    \cA^\prime(z_*)z_*-\cA(z_*) +\sF(z_*) - \sF'(z_*)z_* &&\text{in}&&\Omega, \\
      \sB_*(t)  z                &=    \cB^\prime (z_*)z_* - \cB(z_*)  &&\text{on}&&\pa\Omega,\\
      z(0)&=z_0  &&\text{in}&&\Omega.
    \end{aligned}
    \right.
\end{equation}
Note that the compatibility condition
$$
\sB_*(0)  z_0=  \cB^\prime (z_0)z_0 - \cB(z_0) 
$$
is satisfied, as $\cB(z_0)=\sB(z_0)z_0=0$ by assumption.  
Therefore, Proposition~\ref{Prop: linearized abstract pb} implies that for any $T_0>0$, \eqref{linear pb-z*} has a unique solution $w\in \bE_{1,\mu}(J_{T_0})$.
Fix $T_0, R_0>0$. 
For every $T\in (0,T_0]$ and $R\in (0,R_0]$, we define a closed set in $\bE_{1,\mu}(J_T)$ by
$$
\Sigma(T,R)=\{ z\in \bE_{1,\mu}(J_T): \|z-w\|_{\bE_{1,\mu}(J_T)} \leq R,\, \gamma_0 z= z_0 \}.
$$
Observe that, by \cite[Theorems 4.2 and 4.5]{MeSc12}, there exists some $M>0$ such that for all $T\in (0,T_0]$ and $R\in (0,R_0]$
 and every $\hat{z}\in \Sigma(T,R)$, it holds that
\begin{equation}
\label{unif bdd Sigma}
\|{\rm tr}_{\partial\Omega} \hat{z} \|_{ \bF_\mu (J_T)} , \|{\rm tr}_{\partial\Omega}\nabla\hat{z} \|_{ \bF_\mu (J_T)} ,
 \|\hat{z} \|_{\bE_{1,\mu}(J_T)} , \|\hat{z} \|_{\bB_\mu(J_T)} \leq M.  
\end{equation}
Given any $\hat{z}\in \Sigma(T,R)$, we consider the linear problem 
\begin{equation}
\label{fixed point pb-z*}
\left\{
\begin{aligned}
      \pa_t z+\sA_* (t)  z  &=  \cA'(z_*) \hat{z}   - \cA(\hat{z})  +\sF(\hat{z}) - \sF'(z_*) \hat{z}  &&\text{in}&&\Omega, \\
      \sB_*(t)  z     &= \cB'(z_*) \hat{z}  - \cB(\hat{z})  &&\text{on}&&\pa\Omega,\\
      z(0)&=z_0  &&\text{in}&&\Omega.
    \end{aligned}
    \right.
\end{equation}
As $\cB(z_0)=\sB(z_0)z_0=0$, the compatibility condition 
$$
\sB_*(0)  z_0= \cB'(z_0) z_0   - \cB(z_0)  
$$
is satisfied and we can infer from Proposition~\ref{Prop: linearized abstract pb} that \eqref{fixed point pb-z*} has a unique solution $z=  \cT(\hat{z}) ~\in~\bE_{1,\mu}(J_T)$. \\
 Then it is clear that $z\in \Sigma(T,R)$ solves \eqref{nonlinear abstract equation}  iff it is a fixed point of $\cT$ in $ \Sigma(T,R)$.
Note that $v=\cT(\hat{z})-w$ solves
\begin{equation*}
\left\{
\begin{aligned}
      \pa_t v+\sA_* (t)  v  &=  \bF_*( \hat{z}) &&\text{in}&&\Omega, \\
      \sB_*(t) v     &= \bG_*( \hat{z} ) &&\text{on}&&\pa\Omega,\\
      v(0)&=0  &&\text{in}&&\Omega,
    \end{aligned}
    \right.
\end{equation*}
where
\begin{equation*}
\begin{aligned}
\bF_*(\hat{z}) &=   -\big(\cA(\hat{z})  - \cA(z_*)  -  \cA'(z_*)(\hat z-z_*) \big)   + \sF(\hat{z}) - \sF(z_*) - \sF^\prime(z_*)(\hat{z} - z_*),\\
\bG_*(\hat{z}) &=   -\big(\cB(\hat z) - \cB(z_*)- \cB^\prime(z_*)(\hat z - z_*) \big). 
\end{aligned}
\end{equation*}
In view of Proposition~\ref{Prop: linearized abstract pb}, 
there exits a constant $C>0$, which is independent of $T\in (0,T_0]$, such that
\begin{equation*}
\begin{aligned}
\|\cT(\hat{z})-w\|_{\bE_{1,\mu}(J_T)}  &\leq  C ( \|\cA(\hat{z})  - \cA(z_*)  -  \cA'(z_*)(\hat{z}-z_*)   \|_{\bE_{0,\mu}(J_T) } \\
& \quad + \|\sF(\hat{z}) - \sF (z_*)   -\sF'(z_*) (\hat{z}- z_*)\|_{\bE_{0,\mu}(J_T) } \\
& \quad + \|\cB(\hat{z})  - \cB(z_*)  -    \cB'(z_*)(\hat{z}-z_*) \|_{\bF_\mu (J_T) }  ),
\end{aligned}
\end{equation*}
where we have used the fact that 
$$
\cB(\hat{z})  - \cB(z_*)  -    \cB'(z_*)(\hat{z}-z_*)  \in \zbF_\mu(J_T).
$$ 
Using    \eqref{unif bdd Sigma}   and  \eqref{Frechet der est}, one verifies that 
$\|\cT(\hat{z})-w\|_{\bE_{1,\mu}(J_T)} \le R$, provided $T$ and $R$ are chosen small enough. 
This shows that $\cT$ maps $\Sigma(T,R)$ into itself.

\medskip\noindent
To show that $\cT: \Sigma(T,R) \to \Sigma(T,R)$ is a strict contraction,
 we pick functions  $\hat{z}, \bar{z}\in \Sigma(T,R)$. Then we obtain
\begin{align*}
  \| \cT(\hat{z})- \cT(\bar{z})\|_{\bE_{1,\mu}(J_T)}   
& \leq  C ( \| \cA(\hat{z})  -\cA(\bar{z}) -\cA'(z_*) (\hat{z} - \bar{z})    \|_{\bE_{0,\mu}(J_T)}  \\
& \quad + \|\sF(\hat{z}) - \sF (\bar{z})  -\sF'(z_*) (\hat{z}- \bar{z})\|_{\bE_{0,\mu}(J_T)}  \\
& \quad + \| \cB(\hat{z}) - \cB(\bar{z}) - \cB'(z_*) (\hat{z} - \bar{z}) \|_{\bF_\mu(J_T)}  )
\end{align*} 
for some constant $C$ that is independent of $T\in (0,T_0]$.
Employing   \eqref{unif bdd Sigma} and \eqref{Frechet der est}, \eqref{Frechet der est 2},  one verifies that 
\begin{equation*}
  \| \cT(\hat{z})- \cT(\bar{z})\|_{\bE_{1,\mu}(J_T)}   \leq \frac{1}{2} \| \bar{z} - \hat{z} \|_{\bE_{1,\mu}(J_T)},
\end{equation*} 
provided $T$ and $R$ are chosen sufficiently small.

The contraction mapping principle implies the existence of a 
 unique solution $z\in \Sigma(T,R)$ to \eqref{nonlinear abstract equation} on the time interval $[0,T]$.
A standard argument then yields that $z$ is also the unique solution in  $\bE_{1,\mu}(J_T)$.

\medskip\noindent
The existence of a maximal interval of existence  $[0,T_+(z_0))$ can be obtained in a standard way as in \cite[Corollary~5.1.2]{PrSi16}. 

\medskip\noindent
(b)
Pick an arbitrary $T\in (0,T_+(z_0))$  and let $z=z(\cdot,z_0)$ be the (unique) solution of \eqref{nonlinear abstract equation} obtained in part (a).
Then 
 $w\in \bE_{1,\mu}(J_T)$ is  a solution of \eqref{nonlinear abstract equation} with initial value $w_0\in \cM_\mu$
 iff  $w=z+v$, where $v \in \bE_{1,\mu}(J_T)$ solves the system 
\begin{equation}
\label{continuous dependence pb}
\left\{
\begin{aligned}
      \pa_t v+\sA_0 (t)  v  &= \bF(v(t)) &&\text{in}&&\Omega, \\
      \sB_0(t)  v     &= \bG(v(t) )&&\text{on}&&\pa\Omega,\\
      v(0)&=v_0=w_0-z_0 &&\text{in}&&\Omega 
    \end{aligned}
    \right.
\end{equation}
on $[0,T]$, with
\begin{align*}
\sA_0 (t)  v  &= \cA'(z (t)) v - \sF^\prime(z (t))  v = \sA(z (t)) v + [\sA^\prime(z (t))v] z (t)- \sF^\prime(z (t))   v , \\
\sB_0 (t) v   &= \cB'(z (t))v = \sB(z (t))v + [\sB^\prime (z (t))v] z (t), \\
\bF(v(t))      &=    -\big(\cA(z(t)+v(t))-\cA(z(t))-\cA^\prime(z(t))v(t) \big)  
                    +\sF(z(t)+v(t)) - \sF(z(t) ) - \sF^\prime(z (t))   v(t),\\
\bG(v(t)) &=   -\big( \cB(z(t)+v(t)) -\cB(z(t)) -\cB'(z(t)) v(t) \big)  .\\
\end{align*}
It follows from Proposition~\ref{Prop: mapping properties}  that
\begin{equation}
\label{reg bF and bG}
\bF \in C^1(\bE_{1,\mu}(J_T), \bE_{0,\mu}(J_T)) \quad \text{and}\quad  \bG \in C^1(\bE_{1,\mu}(J_T), \bF_\mu(J_T)) .
\end{equation}
Easy computations show that
\begin{equation}
\label{zeros bF and bG}
\bF (0)=0, \quad \quad \bG (0)=0 , \quad
\bF'(0)=0, \quad  \quad \bG'(0)=0.
\end{equation}
Note that the compatibility condition 
\begin{equation}
\label{diff compability}
\sB_0(0)v(0):= \cB^\prime(z_0)v(0)=\bG(v(0))
\end{equation}
is satisfied, as $\cB(z_0)=\cB(w_0)=0$ by assumption.  Let
$$
X_{\gamma,\mu}^0:=\{\hat{z}_0\in X_{\gamma,\mu}:   \cB^\prime(z_0) \hat{z}_0=0 \}.
$$
We then introduce the map $\cF: X_{\gamma,\mu}^0 \times \bE_{1,\mu}(J_T) \to \bE_{1,\mu}(J_T)$, defined by
$$
\cF(\hat{z}_0, \hat{v})= \hat{v}- \cS(\bF(\hat{v}), \bG(\hat{v}), \hat{z}_0+\cR(z_0)   \bG(\hat{v}(0))),
$$
where $\cS$ is the solution operator defined in Proposition~\ref{Prop: linearized abstract pb} and $\cR(z_0): Y_{\gamma,\mu} \to X_{\gamma,\mu} $ is the bounded right  inverse of   $\cB^\prime(z_0) $  asserted by Lemma~\ref{lem:right inverse}.
Observe that   $	\cF(0, 0)=0$ and the compatibility condition
$$
  \cB^\prime(z_0)   \left( \hat{z}_0+\cR(z_0)   \bG(\hat{v}(0)) \right)= \bG(\hat{v}(0)) 
$$
is satisfied for each $v\in  \bE_{1,\mu}(J_T)$.
It follows from  \eqref{reg bF and bG} that
$$
\cF\in C^1( X_{\gamma,\mu}^0 \times \bE_{1,\mu}(J_T) , \bE_{1,\mu}(J_T))  
$$
and from \eqref{zeros bF and bG}   that $\partial_2 \cF(0,0) \in \cL is(\bE_{1,\mu}(J_T))$.
The implicit function theorem then implies that there exist some $r>0$ and $\Psi\in C^1( B_{X_{\gamma,\mu}^0}(0,r), \bE_{1,\mu}(J_T))$ 
such that  $\hat{v}=\Psi(\hat{z}_0)$ with $\hat{z}_0 \in B_{X_{\gamma,\mu}^0}(0,r)$  iff
$ 
\cF(\hat{z}_0, \hat{v})=  0,  
$ 
or equivalently, $\hat{v} $    solves
\begin{equation*}
\left\{
\begin{aligned}
      \pa_t \hat{v}+\sA_0 (t)  \hat{v}  &= \bF(\hat{v}(t)) &&\text{in}&&\Omega, \\
      \sB_0(t)  \hat{v}     &= \bG(\hat{v}(t) )&&\text{on}&&\pa\Omega,\\
      \hat{v}(0)&= \hat{z}_0+\cR(z_0)  \bG(\hat{v}(0)) &&\text{in}&&\Omega. 
    \end{aligned}
    \right.
\end{equation*} 
We define $\cP(z_0): X_{\gamma,\mu} \to X_{\gamma,\mu}^0$ by
$\cP(z_0)\tilde z=(I-\cR(z_0)   \cB^\prime(z_0) ) \tilde z$.
For sufficiently small $\rho>0$, and $w_0 \in \cM_\mu \cap \in B_{X_{\gamma,\mu}}(z_0,\rho)$, we choose
$$
\hat{z}_0= \cP(z_0)  v_0 \in B_{X_{\gamma,\mu}^0}(0,r), \quad \text{where } v_0=w_0-z_0.
$$
In view of \eqref{diff compability}, it holds that
$$
 v_0= \hat{z}_0 +  \cR(z_0)  \cB^\prime(z_0)   v_0 = \hat{z}_0 +  \cR(z_0) \bG(v(0)).
$$
Therefore, $v$ and $\hat v$ solve the same system of equations.
We conclude that $ v:=\Psi(\cP(z_0)(w_0-z_0))$
is the unique solution of \eqref{continuous dependence pb} on $[0,T]$ with initial value $v_0$.
Hence,  $w=z+ \Psi(\cP(z_0) (w_0-z_0)) $ is the (unique) solution to \eqref{nonlinear abstract equation} with initial value  $w_0$ on $[0,T]$.
Setting $z(\cdot,w_0) =z(\cdot,z_0) + \Psi(\cP(z_0) (w_0-z_0)) $ 
we can infer that the mapping
\begin{equation*}
[w_0\mapsto z(\cdot,w_0)]:  \cM_\mu \cap B_{X_{\gamma,\mu}}(z_0, \rho)\to \bE_{1,\mu}(J_T) 
\end{equation*}
is Lipschitz continuous.

 \medskip\noindent
(c) 
Fix $T\in (0,T_+(z_0))$ and $\epsilon\in (0,1)$ so small  that $(1+\epsilon)T<T_+(z_0)$.
Let $z\in \bE_{1,\mu}(J_T)$ be the unique solution of \eqref{nonlinear abstract equation} with initial value $z_0$.
Let $z_\lambda(t)=z(\lambda t)$.
Then $v=z_\lambda$ solves
\begin{equation}
\label{time regularity pb}
\left\{
\begin{aligned}
      \pa_t v+\sA_0 (t)  v  &= \fF(\lambda, v(t)) &&\text{in}&&\Omega, \\
      \sB_0(t)  v     &=  \fG( v(t) )&&\text{on}&&\pa\Omega,\\
      v(0)&= z_0 &&\text{in}&&\Omega 
    \end{aligned}
    \right.
\end{equation}
on $[0,T]$, where $\sA_0$ and $\sB_0$ are defined as in \eqref{continuous dependence pb} and
\begin{equation*}
\begin{aligned}
\fF(\lambda, v(t)) &=  -\big( \lambda \cA(v(t)))-\cA^\prime(z(t))v(t)\big) + \lambda \sF(v(t)) - \sF^\prime(z(t))v(t),  \\
\fG(v(t))  & =    -\big(\cB(v(t)) - \cB^\prime(z(t) v(t)\big),  \\
\end{aligned}
\end{equation*}
with $(\cA, \cB)$ defined in \eqref{AB}.
As $\fF(1,v)=-\big(\cA(v)-\cA^\prime(z)v\big) + \sF(v)-\sF^\prime(z)v$, 
one readily verifies that
$$
\partial_2 \fF (1,z)=0, \quad \fG'(z)=0.
$$
Similar to part (b), we define $\cF_0: (1-\epsilon,1+\epsilon) \times \bE_{1,\mu}(J_T) \to \bE_{1,\mu}(J_T)$   by
$$
\cF_0(\lambda, v)= v - \cS(\fF( \lambda, v ), \fG(v), \hat{z}_0+\cR(z_0)   \fG(v(0))),
$$
where $ \hat{z}_0 = z_0- \cR(z_0)   \cB^\prime(z_0)z_0   \in X_{\gamma,\mu}^0$. 
Note that the compatibility condition 
$$ \cB^\prime(z_0) (\hat{z}_0+  \cR(z_0)   \fG(v(0)) )= \fG(v(0))$$ is again satisfied.
We have
$\cF_0(1,z)=0$ and $\partial_2 \cF_0(1,z)\in \cL is(\bE_{1,\mu}(J_T))$.
Therefore, the implicit function theorem implies that there exist $\delta \in (0,\epsilon) $ and $\Psi_0\in C^1( (1-\delta, 1+\delta), \bE_{1,\mu}(J_T))$ 
such that $v= \Psi_0(\lambda)$ iff $v$ solves
\begin{equation*}
\left\{
\begin{aligned}
      \pa_t v+\sA_0 (t)  v  &= \fF(\lambda, v(t)) &&\text{in}&&\Omega, \\
      \sB_0(t)  v     &=  \fG( v(t) )&&\text{on}&&\pa\Omega,\\
      v(0)&= z_0- \cR(z_0)   \cB^\prime(z_0)z_0  + \cR(z_0) \fG( v(0) ) &&\text{in}&&\Omega .
    \end{aligned}
    \right.
\end{equation*}
From $\cF_0(1,z)=0$, we infer that $z=\Psi_0(1)$. We want to show that $\Psi_0(\lambda)=z_\lambda$.
To this end, notice that $z_0(\lambda):=\gamma_0 \Psi_0(\lambda)$ satisfies
\begin{equation*}
\begin{aligned}
z_0(\lambda) &=  z_0- \cR(z_0)   \cB^\prime(z_0)z_0  + \cR(z_0) \fG( z_0(\lambda) ) 
 =z_0-\cR(z_0)\big(\cB(z_0(\lambda) - \cB^\prime(z_0)(z_0(\lambda)-z_0)\big).
\end{aligned}
\end{equation*}
By using the fact that $\cB(z_0)=\sB(z_0)z_0=0$, we further obtain
\begin{equation*}
\begin{split}
z_0(\lambda) -   z_0 
 =- \cR(z_0) \big(\cB(z_0(\lambda))-\cB(z_0) - \cB^\prime (z_0)(z_0(\lambda)-z_0))\big).
\end{split}
\end{equation*}
We can thus conclude  that
\begin{align*}
\| z_0(\lambda) -   z_0 \|_{X_{\gamma,\mu}}  
&\leq  \Phi(\| z_0(\lambda) -   z_0 \|_{X_{\gamma,\mu}})\| z_0(\lambda) -   z_0 \|_{X_{\gamma,\mu}}\\
 &\leq  \Phi(\| \Psi_0(\lambda) -  \Psi_0(1) \|_{\bE_{1,\mu}(J_T) } )\| z_0(\lambda) -   z_0 \|_{X_{\gamma,\mu}}.
\end{align*}
By choosing $\delta $ so small that 
$$\sup\limits_{\lambda\in (1-\delta, 1+\delta) } \Phi(\| \Psi_0(\lambda) -  \Psi_0(1) \|_{ \bE_{1,\mu}(J_T) } ) \leq  1/2,$$ 
we have $z_0(\lambda)=z_0$ for all $\lambda\in (1-\delta, 1+\delta) $. 
Thus, $\Psi_0(\lambda)$ solves \eqref{time regularity pb} and hence $\Psi_0(\lambda)=z_\lambda$. The differentiability of $\Psi_0$ implies that
$ \Psi_0'(1)= t \pa_t z  \in \bE_{1,\mu}(J_T)$. 
Therefore, $\pa_t (t z) =z + t \pa_t z \in \bE_{1,\mu}(J_T)$.
Then the asserted the regularity of $z$ follows.
\end{proof} 
  
\goodbreak
In the following we will discuss the remaining issues concerning the pressure function $\pi$ and the constraint $|m|=1$. 
\begin{proposition}\label{Prop:equivalence of 2 sys}
Given $T>0$,  the following statements are equivalent:
 \begin{enumerate}
   \item[{\rm (a)}] \eqref{eqn:Rewrit1} has a solution $(u, F, \theta, m, \pi)\in \bE_{1, \mu}(J_T)\times L_{p, \mu}( J_T; \dot{H}_p^1(\Omega))$. 
   \vspace{1mm}
   \item[{\rm (b)}] \eqref{nonlinear abstract equation} has a solution $(u, F, \theta, m)\in \bE_{1, \mu}(J_T)$.
 \end{enumerate}
\end{proposition}
\begin{proof}
 The implication (a)$\Rightarrow$(b) follows by just applying $P_H$ to both sides of the $u$ equation in \eqref{eqn:Rewrit1}. 
  We are left to show (b)$\Rightarrow$(a). Suppose $z=(u, F, \theta, m)\in \bE_{1, \mu}(J_T)$ solves \eqref{nonlinear abstract equation}.
   Let
  $$v=-u\cdot \nabla u+\nabla\cdot (\upmu(\theta)\nabla u)-\nabla\cdot (\nabla  m\odot \nabla m)+\nabla\cdot (F F^\top).$$
  Let $\pi$ be an $W^1_p$ solution to $\Delta \pi=\nabla\cdot v$ in $\Omega$ with $\pa_\nu \pi=v\cdot\nu$ on $\pa \Omega$, i.e., 
  $$(\nabla\pi(t)| \nabla\phi)=( v(t)| \nabla\phi), \qquad \forall\phi\in \dot{H}_{p'}^1(\Omega), \quad p'=p/(p-1) .$$
  From standard elliptic theory we conclude that $\pi\in L_{p, \mu}( J_T;\dot{H}_p^1(\Omega))$ and $P_H v(t)=v(t)-\nabla \pi(t)$. 
  Then by \eqref{nonlinear abstract equation} we have 
  $$\pa_t u+\nabla \pi-v=\pa_t u-P_H v=0.$$
  Hence  $(u, F, \theta, m, \pi)$ solves \eqref{eqn:Rewrit1}$_1$. 
\end{proof}
 Now we are in a position to state the main theorem concerning local well-posedness of \eqref{eqn:Rewrit1}. To this end, we define the state manifold of \eqref{eqn:Rewrit1} by
\begin{equation*}
 \cSM_\mu :=\{   z= (u, F, \theta, m) \in X_{\gamma,\mu}: \;  \theta>0, \; |m|=1,\; \sB(z)z=0   \} .
\end{equation*}
\goodbreak
\begin{theorem}[Local well-posedness of \eqref{eqn:Rewrit1}]
  \label{thm:MainLocal}
  Assume \eqref{assumption}  and \eqref{indices cond}.
  \begin{itemize}
  \item[{\rm (a)}]
  Suppose that $z_0=(u_0, F_0, \theta_0, m_0)\in \cM_\mu$.
   Then there exists a number $T>0$ such that  \eqref{eqn:Rewrit1} has a unique solution 
  $$\widetilde{z}(\cdot,z_0)= (u, F, \theta, m, \pi)\in \bE_{1, \mu}(J_T)\times L_{p, \mu}( J_T; \dot{H}_q^1(\Omega )).$$
  Each solution can be extended to a maximal existence interval $[0, T_+(z_0))$.
  
  \smallskip\noindent
   If, in addition, $|m_0|\equiv 1$, then the solution also satisfies
  $$|m(t)|\equiv 1, \quad t\in[0, T_+(z_0)).$$ 
Moreover, it holds that $$\theta(t, x)\ge \min_{\overline{\Omega}}\theta_0(x), \quad(t, x)\in [0, T_+(z_0))\times \bar{\Omega}.$$  
 \item[{\rm (b)}] Let $T<T_+(z_0)$. \\
  Then there exists a number $\rho>0$ such that
for every $w_0\in   \cSM_\mu\cap B_{X_{\gamma,\mu}}(z_0, \rho) ,$ the unique solution $\widetilde{z}(\cdot, w_0 )$  of ~\eqref{eqn:Rewrit1}  with initial condition $w_0$ belongs to $\bE_{1,\mu}(J_T) \times L_{p, \mu}( J_T; \dot{H}_q^1(\Omega ))$.
Moreover,
the mapping
\begin{equation*}
[w_0\mapsto \widetilde{z}(\cdot, w_0 )]:  \cSM_\mu \cap B_{X_{\gamma,\mu}}(z_0, \rho)\to \bE_{1,\mu}(J_T) \times L_{p, \mu}( J_T; \dot{H}_q^1(\Omega )) 
\end{equation*}
is (Lipschitz) continuous. \\
Hence, the system \eqref{eqn:Rewrit1} generates a (Lipschitz) continuous semiflow on $\cSM_\mu$.
  \end{itemize}
   \end{theorem}
\medskip\noindent
\emph{Proof.} 
(a)
The fact that $|m(t)|=1$ up to $T_+(z_0)$, provided $|m_0|=1$, follows from a parabolic maximum principle (c.f. \cite[Theorem 2.5]{DSS23}). 
By \eqref{embedding-trace-space}, $\theta_0\in C^1(\overline\Omega)$, so that   $\min_{\bar{\Omega}}\theta_0(x)$ exists.
For $(t,x)\in [0,T_+(z_0))\times \Omega$, let $\psi(t, x)=\min_{\bar{\Omega}}\theta_0(x)-\theta(t, x)$. Then we can derive from \eqref{eqn:Rewrit1} that $\psi$ solves
  \begin{equation}
    \begin{aligned}
      \pa_t \psi+u\cdot \nabla \psi&\le \nabla\cdot (K(z)\nabla \psi)&&\text{in}&&\Omega, \\
      \nu \cdot {\rm tr}_{\partial\Omega} (K(z)\nabla \psi) &=0&&\text{on}&&\pa\Omega, \\
      \psi(0, x) &\le 0&&\text{in}&&\Omega.
    \end{aligned}
    \label{eqn:psieq}
  \end{equation}
  Multiplying both sides of \eqref{eqn:psieq}$_1$ by $\psi_+=\max\{\psi, 0\}$ and integrating over $\Omega$ we can show that 
  \begin{equation}
\pa_t\left( \frac{\|\psi_+\|_2^2}{2} \right)+c\|\nabla \psi_+\|_2^2\le \pa_t\left( \frac{\|\psi_+\|_2^2}{2} \right)+
 (K(z)\nabla \psi_+|\nabla \psi_+)_\Omega   \le 0.\label{eqn:intineq}
  \end{equation}
In fact, we can compute
\begin{align*}
  (\pa_t\psi|\psi_+)_\Omega&=(\pa_t\psi_+|\psi_+)_\Omega=\pa_t \frac{\|\psi_+\|_2^2}{2}, \\
  (u\cdot \nabla\psi|\psi_+)_\Omega&=(u\cdot \nabla\psi_+|\psi_+)_\Omega=\int_\Omega u\cdot\nabla \frac{| \psi_+|^2}{2}\, dx=0, \\
  (\nabla\cdot (K(z)\nabla \psi)|\psi_+)_\Omega&=-\int_{\Omega}(K(z)\nabla \psi)\cdot \nabla \psi_+ \, dx=-(K(z)\nabla \psi_+|\nabla \psi_+)_\Omega.
\end{align*}
Integrating \eqref{eqn:intineq} with respect to $t$, and using the fact that $\psi_+(0, x)=0$ we conclude 
that $\|\psi_+(t)\|_{2}^2=0$ for $t\in[0, T_+(z_0))$. 
Hence, $\psi_+(t, x)=0$ for $(t, x)\in [0, T_+(z_0))\times\overline{\Omega}$. 

\medskip\noindent
(b) This part follows directly from Theorem~\ref{Thm:wellposed abstract}(b), the proof of Proposition~\ref{Prop:equivalence of 2 sys} and the fact that
\begin{align*}
[z\mapsto  & (-u\cdot \nabla u+\nabla\cdot (\upmu(\theta)\nabla u)-\nabla\cdot (\nabla  m\odot \nabla m)+\nabla\cdot (F F^\top))]\\
& \in C^1(\bE_{1,\mu}(J_T) , L_{p,\mu}(J_T; L_p(\Omega; \bR^3))). \qquad\qquad &&\square
\end{align*} 
 \goodbreak


\section{Stability and Long-time behavior}\label{section:stability}

In this section,  we will study global existence and stability of solutions to \eqref{magneto sys}.
The next theorem is to establish the long-time behavior of solutions and the stability with respect to constant equilibria. 
\begin{theorem}\label{thm:global existence}
Assume \eqref{assumption}, \eqref{indices cond}, $|m_0|=1$,  and the positivity condition $ \theta_0>0$. 
Let $z=z(\cdot, z_0)$ be the solution of \eqref{magneto sys}, defined on its maximal interval of existence $[0, T_+(z_0)).$
Then the following properties hold.
\begin{enumerate}
\item[{\rm (a)}]  
We have the following alternatives:
\vspace{2mm}
\begin{enumerate}
\item[{\rm (i)}] $T_+(z_0)=\infty$, that is, $z$ is a global solution;
\vspace{2mm}
\item[{\rm (ii)}] $\displaystyle \lim_{t\to T_+(z_0)}z(t) $ does not exist in $X_{\gamma,\mu}$.
\end{enumerate}
\vspace{1mm}
\item[{\rm (b)}]
Suppose 
$$
\displaystyle \sup_{t\in [\delta, T_+(z_0))} \|z(t)\|_{X_{\gamma, \overline{\mu}}}<\infty
\quad
\text{for some $\delta\in (0, T_+(z_0))$ and some $\bar{\mu}\in(\mu, 1]$}.
$$
Then $z$ exists globally and $\dist(z(t), \cE)\to 0$ in $X_{\gamma, 1}$ as $t\to \infty$. 
\end{enumerate}
\end{theorem}
\begin{proof}
(a) 
We will prove the assertion by following the strategy in \cite[Corollary~5.1.2]{PrSi16}.
Assume that $T_+(z_0)<\infty$ and $z(\cdot,z_0)$ converges to some $z_1$ in $X_{\gamma,\mu}$ as $t\to T_+(z_0)$.
Lemma~\ref{lem: Nemyskii} implies that $\sB(z_1)z_1=0$. Combining Theorem \ref{thm:MainLocal}(a) and the assumption, 
we have that there exists an $\eta>0$ depending on $\min_{\bar{\Omega}}\theta_0$ such that 
	\begin{equation}\label{lower bdd}
		{\rm dist}_{X_{\gamma, \mu}}(z(t), \pa V_\mu)\ge \eta, \quad \text{ for all }t\in[0, T_+(z_0)), 
	\end{equation} 
where 	$V_\mu=\{z=(u, F, \theta, m)\in X_{\gamma, \mu}:\theta>0\}$. We thus infer that $z_1\in \cSM_\mu $. 
Then the orbit $\cV:=\{z(t) : 0\le t <  T_+(z_0) \}$ is relatively compact in   $\cSM_\mu$. 
It   follows from Theorem~\ref{thm:MainLocal}(b) and a compactness argument that there exists $T_0>0$ such that  for each $s \in [0 , T_+(z_0))$, system~\eqref{eqn:Rewrit1} with initial value $z(s)$ has a unique solution in  $\bE_{1,\mu}(J_{T_0})$. 
Fixing $s_0\in (T_+(z_0)-T_0, T_+(z_0))$,  system~\eqref{eqn:Rewrit1} with initial value $z(s_0)$ has a solution 
$v \in \bE_{1,1}(J_{T_0})$, which, by uniqueness, coincides with $z(s_0+ \cdot\, , z_0)$ on $[s_0, T_+(z_0))$.  
In view of Proposition~\ref{Prop:equivalence of 2 sys}, the solution  $\widetilde{z}(\cdot,z_0)$ of \eqref{magneto sys} can be extended beyond $T_+(z_0)$,
a contradiction.

\medskip\noindent
(b)
We will prove the assertion by following the strategy in \cite[Section 5.7]{PrSi16}.
By Theorem~\ref{thm:MainLocal}(b), the system \eqref{eqn:Rewrit1} defines a local semiflow on $\cSM_\mu  $.
From the assumption   and the compact embedding
$$
X_{\gamma, \overline{\mu}} \hookrightarrow X_{\gamma, \mu},
$$
we infer that the orbit $\cV:=\{z(t) : 0\le t < T^+(z_0)\}$ is relatively compact in $\cSM_\mu $.
Denote the closure of $\cV$ in $X_{\gamma,\mu}$ by $\overline{\cV}$.
It follows from a similar argument as in part (a) that
there exist a number $T_0>0$ and an open neighborhood  $\cU$ of $\overline{\cV}$ in $\cM_\mu $ such that for every $\tilde{z}_0\in \cU$, \eqref{nonlinear abstract equation} admits a unique solution $\tilde{z}\in \bE_{1,\mu}(J_{T_0})$. 
Moreover, the solution map $G_1:\cU\to \bE_{1,\mu}(J_{T_0})$ is continuous.
This implies  that, for any $t\in [0, T_+(z_0))$, the solution of \eqref{nonlinear abstract equation} with initial condition $z(t)$ exists on the  interval $[t,t+T_0]$, which further shows that $T_+(z_0)=\infty$.
Now it follows from Proposition~\ref{Prop:equivalence of 2 sys} that the solution to \eqref{magneto sys} is global.

As above, one sees that \eqref{eqn:Rewrit1} also defines a local semiflow on $\cSM_1$, equipped with the metric induced by $X_{\gamma,1}$.
It follows from the inequality
\begin{align*}
\|z(T_0)\|_{X_{\gamma,1}} & \leq  \|z\|_{C([T_0/2,T_0]; X_{\gamma,1})} \leq C(T_0)
\|z\|_{\bE_{1,1}([T_0/2,T_0])} \\
 & \leq C(T_0) (T_0/2)^{ \mu -1} \|z\|_{\bE_{1,\mu}(J_{T_0})}   
\end{align*}
that the map $G_2: \bE_{1,\mu}(J_{T_0}) \to X_{\gamma,1}: \, z \mapsto z(T_0)$  is continuous. This implies that the composition map $G=G_2\circ G_1: \cU\to X_{\gamma,1}: \, z  \mapsto G_1(z)(T_0)$ is continuous.
We thus infer that the orbit $\{z(t)\}_{t\geq T_0}$ is relatively compact in $\cSM_1$ because the continuous image  of a relatively compact set is again relatively compact.
Recall that the definition of $\omega$-limit set of \eqref{nonlinear abstract equation} is given by 
$$
\omega(z_0):=\{w\in X_{\gamma, 1}:\exists t_n\to \infty \text{ s.t. }\|z(t_n)-w\|_{X_{\gamma, 1}}=0 \text{ as }n\to \infty\}. 
$$
By \cite[Theorem~17.2]{Ama90}, $\omega(z_0)$ is nonempty, compact, connected in $\cSM_1$ and 
\begin{equation}
\label{converg omega limit set}
\lim_{t\to \infty}{\rm dist}_{X_{\gamma, 1}}(z(t), \omega(z_0))=0 .
\end{equation}
Now following a similar computation as in \cite[Proposition 4.1]{DSS23} we can show that  $-\sN$ is a strict Lyapunov functional for \eqref{eqn:Rewrit1}. Therefore,  
$\omega(z_0)\subset \cE$. Combining with \eqref{converg omega limit set}, this implies  $$\lim_{t\to \infty}{\rm dist}_{X_{\gamma, 1}}(z(t), \cE)= 0.$$
\end{proof}

Our last result is about the qualitative behavior of   solutions near constant equilibria. Consider the set of constant equilibria of \eqref{nonlinear abstract equation}:
$$\cE_c:=\{0\}\times\{0_3\}\times \bR_+\times \bR^3$$
which is a subset of $\cE$, the set of equilibria. 
Let $z_*\in\cE_c$ be given. Then one readily verifies that
\begin{equation}
\label{zero-derivatives}
([\sA^\prime(z_*)z]z_*,  \sF'(z_*)z, [\sB^\prime(z_*)z]z_*)=(0,0,0), \quad z\in X_1.
\end{equation}
Therefore, the linearization of \eqref{nonlinear abstract equation} at $z_*\in \cE_c$ is given by 
 \begin{equation*}
  \begin{aligned}
    \sA_*& z:=\sA(z_*)z=
    \begin{bmatrix}
    	-\upmu(\theta_*)P_H \Delta u & 0 & 0 & 0 \\
    0 	& -\kappa(\theta_*)\Delta F& 0 & 0 \\
    0	& 0 & -K(z_*):\nabla^2 \theta & 0\\
    	0 & 0 & 0 & (\beta(\theta_*)\sM(m_*)-\alpha(\theta_*)I_3)\Delta m
    \end{bmatrix},
\\
\\
    \sB_* &z=  \nu\cdot {\rm tr}_{\partial\Omega} (K(z_*)\nabla \theta), \qquad z=(u,F,\theta, m).    
     \end{aligned}
  \label{}
\end{equation*}
Note that
$(\sA_*,\sB_*)\in\cL(X_1,  X_0 \times Y_1)$, where $Y_1= W^{1-1/p}_p(\partial\Omega).$
Hence, $A_0:=\sA_*|_{N(\sB_*)}$ is well-defined, where $N(\sB_*)$ is the null-space of $\sB_*$.

\medskip\noindent 
 The next result will be important for proving stability of constant equilibria.
\begin{proposition}
\label{normally-stable}
 Each constant equilibrium $z_*\in \cE_c$ is normally stable.
\end{proposition}
 
\begin{proof}
By definition of normal stability, we need to show that
\begin{enumerate}[label={\rm(\roman*)}]
\item near $z_*$, $\cE$ is a $C^1$-manifold in $X_1$ of finite dimension,
\item the tangent space of $\cE$ at $z_*$ is isomorphic to $N( A_0)$,
\item $0$ is a semi-simple eigenvalue of $ A_0$; i.e. $X_0 =N(A_0)\oplus R(A_0)$,
\item $\sigma ( A_0)\setminus \{0\}\subset \{z\in \bC: {\rm Re}z>0\}$.
\end{enumerate}
We immediately see that (i) is satisfied, as $\cE_c$ is a linear space of dimension $4$. 

\medskip\noindent
Suppose $z=(u, F, \theta, m)$ is an eigenvector of $A_0$ subject to an eigenvalue $\lambda\in \bC$, i.e., $A_0z=\lambda z$. In other words, $\sB_* z=0$ on $\pa \Omega$, and $\sA_* z=\lambda z$ in $\Omega$. Taking the inner product of the later identity with $\bar{z}$ and using integration by parts we can derive
  \begin{align*}
    {\rm Re }\,\lambda \|z\|_2^2&=\upmu (\theta_*)\|\nabla u\|_2^2+\kappa(\theta_*)\|\nabla F\|_2^2
    + {\rm Re }\, (K(z_*)\nabla \theta|\nabla \bar{\theta})_\Omega +\alpha(\theta_*)\|\nabla m\|_2^2\\
    &\ge \underline{\upmu}\|\nabla u\|_2^2+\underline{\kappa}\|\nabla F\|_2^2+c\|\nabla \theta\|_2^2+\underline{\alpha}\|\nabla m\|_2^2, 
  \end{align*}
where we use assumption~\eqref{assumption} and the fact that ${\rm Re}\,  (\sM(m_*)\Delta m|\bar{m})_\Omega =0$ (see \cite[Section 3]{DSS23}).
  Hence ${\rm Re}\, \lambda\ge0$. Furthermore, when ${\rm Re}\, \lambda=0$, we get that $z\in \{0\}\times\{0_3\}\times\bR\times\bR^3$, thus  $\sigma(A_0)\cap i\bR=\{0\}$ and $N(A_0)= \{0\}\times\{0_3\}\times\bR\times\bR^3$. 
 This shows that (ii) and (iiii) hold true.
 
\medskip\noindent
 Finally, we show that $\{0\}$ is a semi-simple.  Since $A_0$ has compact resolvent, it suffices to show that $N(A_0)=N(A_0^2)$. Since $N(A_0)\subset N(A_0^2)$, we just need to show $N(A_0^2)\subset N(A_0)$. For $w=(v, J, \vartheta, n )\in N(A_0^2)$, let $z=(0, 0, \theta, m)\in N(A_0)$ such that $A_0w=z$. Then we can compute
  \begin{align*}
    \|z\|_2^2&=(\sA_0 w|z)_\Omega
    =(-K(z_*):\nabla^2\vartheta|\theta)_\Omega+( (\beta(\theta_*)\sM(m_*)-\alpha(\theta_*)I_3)\Delta n|m)_\Omega=0, 
  \end{align*}
where we use the fact that $\theta, m$ are constants in $N(A_0)$ and $\sB_* w=0$ on $\pa\Omega$. Hence, $A_0 w=z=0$ and $w\in N(A_0)$. This yields that $\{0\}$ is semi-simple. 

 Finally, it follows from \cite[Remark 2.2]{PSZ09}, that all equilibria near $z_*$ are contained in a $C^1$ manifold of dimension 4. 
\end{proof}

By adapting the proof of {\em the generalized  principle of linearized stability} provided in  \cite[Section 5.3]{PrSi16}, we can obtain the following stability property of $\cE_c$. 

\begin{theorem}
\label{thm:stability}
Assume  \eqref{assumption}   and \eqref{indices cond}.
Then each equilibrium $z_*\in \cE_c$ is stable in $X_{\gamma, \mu}$. Moreover, there exists $\delta>0$ such that if $\|z_0-z_*\|_{X_{\gamma, \mu}}\le \delta$, then the solution $z$ of \eqref{eqn:Rewrit1} with initial value $z_0$ exists globally and converges to some $z_\infty \in \cE_c$ at an exponential rate in $X_{\gamma, 1}$. 
\end{theorem}
\begin{proof}
It will be convenient to center \eqref{nonlinear abstract equation} 
around $z_*$, by setting $\bar{z}=z-z_*$. 
Then \eqref{nonlinear abstract equation}  can be rewritten as
\begin{equation}
   \label{abstract eq neaar equlibrium}
  \left\{ 
   \begin{aligned}
      \pa_t \bar{z}+\sA_* \bar{z}&=G(\bar{z})&&\text{in}&&\Omega, \\
      \sB_*  \bar{z} &= H(\bar{z}) &&\text{on}&&\pa\Omega, \\
      \bar{z}(0)&=\bar{z}_0=z_0 -z_* &&\text{in}&&\Omega,
    \end{aligned}
    \right.
  \end{equation}
where
\begin{equation*}
\begin{aligned}
G(\bar{z})
& = -\big(\sA(z_*+\bar z)(z_*+\bar z) - \sA(z_*)\bar z\big) + \sF(z_*+\bar z) \\
& = -\big((\sA(z_*+\bar z)-\sA(z_*))(z_*+\bar z) - [\sA^\prime(z_*)\bar z]z_*\big) + \sF(z_*+\bar z)- \sF(z_*) - \sF'(z_*)\bar z \\
H(\bar{z})
& = -\big(\sB(z_* + \bar z)(z_* +\bar z) - \sB(z_*) \bar z\big) \\
& = -\big((\sB(z_*+\bar z)-\sB(z_*))(z_* +\bar z)- [\sB^\prime(z_*)z]z_*\big).
\end{aligned}
\end{equation*}
Here we used \eqref{zero-derivatives} and the relations $(\sA(z_*)z_*, \sF(z_*), \sB(z_*)z_*)=(0,0,0)$ 
for the second line in the expressions of $G(\bar z)$ and $H(\bar z)$, respectively.

Theorem~\ref{thm:MainLocal} shows that \eqref{abstract eq neaar equlibrium} has a unique solution $\bar{z}$ on some maximal  interval of existence $[0,T_+)$. 

In the following, we use the notation 
$$X^c_0=N(A_0)=\{0\}\times\{0_3\}\times \bR_+\times \bR^3, \quad X_0^s = R(A_0).$$
We know from Proposition~\ref{normally-stable} 
that $ X_0= X_0^c\oplus X_0^s.$
Let $\sP^c$ be the projection from $X_0$ onto $X_0^c$, and $\sP^s$ the projection onto $X_0^s$.
Then we set $X_j^s=\sP^s X_j$, $j\in \{0,1, (\gamma,\mu)\}$. 
We point out  that $X_j^s= X_j\cap X_0^s$ and $\sP^c X_j\doteq X^c$. 
Therefore, in the sequel, we will simply be using $X^c$, equipped with the norm induced by $X_0$.
As $X^c$ is finite dimensional, the projections $\sP^c$ and $\sP^s$ also provide the direct decomposition
$X_j= X^c \oplus X_j^s$.

Following the arguments in parts (b) and (c) of the proof of \cite[Theorem~5.3.1]{PrSi16}, near $z_*$, we decompose $\bar{z}$  as
\begin{align*}
\bar z= \sx +\sy := \sP^c\bar{z} + \sP^s \bar{z}.
\end{align*}
Based on these notations, we define the normal form of \eqref{abstract eq neaar equlibrium}  as
\begin{equation}
    \label{normal form}
  \left\{ 
    \begin{aligned}
      \pa_t \sx &=T(\sx,\sy) &&\text{in}&&\Omega, \\
     \pa_t \sy + \sP^s \sA_* \sP^s \sy &=R(\sx,\sy) &&\text{in}&&\Omega, \\
      \sB_* \sy&=S(\sx,\sy) &&\text{on}&&\pa\Omega, \\
      \sx(0)&=\sx_0 , \quad \sy(0)=\sy_0 &&\text{in}&&\Omega .
    \end{aligned}
    \right.
  \end{equation}
Here $\sx_0=\sP^c \bar{z}_0$ and $\sy_0=\sP^s \bar{z}_0$ and
\begin{align*}
T(\sx,\sy) &= \sP^c \left( G(\sx+ \sy)-  G(\sx) \right) - \sP^c \sA_* \sy \\
R(\sx,\sy) &= \sP^s \left( G(\sx +\sy)-  G(\sx ) \right)  \\
S(\sx,\sy) &= H(\sx +\sy)-  H(\sx )  .
\end{align*}
We note that $G(\sx)=0$ for $\sx\in X^c$.  We,  nevertheless, include this term for reasons of consistency.
It is clear that
$T(\sx,0)=R(\sx,0)=S(\sx,0)=0$.

\medskip\noindent
Before proceeding with the proof, we list a result for a linear version of equation \eqref{normal form} that will be needed in the sequel.
It reads as follows.  
\begin{lemma}\label{lem: stable part}
Let $T>0$.  Then the linear problem
\begin{equation*}
    \begin{aligned}
     \pa_t y + \sP^s \sA_* \sP^s y &={\sf f}(t) &&\text{in}&&\Omega, \\
      \sB_* y&={\sf g}(t) &&\text{on}&&\pa\Omega, \\
      y(0)&=y_0    &&\text{in}&&\Omega  
    \end{aligned}
    \label{normal form stable}
  \end{equation*}
admits for each initial value  $y_0=(u_0,F_0,\theta_0,m_0)\in X_{\gamma,\mu}^s$ 
and each function  
$$({\sf f},{\sf g})\in L_{p,\mu}(J_T;X_0^s) \times \bF_\mu(J_T)$$ 
satisfying the compatibility condition
$
\sB_* y_0= {\sf g}(0)
$
a unique solution 
$$
y\in W^1_{p,\mu}(J_T; X_0^s)\cap L_{p,\mu}(J_T;X_1^s).
$$
Moreover,  there exists a constant $M_0$, which is independent of $T \in (0,\infty)$, such that
$$
\|y\|_{\bE_{1,\mu}(J_T)} \leq M_0 \left( \|{\sf f}\|_{\bE_{0,\mu}(J_T)}+  \|{\sf g}\|_{\bF_\mu(J_T)} + \|y_0\|_{X_{\gamma,\mu}} \right).
$$
\end{lemma}
\begin{proof}
The proof can be reproduced line by line by following the proof of \cite[Proposition~3.3]{PSZ09}.
\end{proof}

\medskip\noindent
We will now continue with the proof of Theorem~\ref{thm:stability}.  
Suppose that 
\begin{equation}
\label{initial value assumption}
\sx_0\in B_{X^c}(0,\delta) \quad \text{and} \quad \sy_0\in B_{X^s_{\gamma,\mu}}(0,\delta)
\end{equation}
for a number $\delta>0$ to be determined later. 
We already know that \eqref{abstract eq neaar equlibrium} has a   solution  $\bar{z} $ with initial value $\bar{z}_0=\sx_0+\sy_0$ on  maximal interval of existence $[0,T_+)$, or equivalently,  \eqref{normal form} has a solution $(\sx,\sy)$ on $[0,T_+)$.

As in \cite[Theorem~5.3.1]{PrSi16}, we can show that there exists some constant $C_1 >0$ such that
\begin{align}
\begin{split}\label{est T&R}
\|T(\sx,\sy)\|_{X_0} & \leq C_1 \|\sy\|_{X_1}, \\
\|R(\sx,\sy)\|_{X_0} & \leq  \Phi(r) \|\sy\|_{X_1} 
\end{split}
\end{align}
for all $\sx\in  B_{X^c}(0,r)$ and $\sy\in  B_{X^s_{\gamma,\mu}}(0,r) \cap X_1$ with sufficiently small $r>0$.
We recall here that 
 $\Phi(r)\to 0^+$ as $r\to 0^+$.  
Define
$$
\omega_0 = \frac{1}{2}\inf\{{\rm Re} \lambda: \lambda\in \sigma(A_*)\setminus \{0\}\}.
$$
For any $\omega\in (0,\omega_0)$, we define the map $e_\omega: L_{1,loc}(\bR_+) \to  L_{1,loc}(\bR_+): u  \mapsto e^{\omega t} u(t)$. 
For arbitrary $T\in (0,T_+)$,  we will establish an estimate of the form
\begin{align}\label{est H}
\| e_\omega S(\sx,\sy)  \|_{\bF_\mu(J_T)} \leq \Phi(r)\| e_\omega \sy\|_{\bE_{1,\mu}(J_T)} , 
\end{align}
whenever $\|\sx(t)\|_{X_0}, \|(\sx+\sy)(t)\|_{X_{\gamma,\mu} } \leq r$, $t\in [0,T ]$.
Let
$
\wK (z)=  K(z_*+ z)  -  K(z_*) .
$
 For $z_i\in X_{\gamma,\mu}$, $i=1,2$, we have the estimates
\begin{align}
\label{est f}
\begin{split}
\|\wK(z_1)\|_{W^{2\mu-2/p}_p(\Omega)}   & \leq  \Phi(\|z_*\|_{X_{\gamma,\mu}}+ \|z_1\|_{X_{\gamma,\mu}})   \|z_1 \|_{W^{2\mu-2/p}_p(\Omega)},    \\
|(\wK(z_1)-\wK(z_2)) (x)| & \leq \Phi(\|z_*\|_{X_{\gamma,\mu}}+ \sum\limits_{i=1,2}\|z_i\|_{X_{\gamma,\mu}} )  |(z_1-z_2)(x)  |  ,\quad x\in \overline{\Omega}, 
\end{split}
\end{align}
in view of Lemma~\ref{lem: Nemyskii} and a mean value theorem argument as in Lemma~\ref{lem:mapping property-linearization}.

We set $\sx=(u_1,F_1,\theta_1,m_1) $ and $\sy=(u_2,F_2,\theta_2,m_2)$.
It holds that
\begin{equation*}
 S(\sx,\sy) = \nu\cdot {\rm tr}_{\partial_\Omega}(\wK( \sx+\sy)  \nabla \theta_2).
\end{equation*}
Note that, in the above computations, we have used the fact that $\nabla \theta_1=0$, 
which follows from the fact that $\sx\in X^c=\{0\}\times \{0_3\}\times \bR\times \bR^3$.
We   start with the $L_{p,\mu}$-estimate, which reads
\begin{align}
\notag & \|e_\omega S(\sx,\sy)\|_{L_{p,\mu}(J_T;L_p(\partial\Omega))} \\
\label{normal form eq0}
& \leq   \Phi(r) \| e_\omega  {\rm tr}_{\partial\Omega } (\nabla \theta_2 ) \|_{L_{p,\mu}(J_T;L_p(\partial\Omega))}  \leq   \Phi(r) \|e_\omega  \sy\|_{\bE_{1,\mu}(J_T )}.
\end{align}
In \eqref{normal form eq0}, we have used \eqref{est f} and the assumption   $\|\sx+\sy\|_{X_{\gamma,\mu}}\leq r$.
$\|e_\omega S(\sx,\sy)\|_{L_{p,\mu}(J_T; W^{1-1/p}_p(\partial\Omega))}$ can be estimated   in a similar way by observing that $W^{1-1/p}_p(\partial\Omega)$ is a Banach algebra.
Let  $r=1/2-1/2p$.  It follows from Lemma~\ref{lem: equivalent seminorm}  that 
\begin{align}
\notag & [ e_\omega  S(\sx,\sy)]_{W^r_{p,\mu}(J_T;L_p(\partial\Omega))}^p \\
\notag & \leq  C\| e_\omega S(\sx,\sy)\|_{L_{p,\mu}(J_T;L_p(\partial\Omega))}^p\\
\label{normal form eq1}   & \quad + C    \iint_{B_T^1} \frac{ \left\| s^{1-\mu} e^{\omega s }  [ \wK (\bar{z}(t)) - \wK  (\bar{z}(s))]   \nabla  \theta_2  (t) \right\|_{L_p(\partial\Omega)}^p}{(t-s)^{1+r p}} \,  ds \,  dt \\
\label{normal form eq2} &  \quad + C    \iint_{B_T^1}  \!  \frac{ \left\| s^{1-\mu} e^{\omega s } \wK  (\bar{z}(s)) \nabla \left[   \theta_2  (t)-   \theta_2  (s)  \right]   \right\|_{L_p(\partial\Omega)}^p}{(t-s)^{1+r p}}  \, ds \,  dt  .
\end{align}
To estimate \eqref{normal form eq1}, we recall that $X^c_0\doteq X^c_1$; and  observe that for 
$$(s,t)\in B_T^1=\{(s,t)\in (0,T)^2: 0<t-s<1\},$$
we obtain
 \begin{align}
\notag
 \frac{\| s^{1-\mu} e^{\omega s} (\sx(t)-\sx(s))\|_{L_p(\partial\Omega)}^p}{(t-s)^{1+\alpha p }} 
\notag  & \leq   C \frac{\| s^{1-\mu} e^{\omega s} (\sx(t)-\sx(s))\|_{X_0}^p}{(t-s)^{1+\alpha p }}  \\
\notag   & \leq \frac{C  }{(t-s)^{1+\alpha p }}\left(\int_s^t   \tau^{1-\mu} e^{\omega \tau } \|T(\sx,\sy)(\tau)\|_{X_0}\, d\tau \right)^p \\
\label{normal form est x1}
& \leq \frac{C  }{(t-s)^{1+\alpha p }} \left( \int_s^t \tau^{1-\mu}    \|e_\omega \sy(\tau)\|_{X_1}\, d\tau \right)^p \\
\label{normal form est x2}
& \le  C (t-s)^\beta \left(\int_s^{ (s+1)\wedge T  }  \tau^{(1-\mu)p} \|e_\omega \sy (\tau)\|_{X_1}^p \, d\tau \right),
\end{align}
for some constant $C$ which is independent of $T \in [0,T_+)$, where $\beta=(1-\alpha)p-2 >0$ due to \eqref{indices cond},
and $(s+1)\wedge T:=\min\{s+1,T\}$.
We have used \eqref{est T&R} in \eqref{normal form est x1} and H\"older's inequality in \eqref{normal form est x2}.
Observe that
\begin{align}
\notag  & \iint_{B_T^1}  (t-s)^\beta  \left(\int_s^{(s+1)\wedge T} \tau^{(1-\mu)p} \|e_\omega \sy (\tau)\|_{X_1}^p \, d\tau \right) \, ds \,dt  \\
\notag &= \int_0^T   \int_s^{ (s+1)\wedge T }   (t-s)^\beta \left(\int_s^{ (s+1)\wedge T} \tau^{(1-\mu)p} \|e_\omega \sy (\tau)\|_{X_1}^p \, d\tau \right)\, dt \,ds  \\
\notag 
& \le  \left( \int_0^T  \int_s^{(s+1)\wedge T} \tau^{(1-\mu)p} \|e_\omega \sy (\tau)\|_{X_1}^p \, d\tau  \, ds  \right) \left( \int_0^1  t^\beta \, dt \right)   \\
\notag 
& = \left( \int_0^T  \tau^{(1-\mu)p} \|e_\omega \sy (\tau)\|_{X_1}^p \, d\tau  \int_{(\tau-1)\vee 0}^\tau  \, ds  \right) \left( \int_0^1  t^\beta \, dt \right)   \\
\label{double ing 2}
& \le \frac{1}{\beta +1}  \| e_\omega \sy  \|_{\bE_{1,\mu}(J_T)}^p,
\end{align}
 where $(\tau -1)\vee 0=\max\{\tau-1, 0\}$.
Employing  \eqref{est f},  \eqref{double ing 2} and Lemma~\ref{lem: equivalent seminorm}, we have 
\begin{align*}
& \iint_{B_T^1}  \frac{ \left\| s^{1-\mu} e^{\omega s }  [\wK (\bar{z}(t)) - \wK (\bar{z}(s))]   \nabla   \theta_2   (t) \right\|_{L_p(\partial\Omega)}^p}{(t-s)^{1+r p}} \,  ds \,  dt \\
& \le   \Phi(r)  \iint_{B_T^1}  \frac{ \left\| s^{1-\mu} e^{\omega s }  ( \sx(t)  - \sx(s))      \right\|_{L_p(\partial\Omega)}^p}{(t-s)^{1+r p}} \,  ds \,  dt 
    + \Phi(r)  \iint_{B_T^1}  \frac{ \left\| s^{1-\mu} e^{\omega s }  ( \sy(t)  - \sy(s))      \right\|_{L_p(\partial\Omega)}^p}{(t-s)^{1+r p}} \,  ds \,  dt \\
& \leq \Phi(r) \|e_\omega \sy\|_{\bE_{1,\mu}(J_T)}^p,
\end{align*}
where we have used the fact that 
$
\|   \nabla   \theta_2   (t)  \|_\infty  \le C \| \sy(t)\|_{X_{\gamma,\mu}} \le Cr.
$
The estimate for \eqref{normal form eq2} is a direct consequence of \eqref{est f}:
\begin{align*}
& \iint_{B_T^1}  \frac{ \left\| s^{1-\mu} e^{\omega s } \wK (\bar{z}(s)) \nabla \left[   \theta_2   (t)-    \theta_2   (s)  \right]   \right\|_{L_p(\partial\Omega)}^p}{(t-s)^{1+r p}}  \, ds \,  dt  \\
& \leq \Phi(r)
\iint_{B_T^1}   \frac{ \left\| s^{1-\mu} e^{\omega s }   \nabla \left[    \theta_2   (t)-    \theta_2  (s)  \right]   \right\|_{L_p(\partial\Omega)}^p}{(t-s)^{1+r p}}  \, ds \,  dt .
\end{align*}
Summarizing the above discussion and applying \eqref{normal form eq0} and Lemma~\ref{lem: equivalent seminorm}, we have
\begin{align*}
  [ e_\omega  S(\sx,\sy)]_{W^r_{p,\mu}(J_T;L_p(\partial\Omega))}^p  
& \leq  \Phi(r) \|e_\omega \sy\|_{\bE_{1,\mu}(J_T)}^p+ \Phi(r)   \|e_\omega \nabla \theta_2\|_{W^r_{p,\mu}(J_T;L_p(\partial\Omega))}^p  \\
& \leq \Phi(r)   \|e_\omega \sy\|_{\bE_{1,\mu}(J_T)}^p.
\end{align*}
In the last step, we have used the embedding
$$
W^{1-1/2p}_{p,\mu}(J_T;L_p(\partial\Omega))\cap L_{p,\mu}(J_T; W^{2-1/p}_p (\partial\Omega))   \hookrightarrow W^r_{p,\mu}(J_T;W^1_p(\partial\Omega)),
$$
see  for instance \cite[Proposition~3.2]{MeSc12}, and \cite[Theorem 4.5]{MeSc12}.
This yields  \eqref{est H}.

Fix $r>0$ so that estimates \eqref{est T&R} and \eqref{est H} hold. 
We put
$$
t_0 = \sup\{t\in (0,T_+): \, \|\sx(\tau)\|_{X_0} , \|\sy(\tau)\|_{X_{\gamma,\mu}} \leq r, \, \tau \in [0,t] \}.
$$
Assume that $t_0<T_+$. Then Lemma~\ref{lem: stable part} implies 
\begin{align*}
\| e_\omega \sy \|_{\bE_{1,\mu}(J_{t_0})} & \leq M_0  \left(  \|\sy_0\|_{X_{\gamma,\mu}} + \| e_\omega S(\sx,\sy)\|_{\bF_\mu(J_{t_0})} + \| e_\omega R(\sx,\sy)\|_{\bE_{0,\mu}(J_{t_0})} \right)  \\
&  \leq M_0 \|\sy_0\|_{X_{\gamma,\mu}} + \Phi(r)  \| e_\omega \sy \|_{\bE_{1,\mu}(J_{t_0})}  .
\end{align*}
Choosing $r>0$ sufficiently small so that $\Phi(r)<1/2$ yields
$$
\| e_\omega \sy \|_{\bE_{1,\mu}(J_{t_0})} \leq 2 M_0 \|\sy_0\|_{X_{\gamma,\mu}} , 
$$
which further implies that
$$
\|e_\omega \sy \|_{C([0,t_0];X_{\gamma,\mu})} \leq M_1 \|\sy_0\|_{X_{\gamma,\mu}} .
$$
We can derive an estimate for $\sx$ by using \eqref{normal form} and \eqref{est T&R}:
\begin{align*}
\|\sx(t)\|_{X_0} & \leq \|\sx_0\|_{X_0} + \int_0^t \|T(\sx,\sy)(\tau)\|_{X_0}\, d\tau \\
& \leq \|\sx_0\|_{X_0} + C_1 \int_0^t e^{-\omega \tau} \tau^{\mu-1} \| \tau^{1-\mu} e^{\omega \tau} \sy(\tau)\|_{X_1}\, d\tau \\
& \leq \|\sx_0\|_{X_0} + C \| e_\omega \sy \|_{\bE_{1,\mu}(J_{t_0})}  \leq \|\sx_0\|_{X_0} + M_2 \|\sy_0\|_{X_{\gamma,\mu}} .
\end{align*}
In the last line we employed   H\"older's inequality  and  $\mu>1/p$.
By choosing $\delta< r/2(1+M_1+M_2)$,  where $\delta$ was introduced in~\eqref{initial value assumption}, we have for all 
$\sx_0\in B_{X^c}(0,\delta)$ and $\sy_0\in B_{X^s_{\gamma,\mu}}(0,\delta)$ 
and all
$ t\in [0,t_0)$
$$
\|\sx(t)\|_{X_0} + \|\sy(t)\|_{X_{\gamma,\mu}} \leq  \|\sx_0\|_{X_0} + (M_1 +M_2)  \|\sy_0\|_{X_{\gamma,\mu}} \leq r/2,
$$
a contradiction to the definition of $t_0$. Therefore, $t_0=T_+$. With the choice $\delta< r/2(1+M_1+M_2)$, 
the above discussion shows that there exists a constant $M_3>0$ such that for any   $t_1\in (0,T_+)$,  
$$
\|\bar{z}\|_{C([0,t_1]; X_{\gamma,\mu})} + \|\bar{z}\|_{\bE_{1,\mu}(J_{t_1})} \leq M_3.
$$

Let $\tau\in (0, t_1)$ be fixed and let $t$ be any number in $[\tau, t_1]$.
Then we have
\begin{equation*}
\|\bar z(t)\|_{X_{\gamma,1}}\le 
\sup_{s\in [\tau, t_1]}\|\bar z(s)\|_{X_{\gamma,1}}\le C(\tau) \| \bar z\|_{\bE_{1,1}([\tau,t_1])}
\le C(\tau)\tau^{\mu-1}\| \bar z\|_{\bE_{1,\mu}([\tau,t_1])
}\le C(\tau)\tau^{\mu-1} M_3.
\end{equation*}
This implies that $\bar{z}\in BC([\tau,T_+),X_{\gamma,1})$.
Theorem~\ref{thm:global existence} then implies that $T_+=\infty$. The rest of the proof is exactly the same as part (f) of the proof of  \cite[Theorem~5.3.1]{PrSi16}.
\end{proof}

\appendix
{\small  
\medskip
\section{Properties of fractional Sobolev spaces with temporal weights}

In this section, we establish some properties of fractional Sobolev spaces with temporal weights.
The results are employed in obtaining estimates for nonlinear mappings, but are also of independent interest.
We recall that
$$
\prescript{}{0}{W}^r_{p,\bar{\mu}}(J_T; X):=\{u\in W^r_{p,\bar{\mu}}(J_T; X) : \, \gamma_0 u=0  \},
$$
where $r,\bar{\mu}\in (1/p,1]$ with $r+\bar{\mu}>1+1/p$,  and  $X$ is a Banach space. 
See \cite[Proposition 2.10]{MeSc12}.

\begin{lemma}
\label{lem: time weight equi norm}
Let $X$ be a Banach space.
Given    $r,\mu\in (1/p,1]$  such that $r+\mu>1+1/p$,  there exists a constant $C>0$, 
which is independent of $T\in (0,\infty]$, such that
\begin{align*}
\left( \int_0^T t^{(1-\mu-r)p} \|u(t)\|_X^p \, dt \right)^{1/p} \leq C  \|u\|_{W^r_{p,\mu}(J_T;X)}
\end{align*}
for all $u\in \zW^r_{p,\mu}(J_T;X)$, where $\zW^r_{p,\mu}(J_T;X)$ is equipped with the intrinsic norm.
\end{lemma}
\begin{proof}
The case $r \in \{0,1\}$ follows from the definition of $L_{p,\mu}(J_T;X)$ and \cite[Lemma~1.1.2(b)]{MeyriesThesis}. When $r\in (0,1)$,
\begin{align}
\notag &\left( \int_0^T t^{(1-\mu-r)p} \|u (t)\|_X^p \, dt \right)^{1/p} \\
\notag & = \left( \int_0^T  \left( t^{  -\mu-r } \int_0^t \|u (t)\|_X \, ds \right)^p \, dt \right)^{1/p}\\
\label{eq time weight equi norm}
& \leq  \left( \int_0^T  \left( t^{  -\mu-r } \int_0^t \|u (t)-u(s)\|_X \, ds \right)^p \, dt \right)^{1/p} + 
\left( \int_0^T  \left( t^{  -\mu-r } \int_0^t \|u (s)\|_X \, ds \right)^p \, dt \right)^{1/p}.
\end{align}
We will use H\"older's inequality to estimate the first term in \eqref{eq time weight equi norm} as follows.
\begin{align*}
& \left( \int_0^T  \left( t^{  -\mu-r } \int_0^t \|u (t)-u(s)\|_X \, ds \right)^p \, dt \right)^{1/p}\\
& \leq \left( \int_0^T  t^{  (-\mu-r)p } \left(    \int_0^t s^{(1-\mu)p} \|u (t)-u(s)\|_X^p \, ds \right) \left(    \int_0^t  s^{(\mu-1)p'} \, ds \right)^{p/p'} \, dt \right)^{1/p}\\
& \leq C   \left( \int_0^T  t^{   -1-r p } \left(    \int_0^t s^{(1-\mu)p} \|u (t)-u(s)\|_X^p \, ds \right)  \, dt \right)^{1/p}\\
& \leq C [u]_{W^r_{p,\mu}(J_T;X)}.
\end{align*}
In the last line, we  used that $1/t < 1/( t-s)$ for $s\in (0,t)$.
Observe that it follows from the condition   $\mu \in (1/p,1]$ that $(\mu-1)p'>-1$.
To estimate the second term in \eqref{eq time weight equi norm}, we will apply Hardy's inequality, 
c.f. \cite[Lemma~3.4.5]{PrSi16}, to obtain 
\begin{align*}
\left( \int_0^T  \left( t^{  -\mu-r } \int_0^t \|u (s)\|_X \, ds \right)^p \, dt \right)^{1/p} \leq 
\frac{1}{(\mu+r-1/p)} \left( \int_0^T t^{(1-\mu-r)p} \|u(t)\|_X^p \, dt \right)^{1/p}.
\end{align*}
Hence, we have shown that
\begin{align}
\notag &\left( \int_0^T t^{(1-\mu-r)p} \|u (t)\|_X^p \, dt \right)^{1/p} 
\le   C \|u\|_{W^r_{p,\mu}(J_T;X)} + \frac{1}{(\mu+r-1/p)} \left( \int_0^T t^{(1-\mu-r)p} \|u(t)\|_X^p \, dt \right)^{1/p}.
\end{align}
In view of the condition $r+\mu>1+1/p$, the asserted estimate then follows.
\end{proof}

\medskip
\goodbreak
    \begin{lemma}\label{lem: auxiliary lemma}
    Suppose  $\mu\in [0,1]$. Then we have
    \begin{equation*}
    \left( t^{1-\mu} - s^{1-\mu} \right)^p  \leq  t^{-\mu p} (t-s)^p\quad\text{and}\quad  |t^{\mu -1} - s^{\mu-1}|^p\le s^{(\mu-1)p} t^{-p} (t-s)^p,
    \quad 0<s<t < \infty.
    \end{equation*}
    \end{lemma}
\begin{proof}
The assertions are clear for $\mu \in\{0,1\}$.
In  case $\mu\in (0,1)$  we obtain
$$
(t^{1-\mu} - s^{1-\mu})^p = t^{(1-\mu)p} \left(1- \left(s/t\right)^{1-\mu}\right)^p
\le  t^{(1-\mu)p}  \left(1- (s/t)\right)^p = t^{-\mu p} (t-s)^p.
$$
This estimate, in turn, yields
$$ 
|t^{\mu-1}- s^{\mu -1}|^p =  s^{(\mu-1)p} t^{(\mu -1)p} (t^{1-\mu}- s^{1-\mu})^p \le s^{(\mu-1)p} t^{-p} (t-s)^p.
$$
\end{proof}

For $u\in L_{1,{\rm loc}}(J_T;X)$ we define $(\Phi_\mu u)(t):= t^{1-\mu} u(t)$, see~\cite{PrSi04}. It is then clear that 
\begin{equation}
\label{isometric}
\Phi_\mu: L_{p,\mu}(J_T; X) \to L_p(J_T;X)\quad\text{is an isometric isomorphism}, 
\end{equation}
 and its inverse $\Phi^{-1}_\mu$ is given by $(\Phi_\mu^{-1} v)(t)= t^{\mu -1}v(t)$.
The next result shows that $\Phi_\mu$ induces an isomorphism for the Sobolev spaces 
${_0}W^r_{\mu, p}(J_T;X)$.

\goodbreak

\begin{lemma}
\label{lem: Phi-mu}
Let $X$ be a Banach space. 
Suppose that   $r,\mu\in (1/p,1]$  and $r+\mu > 1+1/p$.
Then it holds that
$$ \Phi_\mu \in {\cL}is (_{0}W^r_{p,\mu}(J_T;X), {_0}W^r_p(J_T;X)).$$
Moreover, there exists a constant $C$ which is independent of  $T\in (0,\infty]$ such that
\begin{equation}
\label{Phi-mu}
\| \Phi_\mu u\|_{W^r_{p,\mu}(J_T;X)}\le C \| u \|_{W^r_{p}(J_T;X)},\qquad 
\| \Phi^{-1}_\mu v\|_{W^r_{p}(J_T;X)}\le C \| v \|_{W^r_{p,\mu}(J_T;X)},
\end{equation}
where the spaces are equipped with their respective intrinsic norms.
\end{lemma}
\begin{proof}
The first part of the assertion has been established in \cite[Lemma 2.3]{MeSc12},
where the spaces are equipped with their respective interpolation norms   (see Remark~\ref{rem: norms}). 

\medskip
We will now establish the uniform estimates  in \eqref{Phi-mu} for intrinsic norms.
The case $r=1$ follows readily from Lemma~\ref{lem: time weight equi norm}. For the reader's convenience, we include a proof 
(see also \cite[Lemma 1.1.3]{MeyriesThesis}).
Suppose $u\in W^1_{p,\mu}(J_T;X)$. Then we obtain
\begin{equation*}
\| (\Phi_\mu u)^\prime \|_{L_{p}(J_T;X)} \le \left(\int_0^T \| t^{1-\mu} u^\prime (t)\|^p_X\, dt \right)^{1/p} 
  + (1-\mu) \left(\int_0^T t^{-\mu p} \| u(t) \|^p_X\, dt \right)^{1/p}
 \le C \| u\|_{W^1_{p,\mu} (J_T;X)},
\end{equation*}
where we used  Lemma~\ref{lem: time weight equi norm} with $r=1$.
Suppose now that $v\in {_0}W^1_{p,\mu}(J_T; X)$. Then we obtain
\begin{equation*}
\| (\Phi^{-1}_\mu v)^\prime \|_{L_{p,\mu}(J_T;X)} 
\le \Big(\int_0^T \|v^\prime (t)\|^p_X\,dt \Big)^{1/p} + |\mu-1| \left(\int_0^T t^{-1} \| v \|^p_X\,dt\right)^{1/p}
\le \| v \|_{W^1_p(J_T;X)},
\end{equation*}
where we employed, once more, Lemma~\ref{lem: time weight equi norm} with $r=\mu=1.$
These estimates together with \eqref{isometric} imply the assertion.

\medskip\noindent
We will now consider the case $r<1$ and $r+\mu > 1+1/p$.
Suppose $u\in {_0}W^r_{p,\mu}(J_T;X)$. Then we obtain
\goodbreak
\begin{align}
 [\Phi_\mu u ]_{W^r_{p}(J_T;X)} 
\notag & = \left(\int_0^T  \int_0^t  \frac{\| (\Phi_\mu u)(t)- (\Phi_\mu u)(s)\|_X^p}{(t-s)^{1+rp}}\, ds \,  d t \right)^{1/p} \\
\notag & \le \left(\int_0^T  \int_0^t s^{p(1-\mu)} \frac{\| u(t) - u(s) \|_X^p}{(t-s)^{1+rp}}\, ds \,  d t \right)^{1/p}  \\
\notag & \quad + \left(\int_0^T  \int_0^t  \frac{ (t^{1-\mu} - s^{1-\mu})^p}{(t-s)^{1+rp}} \| u (t) \|^p_X \, ds \,  d t \right)^{1/p} \\
\notag & \le [u ]_{W^r_{p,\mu}(J_T;X)} +  \left(\int_0^T  \int_0^t  t^{-\mu p} (t-s)^{(1-r)p-1} \| u (t)\|^p_X \, ds \,  d t \right)^{1/p} \\
& \label{ab} \le  [u ]_{W^r_{p,\mu}(J_T;X)} +c(r,p)  \left(\int_0^T   t^{(1-\mu -r )p} \| u (t)\|^p_X  \,  d t \right)^{1/p} \\
& \label{ac}\le  C \| u \| _{W^r_{p,\mu}(J_T;X)}.
\end{align}
We used Lemma~\ref{lem: auxiliary lemma} in \eqref{ab} and Lemma~\ref{lem: time weight equi norm} in \eqref{ac}.

\medskip\noindent
Suppose that $v\in {_0}W^r_{p}(J_T;X)$. Then we obtain
\begin{align}
 [(\Phi_\mu)^{-1} v ]_{W^r_{p,\mu}(J_T;X)} 
\notag & = \left(\int_0^T  \int_0^t s^{(1-\mu)p} \frac{\| (\Phi^{-1}_\mu v)(t)- (\Phi^{-1}_\mu v)(s)\|_X^p}{(t-s)^{1+rp}}\, ds \,  d t \right)^{1/p} \\
\notag & \le  \left(\int_0^T  \int_0^t  \frac{\|  v(t)- v(s)\|_X^p}{(t-s)^{1+rp}}\, ds \,  d t \right)^{1/p}  \\
 \notag & \quad         + \left(\int_0^T  \int_0^t s^{(1-\mu)p} \frac{|t^{\mu-1}-  s^{\mu-1}|^p}{ (t-s)^{1+rp}} \| v(t) \|^p_X\, ds \,  d t \right)^{1/p} \\
& \label{ad} \le  [u ]_{W^r_{p}(J_T;X)} +c(r,p)  \left(\int_0^T   t^{-rp} \| u (t)\|^p_X  \,  d t \right)^{1/p} \\
& \label{ae}\le  C \| u \| _{W^r_{p}(J_T;X)}.
\end{align}
Here we used, once more,  Lemma~\ref{lem: auxiliary lemma} in \eqref{ad} and Lemma~\ref{lem: time weight equi norm} in \eqref{ae}.
\end{proof}
\begin{proposition}
\label{extension-zero}
Let $X$ be a Banach space.
Suppose    $r,\mu\in (1/p,1]$ and  $r+\mu>1+1/p$. Then there exists an extension operator: 
\begin{equation*}
\cE_{J_T}: \zW^r_{p,\mu}(J_T;X) \to \zW^r_{p,\mu}(\bR_+;X)
\end{equation*}
such that its norm is independent of  $T\in (0, \infty]$,
 where the spaces are equipped with their intrinsic norms.
\end{proposition}
\begin{proof}
We define the extension operator by
\begin{align*}
\cE_{J_T} u (t) :=
\begin{cases}
u(t)  \quad &\text{for } 0<t\le T \\
\left(\frac{2T-t}{t}\right)^{1-\mu} u(2T-t) & \text{for } T < t \leq 2T \\
0 &\text{for }  2T<t.
\end{cases}
\end{align*}
The statement follows from Lemma~\ref{lem: Phi-mu},  \cite[Proposition~6.1]{PSS07}, and the commutativity of the diagram
\begin{equation*}
\begin{aligned}
&{_0}W^r_{p,\mu}(J_T;X)\quad\overset{\Phi_\mu}{\longrightarrow} &&{_0}W^r_{p}(J_T;X) \\
\phantom{X}&\quad\downarrow  \cE_{J_T}                     &&\quad\ \downarrow \cE_T \\ 
&{_0}W^r_{p,\mu}(\bR_+;X)\quad\overset{{\Phi^{-1}_\mu}}{\longleftarrow} &&\  {_0}W^r_p(\bR_+;X), 
\end{aligned}
\end{equation*}
where the extension operator $\cE_T$  on the right side is defined in \cite[Proposition~6.1]{PSS07}.
\end{proof}
\begin{remark}
\label{rem: norms}
For $r\in (0, 1)$, fractional Sobolev spaces with temporal weight can also be defined by means of real interpolation.
It then holds that
$$
W^r_{p,\mu}(J_T;X) \stackrel{\cdot}{=} \left( L_{p,\mu}(J_T;X), W^1_{p,\mu}(J_T;X)  \right)_{r,p},
$$
where the symbol $\stackrel{\cdot}{=}$ means equivalent norms; see
\cite[Proposition 1.1.13]{MeyriesThesis}, or \cite[equation (2.6)]{MeSc12}.
The corresponding norm  is called  the interpolation norm of $W^r_{p,\mu}(J_T;X)$.
It is pointed out in  \cite[Remark~1.1.15]{MeyriesThesis}   that the equivalence constant between the intrinsic and the interpolation norm of $W^s_{p,\mu}(J_T;X)$ 
blows up as $T~\to~0^+$.\\
\cite[Lemma~1.1.15]{MeyriesThesis} shows that there exists an extension operator $\cE_T: \zW^r_{p,\mu}(J_T;X) \to \zW^r_{p,\mu}(\bR_+;X)$ whose norm is independent of $T$, where both spaces are equipped with the corresponding interpolation norms. 

The merit of Proposition~\ref{extension-zero}, as opposed to \cite[Lemma~1.1.15]{MeyriesThesis}, lies in the fact that we can use intrinsic norms. 
This greatly facilitates deriving estimates for nonlinear boundary terms.
We should like to mention, though, that Proposition~\ref{extension-zero} requires
the conditions   $r\in (1/p, 1]$  and $r+\mu > 1+1/p$.

\medskip

\end{remark}

\goodbreak

The following result is used in Section~\ref{section:stability} in order to show stability of (constant) equilibria.
\begin{lemma}\label{lem: equivalent seminorm}
Let $T>0$, $r\in (0,1)$, $\omega \in \bR$, and $\mu \in (1/p,1]$. 
We then set 
$$B_T =\{(s,t)\in (0,T)^2: 0<s<t\}\quad\text{and} \quad B_T^1=\{(s,t)\in (0,T)^2: 0<t-s<1\}.$$
Suppose that $X$ is a Banach space.
Then 
\begin{align*}
 [ e_\omega u ]_{W^r_{p,\mu}(J_T; X)}  
 &\leq C \|e_\omega u \|_{L_{p,\mu} (J_T;X)} +  \left(   \iint_{B_T^1} \frac{\| s^{1-\mu} e^{\omega s} (u(t)-u(s)) \|_X^p}{(t-s)^{1+r p}} \, ds \, dt \right)^{1/p} \\
& \leq C \| e_\omega u \|_{W^r_{p,\mu}(J_T; X)} ,
\end{align*}
where the constant $C=C(p,r, \omega)$ is independent of $T$ and   $e_\omega: L_{1,loc}(\bR_+) \to  L_{1,loc}(\bR_+): u  \mapsto e^{\omega t} u$. 
\end{lemma}
\begin{proof}
Using \eqref{seminorm}, we estimate as in \cite[Lemma~11]{LPS06} and obtain
\begin{align}
\notag & [ e_\omega u ]_{W^r_{p,\mu}(J_T; X)}   \\
\notag  &\leq  \left( \iint_{B_T\setminus B_T^1}  s^{p(1-\mu)} \frac{\|e^{\omega t} u(t)- e^{\omega s}u(s)\|_X^p}{(t-s)^{1+r p}}\, ds \,  d t \right)^{1/p} \\ 
\notag & \quad + \left(\iint_{B_T^1}  s^{p(1-\mu)} \frac{\|e^{\omega t} u(t)- e^{\omega s} u(s)\|_X^p}{(t-s)^{1+r p}}\, ds \,  d t \right)^{1/p} \\
\notag  &\leq    \left( \int_0^T \int_0^{t-1}   \frac{\|  t^{1-\mu}\,  e^{\omega t} u(t) \|_X^p}{(t-s)^{1+r p}} \, ds \,  d t \right)^{1/p} 
                   +   \left( \int_0^T \int_{s+1}^T   \frac{ \|  s^{1-\mu}\,  e^{\omega s} u(s) \|_X^p }{(t-s)^{1+r p}} \, dt \,  d s \right)^{1/p} \\
\notag & \quad + \left( \iint_{B_T^1}  s^{p(1-\mu)}  e^{\omega t p} \|u(t)\|_X^p \frac{|e^{-\omega(t-s)}-1|^p}{(t-s)^{1+r p}}\, ds \,  d t \right)^{1/p} + \left( \iint_{B_T^1}  s^{p(1-\mu)} e^{\omega s p} \frac{\|  u(t)-   u(s)\|_X^p}{(t-s)^{1+r p}}\, ds \,  d t \right)^{1/p}   \\
\notag 
 &\leq    \| e_\omega u\|_{L_{p,\mu} (J_T;X)}
\left[ 2  \left(\int_1^\infty \frac{d\tau }{\tau^{1+r p}}\right)^{1/p}  + c(\omega) \left(\int_0^1 \frac{d\tau }{\tau ^{1+(r -1)p}}\right)^{1/p}\right] 
 \\
\notag & \quad + \left( \iint_{B_T^1}  s^{p(1-\mu)} e^{\omega s p} \frac{\|  u(t)-   u(s)\|_X^p}{(t-s)^{1+r p}}\, ds \,  d t \right)^{1/p}  \\
\notag  & \leq  C \| e_\omega u\|_{L_{p,\mu} (J_T;X)} +   \left( \iint_{B_T^1} s^{p(1-\mu)} e^{\omega s p} \frac{\|  u(t)-   u(s)\|_X^p}{(t-s)^{1+r p}}\, ds \,  d t \right)^{1/p} .
\end{align}
In the derivation above, we have used that  $s<t$ for $(s,t)\in B_T$. 
\end{proof}

\goodbreak
Our next result deals with multiplication properties in weighted Sobolev spaces.
\begin{lemma}
\label{lem: multiplication}
Let $T_*>0$ be given.
\begin{itemize} 
\item[{\em (i)}] There exists a constant $C>0$, which is independent of $T\in (0,T_*]$, such that 
$$
\|u v \|_{\bF_\mu (J_T)} \leq C \|u   \|_{\bF_\mu (J_T)} \|  v \|_{\bF_\mu (J_T)} ,\quad \text{for all } u,v\in \zbF_\mu (J_T).
$$
\item[{\em (ii)}] There exists a constant $C>0$,  which is independent of $T\in (0,T_*]$,  such that 
$$
\|u v \|_{\bF_\mu (J_T)} \leq C \|u   \|_{\bF_\mu (J_T)} \|  v \|_{\bF_{1,\mu}(J_T)} ,\quad \text{for all } (u,v)\in \bF_\mu (J_T) \times \zbF_{1,\mu}(J_T),
$$
where $\bF_{1,\mu}(J_T)$ is defined as
\begin{equation}
\label{F-mu-1}
\bF_{1,\mu}(J_T)   := W^{1-1/2p}_{p,\mu}(J_T;L_p(\partial\Omega))\cap  C([0,T];W^{2\mu-3/p}_p(\partial\Omega)). 
 \end{equation}
\end{itemize}
\end{lemma}
\begin{proof}
(i) The assertion  follows from the fact that $\bF_\mu(J_T)$ is a Banach algebra and Proposition~\ref{extension-zero}. 
See also \cite[Lemma~1.3.23]{MeyriesThesis}.

\medskip\noindent
(ii)
It is an easy task to check that 
\begin{align*}
 \|u v \|_{L_{p,\mu}(J_T; W^{1-1/p}_p(\partial\Omega))}\leq    C \|v\|_{L_\infty(J_T;W^{1-1/p}_p(\partial\Omega))}   \|u   \|_{L_{p,\mu}(J_T; W^{1-1/p}_p(\partial\Omega))}    
\end{align*} 
for some $C>0$ independent of $T\in (0,T_*]$.  
In addition, one has
\begin{align}
\notag &\|u v \|_{W^{1/2-1/2p}_{p,\mu}(J_T; L_p(\partial\Omega))} \\
\label{2nd seminorm}
& \leq \|v\|_{C([0,T]\times \overline{\Omega}) } \|u   \|_{L_{p,\mu}(J_T; L_p(\partial\Omega))}  +  [u v]_{W^{1/2-1/2p}_{p,\mu}(J_T; L_p(\partial\Omega))} .
\end{align}  
We can derive from  Proposition~\ref{extension-zero} and \cite[equation~(2.19)]{MeSc12} that
$$
\prescript{}{0}{W}^{1-1/2p}_{p,\mu}(J_T; L_p(\partial\Omega)) \hookrightarrow \prescript{}{0}{W}^{s_0}_p(J_T; L_p(\partial\Omega))
$$
for some $s_0>1/2 + 2/p $ with embedding constant independent of $T\in (0,T_*]$.
This implies that 
$$
\zbF_{1,\mu}(J_T) \hookrightarrow  C^\sigma([0,T]; L_p(\partial\Omega))
$$
for some $\sigma>1/2+1/p$
with embedding constant independent of $T$.
Therefore, 
the second term on the right hand side of \eqref{2nd seminorm} can be estimated as follows:
\begin{align*}
& [u v]_{W^{1/2-1/2p}_{p,\mu}(J_T; L_p(\partial\Omega))}^p  \\
& \quad \leq      \int_0^T \int_0^t   s^{p(1-\mu)} \frac{\|  u(t) v(t) -  u(s) v(s) \|_{L_p(\partial\Omega)}^p}{ (t-s)^{1/2+p/2} } \, ds  dt  
\\
& \quad \leq   	\|v\|_{C([0,T]\times \overline{\Omega}) }^p [ u]_{W^{1/2-1/2p}_{p,\mu}(J_T; L_p(\partial\Omega))}^p \\
& \qquad +    \int_0^T \int_0^t   s^{p(1-\mu)} \|u(s)\|_{C(  \overline{\Omega}) }^p  \frac{\|  v(t) -  v(s) \|_{L_p(\partial\Omega)}^p}{ (t-s)^{1/2+p/2} } \, ds  dt   
\\
& \quad \leq   		\|v\|_{C([0,T]\times \overline{\Omega}) }^p [ u]_{W^{1/2-1/2p}_{p,\mu}(J_T; L_p(\partial\Omega))}^p \\
& \qquad +  \|v\|_{C^\sigma([0,T]; L_p(\partial\Omega))}^p \int_0^T \int_0^t   s^{p(1-\mu)} \|u(s)\|_{C(  \overline{\Omega}) }^p 
 (t-s)^{\sigma p-(1/2+p/2)}  \, ds  dt    \\
& \quad \leq   	\|v\|_{C([0,T]\times \overline{\Omega}) }^p [ u]_{W^{1/2-1/2p}_{p,\mu}(J_T; L_p(\partial\Omega))}^p + C_1 \|v\|_{C^\sigma([0,T]; L_p(\partial\Omega))}^p \|u\|_{L_{p,\mu}(J_T;C( \overline{\Omega})) }^p  
\end{align*}
for some constant $C_1=C_1(T)>0$ that is uniform in $T\in (0,T_*]$.
This implies
$$[u v]_{W^{1/2-1/2p}_{p,\mu}(J_T; L_p(\partial\Omega))}\le C  \|v\|_{_{0}\bF_{1,\mu}(J_T)} \|v\|_{\bF_\mu(J_T)}, $$
where $C=C(T)$ is uniform in $T\in (0,T_*]$.
\end{proof}

{\small
\medskip
\section{Properties of nonlinear maps} 
In this section, we establish some mapping properties for the nonlinear operators in \eqref{magneto sys}.
Our first step is to study the Nemyskii operators induced by the functions in \eqref{assumption}.

\begin{lemma}
\label{lem: Nemyskii}
Suppose $\varphi\in C^5(\bR)$ and $X\in \{W^{2\mu-2/p}_p(\Omega), W^{2\mu+1-2/p}_p(\Omega), \bE_{2,\mu}^k(J_T), \bF_\mu(J_T)\}$, 
where
\begin{equation*}
\bE_{2,\mu}^k(J_T)   := W^1_{p,\mu}(J_T;W^k_p( \Omega))\cap  L_{p,\mu}(J_T; W^{k+2}_p(\Omega)) , \quad k=0,1.
\end{equation*}
Then the Nemyskii operator induced by $\varphi$, still denoted by $\varphi$, satisfies  
$$
\varphi\in C^1(X).
$$
Moreover, given $T_*>0$
\begin{equation}
\label{bdd boundary space}
\| \varphi (u) \|_{\bF_\mu(J_T)} \leq C  \left( \|\varphi'(u)\|_\infty \| u\|_{\bF_\mu(J_T)} + \| \varphi(u)\|_\infty  \right) ,\quad u\in  \bF_\mu(J_T).
\end{equation}
The constant $C>0$ is uniform with respect to $T\in (0,T_*]$.
\end{lemma}
\begin{proof}
The mapping property $\varphi\in C^1(\bE_{2,\mu}^k(J_T))$  can be proved via direct computations   and the fact that $\varphi \in C^5(\bR)$.
It follows from  Lemma~\ref{lem:B-mu}(b)  that there exists a bounded right inverse $\gamma^c_0$ for the initial trace operator
$$
\gamma_0: \bE_{2,\mu}^0(J_T) \to W^{2\mu-2/p}_p(\Omega).
$$
The $C^1$-continuity of $\varphi$ in $W^{2\mu-2/p}_p(\Omega)$ then follows from the relationship
$$
\varphi(u)= \gamma_0 \varphi(\gamma^c_0(u)), \quad u\in W^{2\mu-2/p}_p(\Omega).
$$ 
The case $X=W^{2\mu+1-2/p}_p(\Omega)$ follows from a similar argument. The assertion in 
\eqref{bdd boundary space} has been proved in \cite[Lemma~4.2.3(a)]{MeyriesThesis}.  A close look at its proof shows that the constant in~\cite[Lemma~4.2.3(a)]{MeyriesThesis} is uniform with respect to $T\in (0,T_*]$.
The $C^1$-continuity of $\varphi$ in $\bF_\mu (J_T) $ can be derived from  \eqref{bdd boundary space} by  a mean value theorem argument  and the fact that $\bF_\mu(J_T)$ is a Banach algebra.
 \end{proof}

Next, we will establish some relevant mapping properties of the operators in \eqref{nonlinear abstract equation}. 
For the analysis below,
note that by  Proposition~\ref{extension-zero} and  \cite[Theorems 4.2 and 4.5]{MeSc12},
there exists  a constant $C>0$ such that
\begin{equation}
\label{boundary embedding F1}
\| {\rm tr}_{\partial\Omega} v\|_{\bF_{1,\mu}(J_T)} \leq C \|v\|_{\bE_{2,\mu}^0(J_T) }, \quad v\in \bE_{2,\mu}^0(J_T),
\end{equation}
where   the embedding constant is independent of $T$ if $v\in \zbE_{2,\mu}^0(J_T)$.

Suppose that $\phi_1\in C^5(\bR^{16}  ) $, $\phi_2\in C^5(\bR^{48}  )$ and  $\phi_3\in C^5(\bR^{9}  )$.
In order to derive an estimate for 
$$\|\cA(z_1+z_2) - \cA(z_1) -\cA'(z_1)z_2\|_{\bE_{0,\mu}(J_T)}$$ 
for proper functions $z_1,z_2\in \bE_{1,\mu}(J_T)$, we will consider five types of mappings, given by
\begin{equation}
\label{Gi}
\begin{aligned}
G_1 (z) & =  \phi_1( z) \phi_2(\partial z), \\
G_2 (z) & = \phi_1 (z)   \partial_{ij} z, \\
G_3 (z)& = \phi_1 (z) \phi_3(\partial m), \\
G_4 (z)& = \phi_1 (z) \phi_3(\partial m) \partial_{ij} m,   \\
G_5 (z) & = \phi_1(z) |\Delta m |^2,
\end{aligned}
\end{equation}
where for any  function $z=(u,F,\theta,m)\in C^1(\Omega,\bR^{16})$, we define
$ \partial z\in C(\Omega, \bR^{48})$ by $\partial z=( \partial_1 z, \partial_2 z, \partial_3 z),$ 
and
$\partial m\in C(\Omega, \bR^9)$ by $\partial m= ( \partial_1 m, \partial_2 m, \partial_3 m) .$

All   terms  in $\cA(z)$ can be estimated by using one of the functions $G_i$. 
For instance, 
\begin{itemize}
\vspace{1mm}\item 
terms like $\mu'(\theta)\partial_i \theta \partial_i u$ can be estimated by using $G_1$
 with $\phi_1(z)=\mu'(\theta)$ and $\phi_2(\partial z)=\partial_i \theta \partial_i u$; 
\vspace{1mm}\item 
terms like $\mu(\theta) \partial_{ii} u$ can be estimated by using $G_2$ with $\phi_1(z)=\mu(\theta)$; 
\vspace{1mm}\item
the term $K(z):\nabla^2 \theta=K_{ij}(z)\partial_{ij}\theta $ can be estimated by using $G_2$ with $\phi_1(z)=  K_{ij}(z)$;
\vspace{1mm}\item
terms like $\alpha(\theta) |m|^2 |\nabla m|^2$ can be estimated by $G_3$ with $\phi_1(z)=\alpha(\theta) |m|^2 $ and $\phi_3(\partial m)=|\nabla m|^2$;
\vspace{1mm}\item
the scalar components of $(\alpha(\theta)  I_3  - \beta(\theta)M(m))\Delta m$
 can be estimated by using $G_4$ with  $\phi_3\equiv 1$ and $\phi_1$ properly chosen;
 \vspace{1mm}\item
 the term $\alpha(\theta) |\nabla m|^2 \;m\cdot \Delta m$, appearing in the $\theta$-equation,  can be estimated by using $G_4$ 
 with $\phi_3(\partial m)= |\nabla m|^2$ and $\phi_1$ properly chosen;
 \vspace{1mm}\item
 lastly, the term $\alpha (\theta) |\Delta m|^2$ can be estimated by using $G_5$ with $\phi_1(z)=\alpha(\theta)$.
\end{itemize}
\begin{lemma}
\label{lem:mapping property-linearization}
Let the functions $G_i, 1\le i\le 5$ be given by \eqref{Gi}. Then
$$G_1, G_2, G_5 \in C^1(\bE_{1,\mu} (J_T) , \bE_{0,\mu}^0 (J_T) ),\quad  G_3, G_4 \in C^1(\bE_{1,\mu} (J_T) , \bE_{0,\mu}^1 (J_T) ),$$
where $\bE_{0,\mu}^k (J_T)=L_{p,\mu}(J_T; W^k_p(\Omega))$, $k=0,1$.
Furthermore, given $T_0, R_0 >0$, then for any $T\in (0,T_0]$, $R\in (0,R_0]$ and 
any $z_1=(u_1,F_1,\theta_1,m_1) \in \bE_{1,\mu} (J_T)$ and $z_2=(u_2,F_2,\theta_2,m_2)\in \zbE_{1,\mu} (J_T)$ satisfying
\begin{align*}
 \|z_1\|_{\bB_\mu(J_T)},\; \|z_1\|_{\bE_{1,\mu} (J_T)},\;  \|z_2\|_{\bE_{1,\mu} (J_T)} \le  R,
\end{align*}
the following estimate holds
\begin{align}
\label{est Gi}
\begin{split}
\| G_i (z_1+z_2)  - G_i(z_1) - G_i'(z_1)z_2 \|_{\bE_{0,\mu}^0(J_T)} & \le \Phi(\|z_2\|_{\bE_{1,\mu} (J_T)} ) \|z_2\|_{\bE_{1,\mu} (J_T)} ,\quad i=1,2,5,\\
\| G_i (z_1+z_2)  - G_i(z_1) - G_i'(z_1)z_2 \|_{\bE_{0,\mu}^1(J_T)} & \le \Phi(\|z_2\|_{\bE_{1,\mu} (J_T)} ) \|z_2\|_{\bE_{1,\mu} (J_T)} ,\quad i=3,4,
\end{split}
\end{align} 
where $G_i'$ is the Frech\'et derivative of $G_i$.
\end{lemma}
\begin{proof}
The continuous differentiability of $G_i$ follows by direct computations.
We will  only establish the estimates in \eqref{est Gi}.
Easy computations lead to
\begin{align*}
& G_1 (z_1+z_2)  - G_1(z_1) - G_1'(z_1)z_2 \\
& = \left( \phi_1(z_1 +z_2)   - \phi_1(z_1  )  -   \phi_1'(z_1 )z_2)\right)\phi_2(\partial z_1 ) \\
& \quad   + \phi_1(z_1 +z_2) \left( \phi_2(\partial z_1 + \partial z_2)  - \phi_2(\partial z_1  ) -\phi_2'(\partial z_1)\partial z_2   \right) \\
 & \quad +  \left( \phi_1(z_1 +z_2)    - \phi_1(z_1  )    \right)  \phi_2'(\partial z_1)\partial z_2.  
\end{align*}
Then the mean value theorem,  Lemma~\ref{lem:B-mu}(a), \eqref{embedding-trace-space} imply
\begin{align*}
& \left\| \left( \phi_1(z_1 +z_2)   - \phi_1(z_1  )  -   \phi_1'(z_1 )z_2)\right)\phi_2(\partial z_1 )\right\|_{\bE_{0,\mu}^0(J_T)} \\
& \le  \|\phi_2(\partial z_1 ) \|_\infty \; \| z_2\|_\infty \; \int_0^1 \left\|  \phi_1'(z_1 +\sigma z_2) - \phi_1'(z_1) \right\|_{\bE_{0,\mu}^0(J_T)}\, d\sigma \\
& \le  \|\phi_2(\partial z_1 ) \|_\infty \; \| z_2\|_\infty \; \int_{[0,1]\times [0,1]} \left\|  \phi_1''(z_1 +\tau \sigma z_2)\right\|_\infty \; \left\| z_2  \right\|_{\bE_{0,\mu}^0(J_T)}\, d\sigma \, d\tau \\
& \le \Phi(\|z_2\|_{\bE_{1,\mu} (J_T)} ) \|z_2\|_{\bE_{1,\mu} (J_T)} .
\end{align*}
In the above, $\phi'(z)$ denotes the Frech\'et derivative of the Nemyskii operator induced by $\phi$ and  we have used the fact that $\phi'(z) =  \sum_{j=1}^{16}   \partial_{j}   \phi (z)  \otimes e_j $, where   $\partial_{j} \phi $  is the partial derivative of $\phi$.  
We will take advantage of this observation in the sequel.
Note that the function $\Phi$ above is uniform with respect to $T\in (0,T_0]$
in view of Lemma~\ref{lem:B-mu}(a). Estimating in the same way, we have
\begin{align*}
  &\| \phi_1(z_1 +z_2) \left( \phi_2(\partial z_1 + \partial z_2)  - \phi_2(\partial z_1  ) -\phi_2'(\partial z_1)\partial z_2   \right) \|_{\bE_{0,\mu}^0(J_T)}\\
   &\quad \le \Phi(\|z_2\|_{\bE_{1,\mu} (J_T)} ) \|z_2\|_{\bE_{1,\mu} (J_T)} .
  \end{align*}
The remaining terms can be estimated again by using the mean value theorem   as follows
\begin{align*}
   & \|\left( \phi_1(z_1 +z_2)    - \phi_1(z_1  )    \right)  \phi_2'(\partial z_1)\partial z_2  \|_{\bE_{0,\mu}^0(J_T)} \\
& \quad \le \|  \phi_2'(\partial z_1)\|_\infty \;  \|\partial z_2\|_\infty   \; \int_0^1  \left(\| \phi_1'(z_1 + \sigma z_2)\|_\infty  \; \| z_2 \|_{\bE_{0,\mu}^0(J_T)} \right)\, d\sigma \\
& \quad \le \Phi(\|z_2\|_{\bE_{1,\mu} (J_T)} ) \|z_2\|_{\bE_{1,\mu} (J_T)} .
\end{align*}
The estimate for $G_3$ and $G_5$ can be obtained in the same manner in view of the additional regularity of $m$.

The estimate for $G_2$ will be slightly different in the sense that we need to evaluate $\partial_{ij} z_k$, $k=1,2$, by using the $\bE_{0,\mu}^0(J_T)$-norm.
First, notice that
\begin{align*}
& G_2 (z_1+z_2)  - G_2(z_1) - G_2'(z_1)z_2, \\
& = \left( \phi_1(z_1 +z_2)    - \phi_1(z_1  )   -    \phi_1'(z_1 )z_2
  \right)  \partial_{ij} z_1    + \left( \phi_1(z_1 +z_2)   -  \phi_1(z_1  )    \right) \partial_{ij} z_2 .
\end{align*}
Then
\begin{align*}
&\| \left( \phi_1(z_1 +z_2)    - \phi_1(z_1  )   -    \phi_1'(z_1 )z_2
  \right)  \partial_{ij} z_1\|_{\bE_{0,\mu}^0(J_T)} \\
  & \le \|\partial_{ij} z_1\|_{\bE_{0,\mu}^0(J_T)} \; \|z_2\|_\infty \; \int_0^1 \left\|  \phi_1'(z_1 +\sigma z_2) - \phi_1'(z_1) \right\|_\infty \, d\sigma \\
  & \le  \Phi(\|z_2\|_{\bE_{1,\mu} (J_T)} ) \|z_2\|_{\bE_{1,\mu} (J_T)} .
\end{align*}
Similarly,
\begin{align*}
\| \left( \phi_1(z_1 +z_2)   -  \phi_1(z_1  )    \right) \partial_{ij} z_2\|_{ \bE_{0,\mu}^0(J_T) } \leq  \Phi(\|z_2\|_{\bE_{1,\mu} (J_T)} ) \|z_2\|_{\bE_{1,\mu} (J_T)} .
\end{align*}
The estimate for $G_4$  
can be derived in a similar way by utilizing the additional regularity of $m$ and the facts that
\begin{align*}
 \|G_6(z_1 + z_2) -G_6(z_1) -G_6'(z_1)z_2 \|_{C([0,T];C^1(\overline{\Omega}))} 
& \leq 
   \Phi(\| z_2\|_{C([0,T];C^1(\overline{\Omega}))}) \| z_2\|_{C([0,T];C^1(\overline{\Omega}))} ,  \\
 \|G_6(z_1 + z_2) -G_6(z_1)   \|_{C([0,T];C^1(\overline{\Omega}))}
 & \leq 
  M \| z_2\|_{C([0,T];C^1(\overline{\Omega}))}  ,
\end{align*}
where $G_6(z)=\phi_1(z) \phi_3(\partial m)$.
\end{proof}

We are now ready to establish the differentiability of $(\sA,\sB,\sF)$ as operators defined on $\bE_{1,\mu}(J_T)$.
\begin{proposition}\label{Prop: mapping properties}
Assume \eqref{assumption} and \eqref{indices cond}.  Then
\begin{equation*}
\begin{aligned}
\cA & \in C^1(\bE_{1,\mu}(J_T),  \bE_{0,\mu}(J_T)),                     \quad &&  \cA^\prime (z_*)z=  \sA(z_*)z + [\sA^\prime(z_*)z] z_*, \\
\sF & \in C^1(\bE_{1,\mu}(J_T),\bE_{0,\mu}(J_T))),                       \quad && \\
\cB & \in C^1(\bE_{1,\mu}(J_T),  \bF_\mu(J_T)) ,                           \quad &&  \cB^\prime (z_*)z= \sB(z_*)z + [\sB^\prime(z_*)z] z_*,
\end{aligned}
\end{equation*} 
for $z_*, z\in \bE_{1,\mu}(J_T)$,   where the mappings $(\cA, \cB)$ were introduced in~\eqref{AB}.  
Moreover, given $T_0, R_0 >0$, then   for any $T\in (0,T_0]$, $R\in (0,R_0]$ 
 and any $z_*\in \bE_{1,\mu}(J_T)$, $z\in \zbE_{1,\mu}(J_T)$  satisfying
\begin{align*}
  \| {\rm tr}_{\partial\Omega} z_*\|_{\bF_\mu(J_T)},\; \|{\rm tr}_{\partial\Omega} \nabla z_*\|_{\bF_\mu(J_T)}, \;  \|z_* \|_{ \bB_\mu(J_T) },\; \| z_*\|_{\bE_{1,\mu}(J_T)},\;   \|z\|_{\bE_{1,\mu}(J_T)}\leq R,
\end{align*}
the following estimates hold
\begin{align}\label{Frechet der est}
\begin{split}
\| \cA (z_* + z)   -   \cA(z_*)    -\cA'(z_*)z    \|_{\bE_{0,\mu}(J_T)} & \le  \Phi(\|z\|_{\bE_{1,\mu}(J_T)} ) \|z\|_{\bE_{1,\mu}(J_T)},  \\
\| \sF (z_* + z)  - \sF(z_*) - \sF'(z_*) z \|_{\bE_{0,\mu}(J_T)} & \le  \Phi(\|z\|_{\bE_{1,\mu}(J_T)} ) \|z\|_{\bE_{1,\mu}(J_T)}, \\
\| \cB (z_* + z)   - \cB(z_*)  -  \cB'(z_*)z    \|_{\bF_\mu(J_T)} & \le  \Phi(\|z\|_{\bE_{1,\mu}(J_T)} ) \|z\|_{\bE_{1,\mu}(J_T)}. 
\end{split}
\end{align}
If, in addition, $\bar{z}\in \bE_{1,\mu}(J_T)$  with $z_*(0)=\bar{z}(0)$ satisfies
\begin{align*}
 \| {\rm tr}_{\partial\Omega} \bar{z}\|_{\bF_\mu(J_T)}, \| {\rm tr}_{\partial\Omega} \nabla \bar{z}\|_{\bF_\mu(J_T)} ,  \| \bar{z}\|_{\bE_{1,\mu}(J_T)}, \|\bar{z}\|_{\bB_\mu(J_T)} \leq R,
\end{align*}
then the following estimates hold
\begin{align}\label{Frechet der est 2}
\begin{split}
\| \cA'(z_*) z - \cA'(\bar{z}) z \|_{\bE_{0,\mu}(J_T)} & \le  \Phi(\|z_* - \bar{z}\|_{\bE_{1,\mu}(J_T)} ) \|z\|_{\bE_{1,\mu}(J_T)},  \\
\| \sF'(z_*) z - \sF'(\bar{z}) z \|_{\bE_{0,\mu}(J_T)} & \le  \Phi(\|z_* - \bar{z}\|_{\bE_{1,\mu}(J_T)} ) \|z\|_{\bE_{1,\mu}(J_T)}, \\
\|  \cB'(z_*) z - \cB'(\bar{z}) z  \|_{\bF_\mu(J_T)} & \le  \Phi(\|z_*-\bar{z}\|_{\bE_{1,\mu}(J_T)} ) \|z\|_{\bE_{1,\mu}(J_T)}. 
\end{split}
\end{align} 
\end{proposition}
\begin{proof}
The continuous differentiability of $\cA$ and $\sF$ and the first two estimates in \eqref{Frechet der est} are immediate 
consequences of Lemma~\ref{lem:mapping property-linearization}. 
The continuous differentiability of $\cB$ is a direct consequence of Lemma~\ref{lem: Nemyskii} and the fact that $\bF_\mu(J_T) $ is a Banach algebra.
To establish the last estimate in \eqref{Frechet der est},   we set $z=(z_j)_{j=1}^{16} =(u ,F ,\theta ,m )$
and
$ z_*=(u_*,F_*,\theta_*,m_*)$.
Then we can apply  \eqref{bdd boundary space}, \eqref{boundary embedding F1}, Lemma~\ref{lem: multiplication}(i) and (ii), 
Proposition~\ref{extension-zero}, and \cite[Theorem  4.5]{MeSc12} to obtain (where we suppress ${\rm tr}_{\partial \Omega}$ in the following computations)
\begin{align*}
& \| \cB  (z_*+z)  - \cB (z_*) - \cB'(z_*)z \|_{\bF_\mu (J_T)}  \\
& \le \left\|   \left[ \left(\int_0^1 \left( K'(z_*+\sigma z ) - K'(z_*) \right) \, d\sigma \right) z \right] \nabla  \theta_*  \right\|_{\bF_\mu (J_T)} \\
& \le C \|\nabla \theta_* \|_{\bF_\mu (J_T)}  \left\|  \int_0^1 \left( K'(z_*+\sigma z ) - K'(z_*) \right) \, d\sigma  \right\|_{\bF_\mu (J_T)}   \|z\|_{\bF_{1,\mu}(J_T)} \\
& \quad + C \left\|  \int_0^1   K'(z_*+\sigma z )   \, d\sigma   \right\|_{\bF_\mu (J_T)} \left\|  \nabla \theta   \right\|_{\bF_\mu (J_T)} \|z\|_{\bF_{1,\mu}(J_T)}\\
& \le \Phi( \|z\|_{\bE_{1,\mu} (J_T)} ) \|z\|_{\bE_{1,\mu} (J_T)} .
\end{align*}
This establishes the last estimate in \eqref{Frechet der est}.

Concerning the estimates in \eqref{Frechet der est 2}, we will only establish the last one. The remaining two follow from a similar argument.  
\begin{align*}
 \|  \cB'(z_*) z - \cB'(\bar{z}) z  \|_{\bF_\mu(J_T)} 
& = \|  \sB (z_*  )z  -   \sB(\bar{z}) z  -  [\sB'(z_*)z] z_*  + [\sB'(\bar{z})z] \bar{z}  \|_{\bF_\mu(J_T)}  \\
& \le \|  \sB (z_*  )z  -   \sB(\bar{z}) z   \|_{\bF_\mu(J_T)} + \| [\sB'(z_*)z] z_*  - [\sB'(\bar{z})z] \bar{z}  \|_{\bF_\mu(J_T)}.
\end{align*}
Let $\bar{z}=(\bar{u},\bar{F},\bar{\theta},\bar{m})$.
Then the first term on the right hand side can be estimated by using \eqref{bdd boundary space}, \eqref{boundary embedding F1}, Lemma~\ref{lem: multiplication}(i) and (ii), 
Proposition~\ref{extension-zero}, and \cite[Theorems 4.2 and  4.5]{MeSc12}    as follows
\goodbreak
\begin{align*}
 \|  \sB (z_*  )z  -   \sB(\bar{z}) z   \|_{\bF_\mu(J_T)}  
& \le C \| K(z_*)- K(\bar{z} ) \|_{\bF_\mu(J_T) } \|\nabla \theta \|_{\bF_\mu(J_T)}  \\
& \le C \int_0^1 \|   K'( \sigma z_* + (1-\sigma) \bar{z} ) \|_{\bF_\mu(J_T)}\,  d\sigma \|z_*-\bar{z}\|_{\bF_{1,\mu}(J_T)} \|z\|_{\bE_{1,\mu}(J_T)}\\
& \le  \Phi(\|z_*-\bar{z}\|_{\bE_{1,\mu}(J_T)} ) \|z\|_{\bE_{1,\mu}(J_T)}.
\end{align*}
The estimate of the second term  can be obtained analogously:
\begin{align*}
& \| [\sB'(z_*)z] z_* - [\sB'(\bar{z})z] \bar{z}  \|_{\bF_\mu(J_T)} \\
& \le \| [  K '(z_*) z] \nabla (\theta_* -\bar{\theta})  \|_{\bF_\mu(J_T)} + \| [ (K '(z_*) -   K '(\bar{z}) ) z ]\nabla  \bar{\theta}  \|_{\bF_\mu(J_T)} \\
& \le C \left( \|   K' (z_*)  z\|_{\bF_\mu(J_T)}   \|\nabla (\theta_* -\bar{\theta})   \|_{\bF_\mu(J_T)} +  
\sum_{j=1}^{16} \|  \left (  \partial_{j}  K (z_*) -    \partial_{j}  K (\bar{z})  \right)   \|_{\bF_\mu(J_T)} \| z_j \nabla  \bar{\theta}  \|_{\bF_\mu(J_T)}   \right)   \\
& \le C \left( \|   K' (z_*)\|_{\bF_\mu(J_T)} + \Phi(R) \| \nabla \bar{\theta}\|_{\bF_\mu(J_T)}  \right)\|  z\|_{\bF_{1,\mu}(J_T)}  \|z_*-\bar{z}\|_{\bE_{1,\mu}(J_T)}\\
& \le \Phi(\|z_*-\bar{z}\|_{\bE_{1,\mu}(J_T)} ) \|z\|_{\bE_{1,\mu}(J_T)}.
\end{align*}
\end{proof}

To study the continuous dependence of solutions to \eqref{nonlinear abstract equation} on the initial data, 
see Theorem~\ref{Thm:wellposed abstract}(b), we need the following result. 

\begin{lemma}\label{lem:right inverse}
Let $\cB$ be as in \eqref{AB}. Then we have
\begin{enumerate}
\item[{\rm (a)}]
$\cB\in C^1(X_{\gamma, \mu}, Y_{\gamma,\mu})\quad\text{and} \quad \cB^\prime(z_0)z = \sB(z_0)z + [\sB^\prime (z_0)z]z_0,\quad z_0,z\in X_{\gamma,\mu}.$
\vspace{1mm}
\item[{\rm (b)}]
For each $z_0\in X_{\gamma,\mu}$, $\cB^\prime(z_0)\in \cL( X_{\gamma,\mu} , Y_{\gamma,\mu} )$ 
has a bounded right inverse
$ \cR(z_0)\in \cL( Y_{\gamma,\mu} , X_{\gamma,\mu}) $.
\end{enumerate}
\end{lemma}
\begin{proof}
(a)
By  Lemma~\ref{lem:B-mu}, the trace operator $\gamma_0\in\cL(\bE_{1,\mu}(J_T),X_{\gamma,\mu} )$
has a right inverse 
$\gamma_0^c \in\cL(X_{\gamma,\mu}, \bE_{1,\mu}(J_T) )$.
It is then easy to see that $\cB(z)= \tilde{\gamma_0}\cB(\gamma^c_0 (z))$, 
where $\tilde{\gamma_0}$ denotes the initial time trace operator for functions defined on $\bF_\mu(J_T)$.
The assertions  then follow from Proposition~\ref{Prop: mapping properties}.

\medskip\noindent
(b)
The existence of $\cR(z_0)$ is proved in  \cite[Proposition 2.5.1]{MeyriesThesis}.
\end{proof}

}
%

\end{document}